\documentclass[twoside, 11pt]{article}

\usepackage{geometry}
\usepackage{latexsym}
\usepackage{enumerate,verbatim}
\usepackage{amsfonts}
\usepackage{enumitem}
\usepackage[english]{babel}
\usepackage[ansinew]{inputenc}
\usepackage{amsmath,amsthm,amsfonts,amssymb}
\usepackage{graphics}
\graphicspath{{figs/}}
\usepackage{subcaption}
\usepackage{mathrsfs}
\usepackage{booktabs}
\usepackage[ruled,linesnumbered]{algorithm2e}

\usepackage{color}

\usepackage{pifont}

\usepackage{tikz}
\usetikzlibrary{fit}
\tikzset{%
	highlight/.style={rectangle,blend mode = multiply,draw=blue!90!black,thick,rounded corners = 0.3 mm,inner sep=0.5pt}
}

\usetikzlibrary{shadows}
\usetikzlibrary{arrows,positioning} 

\usepackage{color}
\usepackage{bm}
\usepackage{url}
\usepackage{hyperref}
\hypersetup{
  colorlinks=true,
  linkcolor=blue,
  filecolor=blue,
  anchorcolor=blue,
  urlcolor=black,
  citecolor=purple
}

\usepackage{appendix}

\usepackage{diagbox}

 \newtheorem{thm}{Theorem}[section]
 
 \newtheorem{lemma}[thm]{Lemma}
 \newtheorem{prop}[thm]{Proposition}
 \theoremstyle{definition}
 \newtheorem{defi}[thm]{Definition}
 \theoremstyle{remark}
 \newtheorem{rmk}[thm]{Remark}
 \theoremstyle{eg}

  \newtheorem{ass}[thm]{Assumption}
 \theoremstyle{fact}
 \newtheorem{fact}[thm]{Fact}
\numberwithin{equation}{section}

\newcommand{\bz}{\mathbf 0}

\newcommand{\R}{\mathbb R}

\DeclareMathOperator{\prox}{prox}
\DeclareMathOperator{\proj}{proj}
\DeclareMathOperator{\sign}{sign}

\DeclareMathOperator{\dist}{dist}

\DeclareMathOperator*{\argmin}{argmin}
\DeclareMathOperator*{\argmax}{argmax}
\DeclareMathOperator{\tr}{tr}

\newcommand{\grad}{\operatorname{grad}}

\newcommand{\St}{\operatorname{St}}

\newcommand{\bX}{\bm{X}}
\newcommand{\bI}{\bm{I}}
\newcommand{\bY}{\bm{Y}}
\newcommand{\bZ}{\bm{Z}}

\definecolor{cobalt}{rgb}{0.0, 0.28, 0.67}
\definecolor{slateblue}{HTML}{6a5acd}
\definecolor{darkgreen}{HTML}{005f5f}

\usepackage[square,numbers]{natbib}

\usepackage{bigstrut}

\title{\bf Primal-Dual Methods for Nonsmooth Nonconvex Optimization with Orthogonality Constraints}
\author{Linglingzhi Zhu\thanks{H. Milton Stewart School of Industrial and Systems Engineering, Georgia Institute of Technology, Atlanta, GA, USA (\href{mailto:llzzhu@gatech.edu}{llzzhu@gatech.edu})} \and Wentao Ding\thanks{Department of Systems Engineering and Engineering Management, The Chinese University of Hong Kong, Shatin, N.T., Hong Kong (\href{mailto:wentaoding@cuhk.edu.hk}{wentaoding@cuhk.edu.hk})}
\and Shangyuan Liu\thanks{Department of Systems Engineering and Engineering Management, The Chinese University of Hong Kong, Shatin, N.T., Hong Kong (\href{mailto:shangyuanliu@link.cuhk.edu.hk}{shangyuanliu@link.cuhk.edu.hk})} \and Anthony Man-Cho So\thanks{Department of Systems Engineering and Engineering Management, The Chinese University of Hong Kong, Shatin, N.T., Hong Kong (\href{mailto:manchoso@se.cuhk.edu.hk}{manchoso@se.cuhk.edu.hk})}}
\date{\today}

\begin{document}
\maketitle

\begin{abstract}
Recent advancements in data science have significantly elevated the importance of orthogonally constrained optimization problems. The Riemannian approach has become a popular technique for addressing these problems due to the advantageous computational and analytical properties of the Stiefel manifold. Nonetheless, the interplay of nonsmoothness alongside orthogonality constraints introduces substantial challenges to current Riemannian methods, including scalability, parallelizability, complicated subproblems, and cumulative numerical errors that threaten feasibility.
In this paper, we take a retraction-free primal-dual approach and propose a linearized smoothing augmented Lagrangian method specifically designed for nonsmooth and nonconvex optimization with orthogonality constraints. Our proposed method is single-loop and free of subproblem solving. We establish its iteration complexity of $\mathcal{O}(\epsilon^{-3})$ for finding $\epsilon$-KKT points, matching the best-known results in the Riemannian optimization literature. Additionally, by invoking the standard Kurdyka-\L ojasiewicz (K\L) property, we demonstrate asymptotic sequential convergence of the proposed algorithm. Numerical experiments on both smooth and nonsmooth orthogonal constrained problems demonstrate the superior computational efficiency and scalability of the proposed method compared with state-of-the-art algorithms.
\end{abstract}

\section{Introduction}

In this paper, we focus on the following general nonsmooth nonconvex problem with orthogonality constraints:
\begin{equation}\tag{P}\label{eq:problem_composite}
\begin{array}{cl}
\min\limits_{\bX\in\mathbb{R}^{m\times n}}  &f(\bX):=\ell(\bX)+g(\bX)\\
\text{s.t.} &\bX^\top \bX=\bI_n,
\end{array}
\end{equation}
where $m\ge n$, $\ell:\R^{m\times n}\rightarrow \R$  is continuously differentiable and $g:\mathbb{R}^{m\times n}\rightarrow \mathbb{R}$ is a proximal friendly weakly convex function. 
Extending beyond classical quadratic programming with quadratic constraints, 
these problems constitute a fundamental class of nonconvex optimization challenges with broad applications across scientific and engineering domains, including principal component analysis \cite{jolliffe2016principal, journee2010generalized, wang2023linear}, low-rank matrix completion \cite{keshavan10matrix, vandereycken2013low}, group synchronization \cite{ling2022near, liu2023unified, zhu2023rotation}, dictionary learning \cite{sun2016completeI, cherian2016riemannian}, and deep learning \cite{saxe2014exact,arjovsky2016unitary,kingma2018glow}, among others. 

General orthogonally constrained optimization problems are nonconvex, posing significant challenges for traditional nonlinear programming methods. Over the past decades, Riemannian optimization \cite{absil2009optimization, boumal2023introduction} has emerged as a powerful approach for tackling these problems by leveraging their manifold structures, facilitated by well-defined and computable exponential maps. This approach transforms the original constrained problem into an intrinsic unconstrained formulation, shifting the focus to additional constraints beyond orthogonality. However, the interplay of nonsmoothness (introduced by $g$) and orthogonality constraints severely complicates the Riemannian paradigm. To solve the problem \eqref{eq:problem_composite}, \cite{chen2020proximal} introduced the ManPG algorithm, which applies a proximal gradient method on the tangent space and employs a carefully designed step size to control local errors arising from the geometric structure. With a modified nonconvex subproblem, Huang and Wei \cite{huang2022riemannian} extended ManPG and established its iterative convergence using the Riemannian Kurdyka-\L ojasiewicz inequality. However, these approaches require a strongly convex subproblem to be solved iteratively using the semi-smooth Newton method, resulting in an overall computational complexity of $\mathcal{O}(n^4)$, which may pose scalability challenges as the problem size increases. 

To overcome this issue, another line of research building on Riemannian primal-dual methods has been proposed, which introduces ancillary variables to separate the nonsmooth objective and the orthogonality constraints. Initiated by \cite{kovnatsky2016madmm}, this approach has inspired numerous theoretical studies, leading to the development of various Riemannian Lagrangian-based algorithms for solving \eqref{eq:problem_composite}. In particular, \cite{deng2023manifold, zhou2023semismooth, deng2025oracle, xu2025oracle} investigate the Riemannian augmented Lagrangian method and establish asymptotic and non-asymptotic convergence results, incorporating additional linear/nonlinear composite terms within the nonsmooth function $g$. Moreover, leveraging a different splitting scheme, \cite{li2025riemannian} proposed a single-loop Riemannian alternating direction method of multipliers (RADMM) with an iterative complexity guarantee.
 
Nevertheless, the aforementioned Riemannian-based methods still face several challenges in solving \eqref{eq:problem_composite}. First, with very few recent exceptions, the majority of these algorithms typically rely on a double-loop framework, requiring the solution to a subproblem in each iteration. Second, as a common issue highlighted by \cite{ablin2022fast}, these approaches often suffer from geodesic-based retraction difficulties: accumulated errors over iterations can compromise feasibility. Moreover, most of retraction operations (including geodesic and projection-based variants) frequently involve expensive linear algebra computations, such as matrix inversion and exponentiation, which become increasingly prohibitive as the matrix dimension grows and are challenging to parallelize on modern hardware.

In this paper, we propose a single-loop retraction-free approach for solving nonsmooth optimization problems with orthogonality constraints \eqref{eq:problem_composite} from a primal-dual perspective. Instead of adopting a Riemannian approach, we introduce dual variables to handle the nonconvex orthogonality constraints. Specifically, we develop a linearized smoothing augmented Lagrangian method (LSALM), which incorporates a smoothing technique for the nonsmooth objective function within the standard augmented Lagrangian framework. Our proposed LSALM involves no subproblems and all steps of the algorithm are explicit, requiring neither expensive matrix inversion nor exponential function computation.

A fundamental challenge in analyzing primal-dual algorithms for the nonsmooth nonconvex problem \eqref{eq:problem_composite} lies in delicately balancing the primal and dual variables to ensure the eventual feasibility of the iterates. To overcome this theoretical bottleneck, we identify a locally uniform constraint qualification (UCQ) condition for the orthogonality constraints and design an appropriate iterative scheme to ensure the generated sequence remains within this region. By exploiting the benign landscape within this locally UCQ region, we prove a crucial geometric property that any stationary point of the quadratic penalty function for the nonconvex orthogonality constraints is also a global minimizer. This result helps ensure the feasibility of the updates, and thereby guarantees global convergence.

Building upon these theoretical observations, we establish the first global complexity result for retraction-free algorithms applied to nonsmooth problems with orthogonality constraints. Among existing retraction-free algorithms, the penalty-based dissolving approach \cite{hu2024constraint} handles nonsmoothness but lacks iteration complexity guarantees, while the landing field approaches \cite{ablin2022fast,ablin2024infeasible} and the related primal-dual algorithm in \cite{gao2019parallelizable} rely on smoothness assumptions. On the other hand, methods such as SOC \cite{lai2014splitting} and PAMAL \cite{chen2016augmented} address nonsmooth formulations and avoid Riemannian retractions such as exponential maps, but still require projection onto the Stiefel manifold via SVD in their subproblems. In our work, we show that LSALM achieves a global iteration complexity of $\mathcal{O}(\epsilon^{-3})$ for finding $\epsilon$-KKT points, which matches the best-known results including Riemannian algorithms \cite{beck2023dynamic,peng2023riemannian,deng2025oracle,xu2025oracle}.
Additionally, we establish the asymptotic sequential convergence of LSALM under the standard Kurdyka-\L ojasiewicz (K\L) property, and the limiting point can be proved to be an $ \mathcal{O}(\epsilon)$-KKT point. The theoretical advantages of our approach compared to existing algorithms are summarized in the following table.

\begin{table}[htbp]\scriptsize
  \centering
\begin{tabular}{cccccc}
  \toprule
& Nonsmooth  & Single-loop & Complexity  & Sequential Convergence & Retraction-Free \\ \midrule
 SOC \cite{lai2014splitting}/PAMAL \cite{chen2016augmented}                         & \checkmark &             &          ---                    &                        &                 \\
  PCAL \cite{gao2019parallelizable}/Landing \cite{ablin2022fast} &            & \checkmark  & $\mathcal{O}(\epsilon^{-2})$ &                        & \checkmark      \\
  ManPG \cite{chen2020proximal}                       & \checkmark &             & $\mathcal{O}(\epsilon^{-2})^*$ &            &                 \\
  RPG \cite{huang2022riemannian}                       & \checkmark &             & $\mathcal{O}(\epsilon^{-2})^*$ &    \checkmark      &                 \\
   ManAL \cite{deng2025oracle,xu2025oracle}                           & \checkmark &   & $\mathcal{O}(\epsilon^{-3})$ &                        &       \\
    RADMM \cite{li2025riemannian}                                                     & \checkmark & \checkmark  & $\mathcal{O}(\epsilon^{-4})$ &                        &                 \\
  Smoothing RGD \cite{beck2023dynamic,peng2023riemannian}                           & \checkmark & \checkmark  & $\mathcal{O}(\epsilon^{-3})$ &                        &       \\ 
  LSALM (ours)                                                                     & \checkmark & \checkmark  & $\mathcal{O}(\epsilon^{-3})$ & \checkmark             & \checkmark      \\
  \bottomrule
\end{tabular}
\\\vspace{1mm}{\raggedleft\footnotesize{$^*$ Subproblem solver required due to lack of explicit solution.}\par}
\caption{Comparison of algorithms for nonconvex structured problems with orthogonality constraints}
\end{table}

Numerical experiments demonstrate that the proposed LSALM exhibits robust performance with respect to parameter choices, aligning well with the theoretical convergence guarantees and converging reliably in practice. On the nonsmooth sparse PCA task, LSALM consistently achieves the significantly lower CPU time and per-iteration cost compared to several popular algorithms. For instance, on a problem of size $(m,n)=(800,400)$, LSALM completes in 41.0s, outperforming ManPG-Ada (763.2s), SOC (231.1s), and RADMM (106.0s), while attaining comparable objective values and sparsity levels. Moreover, LSALM demonstrates superior scalability, benefiting from a more favorable dependence on problem dimension by its fast convergence and a low per-iteration computational cost involving only matrix multiplications. Beyond nonsmooth problems, LSALM also performs competitively on a smooth graph matching benchmark, matching or improving upon the objective and F-measure of state-of-the-art baselines while being comparable in CPU time. These results highlight LSALM as a versatile and efficient algorithm across both nonsmooth and smooth orthogonal constrained optimization tasks.

\subsection{Notation}
Throughout the paper, we use the standard notation. Let the Euclidean space of all $m\times n$ real matrices $\mathbb{R}^{m\times n}$ be equipped with inner product $\langle \bm{X},\bm{Y}\rangle:=\operatorname{tr}(\bm{X}^\top \bm{Y})$ for any $\bm{X},\bm{Y}\in\mathbb{R}^{m\times n}$. Its induced Frobenius norm is denoted by $\Vert\cdot\Vert_F$, and the operator norm is denoted by $\| \cdot \|$.
% , and the infinity norm is defined as $\| \bm{X} \|_{\infty}:=\max_{i,j}|\bm{X}_{ij}|$. 
Let $\mathbb{S}^{n}$ denote the set of real symmetric $n\times n$ matrices. Let $\mathcal{M}\subseteq\R^{m\times n}$ be an embedded smooth manifold and the tangent (resp. normal) space at $\bm{X}\in\mathcal{M}$ is denoted by $\operatorname{T}_{\bm{X}} \mathcal{M}$ (resp. $\operatorname{N}_{\bm{X}} \mathcal{M}$). We consider the Riemannian metric $\langle\cdot,\cdot\rangle_{\bm{X}}$ on $\mathcal{M}$ that is induced from the Euclidean inner product, i.e, at $\bm{X}\in \mathcal{M}$ we have $\langle\bm{\xi}, \bm{\eta}\rangle_{\bm{X}}:=\langle \bm{\xi}, \bm{\eta}\rangle$ for any $\bm{\xi}, \bm{\eta} \in \operatorname{T}_{\bm{X}} \mathcal{M}$. For simplicity we denote $G(\bX):=\bX^\top \bX-\bI_n$, and denote the Stiefel manifold by $\St(m,n):=\{\bm{X}\in\mathbb{R}^{m\times n}: G(\bX)=\bz\}$. For any set $\mathcal{X} \subseteq \mathbb{R}^{m\times n}$, we use $\iota_{\mathcal{X}}: \mathbb{R}^{m\times n} \rightarrow \{0, +\infty\}$ to denote the indicator function associated with $\mathcal{X}$ and $\operatorname{proj}_{\mathcal{X}}$ to denote the orthogonal projection onto $\mathcal{X}$. The proximal mapping of a proper lower-semicontinuous function $g:\R^{m\times n}\rightarrow \R \cup \{+\infty\}$  at the point $\bX \in \mathbb{R}^{m\times n}$ is defined by $\prox_{g}(\bX) := \mathop{\argmin}_{\bZ\in \R^{m\times n}} \{ g(\bZ) + \frac{1}{2} \|\bZ-\bX\|_F^2\}$. 

\subsection{Organization}
    The rest of the paper is organized as follows. Section \ref{sec:prelim} introduces the key definitions and preliminary results linking the nonlinear programming and Riemannian optimization. In Section \ref{sec:main}, we propose our primal-dual algorithm to solve the orthogonal constrained problem \eqref{eq:problem_composite}. The global convergence rate results are established in Section \ref{sec:global} with asymptotic iterative convergence guarantees in Section \ref{sec:sequential}. Section~\ref{sec:numerical} presents numerical results, including a study of algorithmic parameters and a performance comparison with related algorithms on specific nonsmooth and smooth problems. Finally, we end with some closing remarks in Section \ref{sec:conclusion}. Some standard definitions and auxiliary lemmas are provided in the appendix.

\section{Preliminaries}
\label{sec:prelim}

To begin with, we briefly characterize the first-order optimality conditions and identify suitable stationarity measures to evaluate the convergence behavior from the primal-dual perspective. We highlight its connection to the corresponding notions in the Riemannian setting, enabling the comparison between the Riemannian and our approaches from the nonlinear programming.

Recall the problem \eqref{eq:problem_composite}. Let $\rho\ge0$ and denote the augmented Lagrangian function as
\begin{equation*}
\mathcal{L}_{\rho}(\bX,\bY):=f(\bX)+\langle \bY,\bX^\top\bX-\bI_n\rangle+\frac{\rho}{2}\|\bX^\top \bX-\bI_n\|_F^2.
\end{equation*}
In the remainder of the paper, we consider the weakly convex objective function $f$ defined as follows.
\begin{defi}
The function $f: \R^{m\times n} \rightarrow \R$ is $\mu$-weakly convex on $\mathcal{X}\subseteq \R^{m\times n}$ if for any $\bX, \bX' \in \mathcal{X}$ and $\tau \in[0,1]$,
$$
f(\tau \bX+(1-\tau) \bX') \leq \tau f(\bX)+(1-\tau) f(\bX')+\frac{\mu \tau(1-\tau)}{2}\|\bX-\bX'\|_F^{2}.
$$
When $f$ is locally Lipschitz, it is equivalent to saying that $f+\frac{\mu}{2}\|\cdot\|_F^{2}$ is convex on $\mathcal{X}$.
\end{defi}

Since weakly convex functions are subdifferentially regular, we can utilize the Fr\'{e}chet subdifferential in the above definitions without confusion with other notions such as the limiting or Clarke subdifferentials. A brief overview of various subdifferential constructions in nonsmooth nonconvex optimization can be found in \citep{li2020understanding}. Then we have the following definition of the KKT points as in standard nonlinear programming.

\begin{defi}[KKT point]\label{defi:KKT}
The pair $(\bX, \bY)\in\mathbb{R}^{m\times n}\times \mathbb{S}^{n}$ is a KKT point of \eqref{eq:problem_composite} if 
$$
\bz\in\partial f(\bX)+  2\bX\bY,\quad \bX^\top \bX=\bI_n.
$$
Moreover, it is an $\epsilon$-KKT point if 
\[
\dist(-2\bX\bY,\partial f(\bX))\leq\epsilon,\quad \|\bX^\top \bX- \bI_n\|_F\leq\epsilon.
\]
\end{defi}

On the other hand, we introduce the following concept of a stationary point, which is widely used in Riemannian optimization; see also Appendix~\ref{sec:Riesubdiff}.

\begin{defi}[Stationary point]\label{def:stationary}
The point $\bX\in \St(m,n)$ is called a stationary point if
$$
\bz\in\operatorname{proj}_{\operatorname{T}_{\bX} \St(m,n)}(\partial f(\bX))
$$
and it is an $\epsilon$-stationary point if 
\[
\dist(\bz,\operatorname{proj}_{\operatorname{T}_{\bX} \St(m,n)}(\partial f(\bX)))\leq\epsilon. 
\]
\end{defi}

Now, we establish the following relationship between $\epsilon$-KKT and $\epsilon$-stationary points.
\begin{lemma}
The following implications hold: 
\begin{enumerate}[label=\normalfont(\alph*)]
    \item  If $(\bX, \bY)$ is an $\epsilon$-KKT point, then $\bX$ is an $(1+2\|\bX\|\|\bY\|)\epsilon$-stationary point;
    \item If $\bX$ is an $\epsilon$-stationary point, then $\bX$ is the $\bX$-coordinate of an $\epsilon$-KKT point.
\end{enumerate}

\end{lemma}
\begin{proof}
First, suppose that the pair $(\bX, \bY)\in\mathbb{R}^{m\times n}\times \mathbb{S}^{n}$ is an $\epsilon$-KKT point. Then, 
there exists $\bm{\zeta}\in\mathbb{R}^{m\times n}$ such that
\[
\bm{\zeta}\in\partial f(\bX)+  2\bX\bY,\quad \|\bm{\zeta}\|_F\leq \epsilon,\quad \|\bX^\top \bX-\bI_n\|_F\leq \epsilon.
\]
Let $\bZ:=2\bX \bY$. Then we know that
\[
\bm{\zeta}=-2\bX\bY+\bm{\zeta}+2\bX\bY\in\partial f(\bX)+\frac{1}{2}\bX(\bX^\top \bZ+\bZ^\top \bX)+\bX( (\bI_n-\bX^\top\bX) \bY+ \bY^\top (\bI_n-\bX^\top\bX)^\top).
\]
This implies that
\begin{align}\label{eq:kkttostat}
\dist\left(\bz,\partial f(\bX)+\frac{1}{2}\bX(\bX^\top \bZ+\bZ^\top \bX)\right)
&\leq\|\bm{\zeta}-\bX( (\bI_n-\bX^\top\bX) \bY+ \bY^\top (\bI_n-\bX^\top\bX)^\top)\|_F\notag\\
&\leq\epsilon+2\|\bX\|\|\bY\|\epsilon.
\end{align}
On the other hand, since
$
\frac{1}{2}\bX(\bX^\top \bZ+\bZ^\top \bX)\in \operatorname{N}_{\bX} \St(m,n)
$,
we know from \eqref{eq:kkttostat} that 
\[
\dist(\bz,\operatorname{proj}_{\operatorname{T}_{\bX} \St(m,n)}(\partial f(\bX)))=\dist(\bz,\partial f(\bX)+\operatorname{N}_{\bX} \St(m,n))
\leq (1+2\|\bX\|\|\bY\|)\epsilon.
\]

Conversely, suppose that $\bX\in \St(m,n)$ is an $\epsilon$-stationary point. Then there exist $\bm{\xi}\in\mathbb{R}^{m\times n}$ and $\bZ\in\mathbb{R}^{m\times n}$ such that
\[
\bm{\xi}\in \partial f(\bX)+\frac{1}{2}\bX(\bX^\top \bZ+\bZ^\top \bX)\subseteq\partial f(\bX)+\operatorname{N}_{\bX}\St(m,n) \quad \text{and}\quad \|\bm{\xi}\|_F\leq\epsilon.
\]
By setting $\bY=\frac{1}{4}(\bX^\top \bZ+\bZ^\top\bX)$ we have
\[
\begin{aligned}
\bm{\xi}=\bm{\xi}-\frac{1}{2}\bX(\bX^\top \bZ+\bZ^\top \bX)+  2\bX\bY\in\partial f(\bX)+  2\bX\bY.
\end{aligned}
\]
Thus, $(\bX,\bY)$ is a pair of $\epsilon$-KKT point. The proof is complete.
\end{proof}

\section{Linearized Smoothing Augmented Lagrangian Method}\label{sec:main}
In this section, we address the first-order algorithmic design of the optimization problem~\eqref{eq:problem_composite} from a primal-dual perspective. To handle the nonsmoothness inherent in the objective, a natural starting point is to apply a Moreau envelop smoothing technique to the standard augmented Lagrangian function. This leads to the following iterative gradient-based scheme, built upon the smoothing augmented Lagrangian method:
\begin{equation}\label{eq:smoothingscheme}
\left\{
\begin{aligned}
\bZ^{k+1} &:= \bZ^k + \beta (\bX^{k+1} - \bZ^k),
\ \text{where}\
\bX^{k+1} := \argmin_{\bX\in\mathcal{X}}\,\left\{\mathcal{L}_{\rho}(\bX,\bY^k)+\frac{r}{2}\|\bX-\bZ^k\|_F^2\right\},\\[6pt]
\bY^{k+1} &:= \proj_{\mathcal{Y}}(\bY^k+\alpha ((\bX^{k+1})^\top\bX^{k+1}-\bI_n))
\end{aligned}
\right.
\end{equation}
Here, $\mathcal{X} \subset \R^{m \times n}$ is introduced as a compact ambient domain constraint, and $\mathcal{Y}$ denotes the domain of the dual variables. Note that since the feasible set induced by the orthogonality constraint $\bX^\top\bX=\bI_n$ is naturally bounded, explicitly restricting the primal updates within a sufficiently large compact set $\mathcal{X}$ does not alter the optimal solutions of the original problem, while safely preventing the sequence from diverging during the infeasible phase. Furthermore, $r > 0$ is the smoothing parameter. By choosing $r$ sufficiently large, we ensure that the surrogate function $\bX\mapsto\mathcal{L}_{\rho}(\bX,\bY)+\frac{r}{2}\|\bX-\bZ\|_F^2$ becomes strongly convex over $\mathcal{X}$. This strong convexity guarantees that the primal minimization step is well-defined and unique, which consequently ensures the continuous differentiability of the Moreau envelope:
\[
\bZ\mapsto\min_{\bX\in\mathcal{X}} \left\{\mathcal{L}_{\rho}(\bX,\bY)+\frac{r}{2}\|\bX-\bZ\|_F^2\right\}.
\] 
Under this condition, the overall update scheme \eqref{eq:smoothingscheme} can be interpreted as performing a gradient descent-ascent step (with stepsizes $\beta$ and $\alpha$, respectively) on the smoothed augmented Lagrangian function.

The well-definedness of the parameter $r$ is ensured by the following lemmas. We begin by verifying this through the weak convexity of the quadratic penalty function associated with the orthogonality constraint, which is a nontrivial result as the penalty function is quartic and typically is not weakly convex.

\begin{lemma}\label{lem:sc}
The function $\bX\mapsto \frac{1}{2}\|\bX^{\top}\bX-\bI_n\|_F^2$ is 2-weakly convex.
\end{lemma}

\begin{proof}

First, we denote $h(\bX):=\frac{1}{2}\|\bX^{\top}\bX-\bI_n\|_F^2$. Then the gradient of the function $h$ is
$$
\nabla h(\bX)=2\bX(\bX^{\top}\bX-\bI_n).
$$
Also, we have the Hessian at $\bX$ satisfies for any $\bm{H}_{\bX} \in \mathbb{R}^{m \times n}$ that
\begin{equation}\label{Hess-def}
\begin{aligned}
\nabla^2 h(\bX)\left[\bm{H}_{\bX}\right]
&=\operatorname{D} \nabla h(\bX)\left[\bm{H}_{\bX}\right]\\
&= \lim_{t\rightarrow 0}\frac{2(\bX+t\bm{H}_{\bX})((\bX+t\bm{H}_{\bX})^{\top}(\bX+t\bm{H}_{\bX})-\bI_n)-2\bX(\bX^{\top}\bX-\bI_n)}{t}\\
&=2\bX(\bX^{\top}\bm{H}_{\bX}+\bm{H}_{\bX}^\top\bX)+2\bm{H}_{\bX}(\bX^\top\bX-\bI_n),\notag
\end{aligned}
\end{equation}
where $\operatorname{D} \nabla h(\bX)\left[\bm{H}_{\bX}\right]$ stands for the classical directional derivative. Thus,
\begin{align*}
\langle\nabla^2 h(\bX)\left[\bm{H}_{\bX}\right],\bm{H}_{\bX}\rangle
&=2\operatorname{tr}(\bm{H}_{\bX}^\top\bX\bX^{\top}\bm{H}_{\bX}+\bm{H}_{\bX}^\top\bX\bm{H}_{\bX}^\top\bX)+2\operatorname{tr}(\bm{H}_{\bX}^\top\bm{H}_{\bX}(\bX^{\top}\bX-\bI_n))\\
&=\|\bX^\top \bm{H}_{\bX}+\bm{H}_{\bX}^\top \bX\|_F^2+2\operatorname{tr}(\bm{H}_{\bX}(\bX^{\top}\bX-\bI_n)\bm{H}_{\bX}^\top)\\
&=\|\bX^\top \bm{H}_{\bX}+\bm{H}_{\bX}^\top \bX\|_F^2+2\|\bX \bm{H}_{\bX}^{\top}\|_F^2-2\|\bm{H}_{\bX}\|_F^2.
% \ge0
\end{align*}
Then we know by definition that $h$ is 2-weakly convex. Moreover,
when $\lambda_{\min}(\bX^\top\bX-\bI_n)\ge0$, one has
$\langle\nabla^2 h(\bX)\left[\bm{H}_{\bX}\right],\bm{H}_{\bX}\rangle\ge0$, and then $h$ is convex over the set $\{\bX\in\mathbb{R}^{m\times n}:\lambda_{\min}(\bX^\top\bX-\bI_n)\ge0\}$.
\end{proof}

By taking into account the composite structure of~\eqref{eq:problem_composite} in the following assumption, we can show that the Lagrangian function $\mathcal{L}_{\rho}(\cdot, \bY)$ is weakly convex.
\begin{ass}\label{ass:basic}
For the problem \eqref{eq:problem_composite}, the function $g:\mathbb{R}^{m\times n}\rightarrow \mathbb{R}$ is $L_g$-weakly convex  on some $\mathcal{X}\subseteq\R^{m\times n}$,
and $\ell:\R^{m\times n}\rightarrow \R$  is  smooth and its gradient is $L_\ell$-Lipschitz continuous on $\mathcal{X}$, i.e.,
\[
\|\nabla \ell(\bX)-\nabla \ell(\bX')\|_F \leq L_{\ell}\|\bX-\bX'\|_F \quad \text { for all } \bX, \bX' \in  \mathcal{X};
\]
Without loss of generality, we may take  $L=L_{\ell}+L_g$.
\end{ass}

\begin{lemma}\label{lem:lagweakcvx}
Suppose that Assumption \ref{ass:basic} holds, then the Lagrangian function $\mathcal{L}_{\rho}(\cdot, \bY)$  is $ \mu_{\rho, \bY} := (L+ 2\| \bY \| + 2\rho)$-weakly convex on $\mathcal{X}$ for any $\bY\in\mathbb{S}^{n}$.
\end{lemma}

\begin{proof}
Recall the definition that 
\[
\mathcal{L}_{\rho}(\bX,\bY)=\ell(\bX)+g(\bX)+\langle \bY,\bX^\top\bX-\bI_n\rangle+\frac{\rho}{2}\|\bX^\top \bX-\bI_n\|_F^2.
\]
By Assumption \ref{ass:basic} we directly know that $\ell$ is $L_{\ell}$-weakly convex and then $f=\ell+g$ is $ L $-weakly convex  on $\mathcal{X}$. On the other hand, for any $ \bX $ and $ \bX' $ in $\mathcal{X}$, we have
\begin{align*}
  \| \bY\bX' - \bY\bX \|_F & = \| \bY(\bX' - \bX) \|_F \leq   \| \bY \| \| \bX' - \bX \|_F.
\end{align*}
It follows that the function $\bX\mapsto \bY \bX $ is $ \| \bY \| $-Lipschitz continuous, implying that $\bX\mapsto \langle \bY, \bX^\top \bX-\bI_n \rangle $ is $2\| \bY \|$-gradient Lipschitz. Combining these results with Lemma \ref{lem:sc}, we obtain that $ \mathcal{L}_{\rho}(\cdot, \bY) $ is $( L+ 2 \| \bY \|+ 2 \rho)$-weakly convex  on $\mathcal{X}$.
\end{proof}

Thus, under the following assumption on the compactness of the dual variable domain 
$\mathcal{Y}$, the augmented Lagrangian function becomes weakly convex.

\begin{ass}\label{ass:ycompact}
The set $\mathcal{Y}$ is convex compact with $\bz\in\mathcal{Y}$ and $\|\bY\|_F\leq R_{\bY}$ for any $\bY\in\mathcal{Y}$.
\end{ass}

Under Assumption \ref{ass:ycompact}, we know from Lemma \ref{lem:lagweakcvx} that for
  \begin{equation*}
 \mu_{\rho} := L + 2R_{\bY} + 2\rho,
\end{equation*} 
it follows $ \mu_{\rho} \geq \mu_{\rho, \bY} $ for all $ \bY \in \mathcal{Y} $ and $ \mathcal{L}_{\rho}(\cdot, \bY) $ is $\mu_{\rho}$-weakly convex on $\mathcal{X}$ for all $ \bY \in \mathcal{Y} $. Then the smoothing parameter can now be properly defined as $r> \mu_{\rho}$.

Despite its theoretical validity, the smoothing scheme \eqref{eq:smoothingscheme} poses significant computational and analytical challenges in practice. First, the presence of the quartic penalty term $\frac{\rho}{2}\|\bX^\top \bX - \bI_n\|_F^2$ implies that the $\bX$-subproblem lacks a closed-form solution, which again requires computationally expensive inner iterations. Second, the convergence of augmented Lagrangian methods often requires a sufficiently large penalty parameter $\rho$ \cite{rockafellar1976augmented, bertsekas2014constrained}. This forces the smoothing parameter $r$ to be excessively large ($r > \mu_{\rho}$). A massive $r$ practically restricts the step size of the primal update, leading to algorithmic instability and extremely slow convergence when high accuracy is desired.

To simultaneously overcome the computational bottleneck and alleviate the reliance on large smoothing parameters, we propose a linearized surrogate for the primal step. Specifically, we linearize the smooth components of the augmented Lagrangian function at a given point $\bar{\bX}\in\R^{m\times n}$, introducing a proximal parameter $\lambda > 0$:
\begin{equation*}
\begin{aligned}
\mathcal{L}_{\bar{\bX},\lambda}(\bX,\bY)
:=\ &
% f(\bar{\bX})+\langle \bY, \bar{\bX}^\top\bar{\bX}- \bI_n \rangle +\frac{\rho}{2} \| \bar{\bX}^\top\bar{\bX} - \bI_{n} \|_F^2+g(\bX)\\
% &+\langle \nabla f(\bar{\bX}), \bX - \bar{\bX} \rangle  +\langle 2\bar{\bX}(\bY+\rho (\bar{\bX}^\top\bar{\bX}-\bI_n)),\bX-\bar{\bX}\rangle+\frac{1}{2\lambda}\|\bX-\bar{\bX}\|_F^2.
\ell(\bar{\bX})+\langle \bY, \bar{\bX}^\top\bar{\bX}- \bI_n \rangle +\frac{\rho}{2} \| \bar{\bX}^\top\bar{\bX} - \bI_{n} \|_F^2+g(\bX)\\
&+\langle \nabla \ell(\bar{\bX}), \bX - \bar{\bX} \rangle  +\langle 2\bar{\bX}(\bY+\rho (\bar{\bX}^\top\bar{\bX}-\bI_n)),\bX-\bar{\bX}\rangle+\frac{1}{2\lambda}\|\bX-\bar{\bX}\|_F^2.
\end{aligned}
\end{equation*}
By replacing the exact augmented Lagrangian with this strictly convex surrogate, we derive the modified explicit primal update on the ancillary set $\mathcal{X}$:
\begin{equation}
\begin{aligned}
    \label{eq:primal_update}
    \bX^{k+1} 
    :=& \argmin_{\bX\in\mathcal{X}}  \left\{\mathcal{L}_{\bX^k,\lambda}(\bX,\bY^k)+\frac{r}{2}\|\bX-\bZ^k\|_F^2\right\}  \\
    % =&\prox_{g / (r+\lambda^{-1})+\iota_{\mathcal{X}}}\left(\frac{\frac{1}{\lambda}\bX^k+r\bZ^k- \left(\nabla f(\bX^k)+2\bX^k\bY^k+2\rho\bX^k((\bX^k)^{\top}\bX^k-\bI_n)\right)}{r+\lambda^{-1}}\right).
    =&\prox_{g / (r+\lambda^{-1})+\iota_{\mathcal{X}}}\left(\frac{\frac{1}{\lambda}\bX^k+r\bZ^k- \left(\nabla \ell(\bX^k)+2\bX^k\bY^k+2\rho\bX^k((\bX^k)^{\top}\bX^k-\bI_n)\right)}{r+\lambda^{-1}}\right).
\end{aligned}
\end{equation}

Our final assumption concerns the ancillary set $\mathcal{X}$, which should encompass the neighborhood of the orthogonality constraints while allowing sufficient flexibility for the intermediate optimization iterates. Additionally, we enforce the compactness of $\mathcal{X}$ to theoretically guarantee the boundedness of the dual variables.

\begin{ass}\label{ass:xcompact}
The set $\mathcal{X}$ in Assumption \ref{ass:basic} is compact with 
\[
\{\bX\in\mathbb{R}^{m\times n}:\|\bX^\top\bX-\bI_n\|_{F}\leq1\}\subseteq\mathcal{X},\]
and $\|\bX\|\leq R^{\operatorname{op}}_{\bX}$, $|\ell(\bX)-\ell(\bX')|\leq L_{\ell}\|\bX-\bX'\|_F$, $|g(\bX)-g(\bX')|\leq L_g\|\bX-\bX'\|_F$ for any $\bX,\bX'\in\mathcal{X}$.
\end{ass}

In the remainder of this paper, we presume that Assumptions \ref{ass:basic}, \ref{ass:ycompact}, and \ref{ass:xcompact} hold. We fix $r > \mu_{\rho}$ to ensure the well-definedness of the smoothed Lagrangian. To further stabilize the dual dynamics and ensure multiplier boundedness in the highly nonconvex landscape, we introduce a small dual regularization parameter $\varepsilon > 0$. The fully explicit, single-loop \textit{Linearized Smoothing ALM (LSALM)} is formally presented in Algorithm \ref{alg}.

\begin{algorithm}[H]\label{alg}
  \caption{Linearized Smoothing ALM (LSALM)}
  \SetKwInOut{Input}{Input}
    \Input{Initial point $\bX^0$ satisfying $(\bX^0)^\top\bX^0=\bI_n$, $\bY^0=\bz$, $\bZ^0=\bX^0$, and $\rho\ge0$, $\lambda>0$, $\alpha>0$, $\beta\in(0,1)$, $\varepsilon>0$.
    }
    \For{$k=0,1,\ldots$}{
$\bX^{k+1}:=\mathop{\argmin}_{\bX\in\mathcal{X}}\left\{\mathcal{L}_{\bX^k,\lambda}(\bX,\bY^k)+\frac{r}{2}\|\bX-\bZ^k\|_F^2\right\} $\\
    $\bZ^{k+1}:=\bZ^k+\beta(\bX^{k+1}-\bZ^k)$\\
    $\bY^{k+1}:=\proj_{\mathcal{Y}}(\bY^k+\alpha\cdot ((\bX^{k+1})^\top\bX^{k+1}-\bI_n-\varepsilon\bY^k))$ \\
     }
\end{algorithm}

\begin{rmk}
The assumptions on the set~\(\mathcal{X}\) are relatively mild and do not significantly increase the computational cost. Thanks to the flexibility in choosing \(\mathcal{X}\), the proximal operator of the composite function \(g + \iota_{\mathcal{X}}\) can often admit a closed-form solution for suitable choices of \(\mathcal{X}\) depending on the structure of \(g\). As a result, the subproblem \eqref{eq:primal_update} can be solved analytically. For example, when \(g(\bX) = 0\), i.e., \(f\) is smooth, the subproblem \eqref{eq:primal_update} reduces to projecting onto \(\mathcal{X}\). In this case, one can choose \(\mathcal{X}\) as a set that admits efficient projection. When \(g(\bX) = \mu \|\bX\|_1\), where $\| \bX \|_1:=\sum_{ij} |\bX_{ij}| $ for some \(\mu > 0\), a natural choice is \(\mathcal{X} = \{\bX \in \mathbb{R}^{m \times n} : |\bX_{i,j}| \leq c,\ \forall i \in [m],\ j \in [n]\}\) for some constant \(c > 0\). In this setting, the proximal operator \(\prox_{g + \iota_{\mathcal{X}}}(\bX)\) is given element-wise by
\[
(\prox_{g+\iota_{\mathcal{X}}}(\bX))_{i,j} = \sign(\bX_{i,j})\cdot\min\{\max\{| \bX_{i,j} | - \mu,0\},c\},\ \ \forall i \in [m], j \in [n].
% (\prox_{g+\iota_{\mathcal{X}}}(\bX))_{i,j} =
% \begin{cases}
% -c & \text{if } \bX_{i,j} \leq -c - \mu, \\
% \bX_{i,j} + \mu & \text{if } -c - \mu < \bX_{i,j} \leq -\mu, \\
% 0 & \text{if } -\mu < \bX_{i,j} < \mu, \\
% \bX_{i,j} - \mu & \text{if } \mu \leq \bX_{i,j} < c + \mu, \\
% c & \text{if } \bX_{i,j} \geq c + \mu.
% \end{cases}
\]
This corresponds to applying element-wise soft-thresholding followed by projection onto the interval \([-c, c]\), which is computationally efficient. In practice, we may also solve the subproblem~\eqref{eq:primal_update} without explicitly enforcing the domain constraint~\(\mathcal{X}\), since the iterates remain bounded and~\(\mathcal{X}\) can always be chosen sufficiently large. This is implicitly reflected in our stepsize choices during the following convergence analysis, as they depend on the radius of~\(\mathcal{X}\).
\end{rmk}

\begin{rmk} 
(i) 
When the proximal mapping is available, our algorithm operates as a single-loop method. In contrast, most existing approaches adopt a double-loop framework, where each iteration requires solving a subproblem. For instance, \cite{chen2020proximal, huang2022riemannian} rely on the semi-smooth Newton method. Other Riemannian approaches introduce different splitting strategies for the manifold constraint and the nonsmooth objective, but still require solving manifold-constrained subproblems at each iteration, e.g., \cite{deng2023manifold, deng2025oracle, xu2025oracle}. To the best of our knowledge, smoothing approaches \cite{beck2023dynamic,peng2023riemannian} and the recently proposed RADMM \cite{li2025riemannian} are among the very few methods that achieve a provably convergent single-loop scheme. However, it is well known that smoothing approaches without linearization become unstable when high-accuracy solutions are required. Furthermore, while the RADMM avoids inner loops, it only establishes a worse iteration complexity;
(ii)
We employ a dual perturbation technique to stabilize the dual update, a similar strategy also adopted in \cite{koshal2011multiuser, hajinezhad2019perturbed, lu2022single}. This perturbation acts as a Tikhonov regularization that analytically induces strong concavity in the dual function allowing us to establish a strict dual error bound and effectively bypass the reliance on the global Linear Independence Constraint Qualification (LICQ). As detailed in the subsequent convergence analysis, this perturbation mechanism serves as a cornerstone for deriving our global theoretical guarantees; (iii) The initial point should be chosen carefully due to the nonconvex nature of the orthogonality constraint. Nevertheless, the explicit structure of the constraint makes it computationally tractable to find a feasible starting point. Without a near-feasible initialization, the algorithm may quickly violate the constraints.
\end{rmk}

\section{Global Convergence of LSALM}\label{sec:global}

In this section, we investigate the convergence behavior of the proposed LSALM. Our analysis follows a structured roadmap: First, we establish the boundedness of the dual variables and ensure the feasibility of the iterates in Section \ref{subsec:feasiblity}. Next, from the sufficient decrease property of a carefully constructed potential function (Proposition \ref{prop:suff-decrease} with the proof in Appendix \ref{subsec:suff}), we derive the iteration complexity results presented in Section \ref{subsec:itercomplex}.

Before proceeding to the proof details of convergence theorems, we define the function $F:\R^{m\times n}\times \mathbb{S}^{n}\times \R^{m\times n}\rightarrow\R$
\[
F(\bX,\bY,\bZ) := \mathcal{L}_{\rho}(\bX,\bY)-\frac{\varepsilon}{2}\|\bY\|_F^2+\frac{r}{2}\|\bX-\bZ\|_F^2.
\]
Then we define the potential function $\Phi:\R^{m\times n}\times \mathbb{S}^{n}\times \R^{m\times n}\rightarrow\R$ as:
\[
\Phi(\bX,\bY,\bZ):= F(\bX,\bY,\bZ)-d(\bY,\bZ) + p(\bZ)-d(\bY,\bZ) + p(\bZ), 
\]
where $d(\bY,\bZ):=\min\limits_{\bX\in \mathcal{X}}F(\bX,\bY,\bZ)$ is the dual function and $p(\bZ):=\max\limits_{\bY\in\mathcal{Y}}\min\limits_{\bX\in\mathcal{X}}F(\bX,\bY,\bZ)$ is the proximal function. The potential function is designed to bridge the algorithmic updates with the underlying target proximal function $p(\cdot)$, which is frequently used in constrained optimization~\cite{zhang2020proximal, zhang2022global} and minimax optimization~\cite{zhang2020single, yang2022faster, li2025nonsmooth}.

To characterize the descent property of the potential function, we first identify the gradient Lipschitz constant of the smooth part of the Lagrangian function, i.e., $\mathcal{L}_{\rho}(\cdot, \bY)-g$ on $\mathcal{X}$.

\begin{lemma}\label{lemma:laglipschitz}
The function $\mathcal{L}_{\rho}(\cdot, \bY)-g$ for any $\bY\in\mathbb{S}^{n}$ is gradient Lipschitz with constant $ L_{\rho} := L_{\ell}+  2R_{\bY} + 6 \rho (R^{\operatorname{op}}_{\bX})^2+2\rho $ on $\mathcal{X}$.
\end{lemma}
\begin{proof}
By definition we know the gradient of $\mathcal{L}_{\rho}(\cdot, \bY)-g$ is
\[
\nabla_{\bX}(\mathcal{L}_{\rho}(\bX,\bY)-g(\bX))=\nabla \ell(\bX)+2 \bY\bX +2\rho\bX (\bX^\top\bX-\bI_n).
\]
Then for any $ \bX $ and $ \bX' $, we have
\begin{align*}
  \| \bY\bX' - \bY\bX \|_F & = \| \bY(\bX' - \bX) \|_F \leq   \| \bY \| \| \bX' - \bX \|_F
\end{align*}
and
\begin{align*}
&\|\bX'\bX'^\top \bX' - \bX\bX^\top \bX\|_F\\
 =\ & \|(\bX'\bX'^\top \bX' - \bX' \bX'^\top \bX) + (\bX' \bX'^\top \bX - \bX' \bX^\top \bX) + (\bX' \bX ^\top \bX - \bX \bX ^\top \bX)\|_F \\
\leq\ & \| \bX'\|^2  \| \bX' - \bX \|_F + \| \bX' \| \| \bX \| \| \bX' - \bX \|_F + \| \bX\|^2  \| \bX' - \bX \|_F.
  \end{align*}
It follows that the mappings $\bX\mapsto \bY \bX $ is $ \| \bY \| $-Lipschitz  and $\bX\mapsto\bX\bX^\top \bX$ is $ 3 \max_{\bX \in \mathcal{X}} \| \bX \|^2 $-Lipschitz, implying that the gradient of $\bX\mapsto \langle \bY, \bX^\top \bX \rangle $ is $2\| \bY \|$-Lipschitz and the gradient of $\bX\mapsto \| \bX^\top \bX - \bI \|^2_F $ is $( 12 \max_{\bX \in \mathcal{X}} \| \bX \|^2 + 4)$-Lipschitz. Combining these results with the fact that $\nabla \ell $ is $ L_{\ell}$-Lipschitz continuous, we obtain that $ \mathcal{L}_{\rho}(\cdot, \bY) $ is gradient Lipschitz with constant $L_{\ell}+ 2 \| \bY \| + 6 \rho \max_{\bX \in \mathcal{X}} \| \bX \|^2 + 2 \rho$. The proof is complete.
\end{proof}

With the gradient Lipschitz continuity established, we can derive the basic descent property of the potential function (see Appendix~\ref{sec:basic_decrease} for details):
\begin{equation}\label{eq:basicdes}
\begin{aligned}
\Phi^k-\Phi^{k+1}\ge\ & 
\Omega\left( \|\bX^{k} - \bX^{k+1} \|_F^2 + \|\bY^{k} - \bY_+^{k}(\bZ^k) \|_F^2 + \frac{1}{\beta} \|\bZ^{k} - \bZ^{k+1} \|_F^2 \right)\\
&- \mathcal{O}\left( \beta \cdot \| \bX(\bY(\bZ^k), \bZ^k) - \bX(\bY_+^{k}(\bZ^k), \bZ^k) \|_F^2 \right)
\end{aligned}
\end{equation}
where $\Phi^k:= \Phi(\bX^k,\bY^k,\bZ^k)$, $\bX(\bY,\bZ):=\mathop{\argmin}\limits_{\bX\in\mathcal{X}}F(\bX,\bY,\bZ)$, $\bY(\bZ):=\mathop{\argmax}\limits_{\bY\in\mathcal{Y}}d(\bY,\bZ)$, and $\bY_{+}(\bZ):=\proj_{\mathcal{Y}}(\bY+\alpha\cdot (\bX(\bY, \bZ)^\top\bX(\bY, \bZ)-\bI_n-\varepsilon\bY))$ is the one-step projected gradient of the dual function.
The remaining part of the proof to establish sufficient descent property relies on another key component: the primal and dual error bound conditions stated in Lemma~\ref{lemma:lip} and Lemma~\ref{prop:dual_eb_KL}, whose proofs follow a strategy similar to that of \cite[Propositions 2 and 3(b)]{li2025nonsmooth}. We omit the proof of Lemma~\ref{lemma:lip} and defer the proof of Lemma~\ref{prop:dual_eb_KL} and the subsequent Proposition~\ref{prop:suff-decrease} to Appendix~\ref{subsec:suff}.

\begin{lemma}[Primal error bound]
\label{lemma:lip}
 For any $k\ge0$, it holds that
    \begin{equation*}
    \label{iter-result1}
        \|\bX^{k+1}-\bX(\bY^{k}, \bZ^{k})\|_F \leq \zeta\|\bX^{k}-\bX^{k+1}\|_F,
        \end{equation*}
where  $\zeta: =\frac{2(r-\mu_{\rho})^{-1}+(\lambda^{-1}+L_{\rho}) ^{-1}}{(\lambda^{-1}+L_{\rho}) ^{-1}} \left(\sqrt{\frac{2L_{\rho}}{\lambda^{-1}+L_{\rho}}}+1\right)$.
\end{lemma}

\begin{lemma}[Dual error bound] 
\label{prop:dual_eb_KL}
For any $(\bY,\bZ)\in\mathbb{S}^{n}\times \R^{m\times n}$, the following inequality holds:
\begin{equation*}
% \label{eq:dualeb1}
\|\bX(\bY_+(\bZ),\bZ)-\bX(\bY(\bZ),\bZ)\|_F \leq \omega\cdot \|\bY_+(\bZ)-\bY\|_F,
\end{equation*}
where $\omega:=\frac{1}{\sqrt{r-\mu_{\rho}}}\left(\frac{1+ (2 \sigma_2 R^{\operatorname{op}}_{\bX}+\varepsilon)\alpha}{\sqrt{\varepsilon}\alpha}\right)$ with $\sigma_{2}:=\frac{2R^{\operatorname{op}}_{\bX}}{r-\mu_{\rho}}$ being the Lipschitz constant of $\bX(\cdot,\bZ)$.
\end{lemma}

\begin{prop}[Sufficient descent property]\label{prop:suff-decrease}
       Let $r\ge \max\{ L_{\rho} + L_g+ 4R^{\operatorname{op}}_{\bX},  3(L_{\rho}+L_g) \}$, $\lambda\le \frac{1}{2R^{\operatorname{op}}_{\bX}}$, $\alpha \leq \min\left\{\frac{1}{20R^{\operatorname{op}}_{\bX}}, \frac{1}{ 8R^{\operatorname{op}}_{\bX}\zeta^2} \right\}$, $\beta  \leq \frac{1}{28}\min\left\{1,\frac{(r-\mu_{\rho})^2}{2\alpha r( R^{\operatorname{op}}_{\bX})^2},\frac{ 1}{16  r\omega^2\alpha}\right\}$. Then for any $k\ge0$,
\begin{equation*}
\Phi^k-\Phi^{k+1}\ge \frac{7}{16\lambda}\|\bX^{k}-\bX^{k+1}\|_F^{2}+\frac{1}{16\alpha}\|\bY^{k}-\bY_+^{k}(\bZ^k)\|_F^2+\frac{4r}{7\beta}\|\bZ^{k}-\bZ^{k+1}\|_F^{2}.
\end{equation*}
\end{prop}

We have now reached a critical stage in which the descent property of a suitably defined potential function is successfully established. However, to derive convergence guarantees, it remains essential to control the boundedness of the dual variable, which plays a vital role in maintaining the feasibility of the algorithm and, subsequently, in ensuring that the limiting point satisfies the KKT conditions.

\subsection{Boundedness of Iterates and Feasibility}\label{subsec:feasiblity}

In our algorithm design, we introduce the auxiliary compact set $\mathcal{Y}$ to facilitate the derivation of sufficient descent. However, it is crucial to ensure that the boundary of $\mathcal{Y}$ is not reached, as we aim to preserve feasibility. We begin by establishing the boundedness of the primal updates.

\begin{prop}[Primal boundedness]\label{prop:primalbd}
Suppose that the parameter conditions in Proposition \ref{prop:suff-decrease} hold with $\{(\bX^k,\bY^k,\bZ^k)\}_{k\in\mathbb{N}}$  being the sequence generated by Algorithm \ref{alg} and  $\alpha\leq\varepsilon^{-1}$. Then for any fixed $ K\in\mathbb{N}_+$ and $\delta>0$, one has
\[
\{\bX^k\}_{0\leq k \leq K} \subseteq \{\bX\in\mathbb{R}^{m\times n}:\|\bX^\top\bX-\bI_n\|_{F}\leq\delta\}
\]
at least $ K - \frac{2K(\Phi^0 - f_{\min})}{\rho \delta^2} $ steps.
\end{prop}

\begin{proof}
Since $\bz\in\mathcal{Y}$, we know from the update of Algorithm \ref{alg} with the convex projection theorem that $\langle \bY^k , \bY^k - (\bY^{k-1} + \alpha (G(\bX^k) - \varepsilon \bY^{k-1})) \rangle \leq 0$. Then we have
\begin{align*}
  \langle \bY^k, G(\bX^k) \rangle - \varepsilon \| \bY^k \|_F^2 &\geq \frac{1}{\alpha} \langle \bY^k, \bY^k - \bY^{k-1} \rangle + \varepsilon \langle \bY^k, \bY^{k-1} \rangle - \varepsilon \| \bY^k \|_F^2 \\
 &  = \frac{1}{\alpha} \langle \bY^k, \bY^{k} - \bY^{k-1} \rangle - \varepsilon \langle \bY^k,\bY^k - \bY^{k-1} \rangle\\
 & = \left(\frac{1}{\alpha} - \varepsilon\right)\cdot \langle \bY^k, \bY^k - \bY^{k-1} \rangle.
\end{align*}
This together with Proposition~\ref{prop:suff-decrease} implies that
\begin{align*}
&f_{\min}+\left(\frac{1}{2\alpha} - \frac{\varepsilon}{2}\right)(\|\bY^k\|_F^2+\|\bY^k-\bY^{k-1}\|_F^2-\|\bY^{k-1}\|_F^2)+\frac{\rho}{2}\|G(\bX^k)\|_F^2\\
= \ & f_{\min} +\left(\frac{1}{\alpha} - \varepsilon\right) \langle \bY^k,\bY^k-\bY^{k-1}\rangle+\frac{\rho}{2}\|G(\bX^k)\|_F^2\\
\leq\ & f(\bX^k)+\langle \bY^k,G(\bX^k)\rangle - \varepsilon \| \bY^k \|^2 +\frac{\rho}{2}\|G(\bX^k)\|_F^2\leq\Phi^k\leq\Phi^0.
\end{align*}
Summing the above inequality over $k$, we derive with $\alpha\leq\varepsilon^{-1}$ that
\begin{align*}
&\left(\frac{\varepsilon}{2}-\frac{1}{2\alpha}\right)\|\bY^{0}\|_F^2+\sum_{k=1}^K \left(f_{\min}+\frac{\rho}{2}\|G(\bX^k)\|_F^2\right)\leq K\cdot\Phi^0,
\end{align*}
which implies
\begin{equation}\label{eq:accum_key}
\sum_{k=1}^K\|G(\bX^k)\|_F^2\leq \frac{2K(\Phi^0 - f_{\min}) + (\alpha^{-1}-\varepsilon)\| \bY^0 \|_F^2}{\rho }\leq\mathcal{O}\left(\frac{K}{\rho}\right),
\end{equation}
where the last inequality is from $G(\bX^0)=\bz$, $\bY^0=\bz$, $\bZ^0=\bX^0$ and  $\Phi^0$ is a constant independent of $\mathcal{Y}$ as 
\[
f_{\min}\leq d(\bY^0,\bZ^0)\leq
p(\bZ^0)=\max\limits_{\bY\in\mathcal{Y}}\min\limits_{\bX\in\mathcal{X}}F(\bX,\bY,\bZ^0)\leq \max\limits_{\bY\in\mathcal{Y}}F(\bX^0,\bY,\bZ^0)\leq f(\bX^0).
\]
Then we know from \eqref{eq:accum_key} that there is at most $\frac{2K(\Phi^0 - f_{\min})}{\rho \delta^2}$ steps such that $\|G(\bX^k)\|_F > \delta$. The proof is complete.
\end{proof}

\begin{rmk}
From the result in Proposition \ref{prop:primalbd}, we refer to the set 
\[
\{ \bm{X} \in \mathbb{R}^{m \times n} : \| \bm{X}^\top \bm{X} - \bm{I}_n \| \leq \delta \}, \quad \delta \in (0,1)
\]
as the region where the \emph{locally uniform constraint qualification} holds. In this region, we have 
\begin{equation}\label{eq:uniformcq}
\sigma_{\min}(\bm{X})=\sqrt{\lambda_{\min}(G(\bX)+\bI_n)}\ge \sqrt{1-\|G(\bX)\|} \geq \sqrt{1 - \delta} > 0,
\end{equation}
which leads to the following error bound:
\[
2\sqrt{1 - \delta} \cdot \| G(\bm{X}) \|_F 
\leq 2\sigma_{\min}(\bm{X}) \cdot \| G(\bm{X}) \|_F 
\leq 2\| \bm{X}\cdot G(\bm{X}) \|_F 
= \left\| \nabla \left( \frac{1}{2} \| G(\cdot) \|_F^2 \right) (\bm{X}) \right\|_F.
\]
This inequality is important because minimizing the function 
$\bm{X} \mapsto \frac{1}{2} \| G(\bm{X}) \|_F^2$
to stationarity in this region ensures that the resulting point satisfies the constraint $G(\bm{X}) = \bz$, i.e., feasibility is recovered.
\end{rmk}

By incorporating the properties of the primal iterates and local UCQ region identified, we derive the following result on the dual variables and the feasibility.

\begin{prop}[Feasibility]
\label{prop:dual-bd-ybd-linearized}
Suppose that the parameter conditions in Proposition \ref{prop:suff-decrease} hold. Let $\xi>0$, $\delta\in(0, 1)$, and $K \in \mathbb{N}_+$. If 
$$
\quad
\rho > \frac{2(\Phi^0 - f_{\min})}{ \delta^2},\quad K>\left(\frac{7}{4\beta\xi}+\frac{2K}{\rho \delta^2}\right)(\Phi^0-f_{\min}), 
$$
and $\mathcal{Y}\supseteq\{\bY\in\mathbb{S}^n:\|\bY\|_F\leq \bar{R}_{\bY}\}$ with
\begin{equation}\label{eq:ryrequire}
R_{\bY}\ge\bar{R}_{\bY} >\frac{L+2\rho  R^{\operatorname{op}}_{\bX} + \max\{ \sqrt{r}, \sqrt{\lambda^{-1}} \}\sqrt{\xi}}{2\sqrt{1-\delta}}.
\end{equation}
Then the sequence $\{(\bX^k,\bY^k,\bZ^k)\}_{0\leq k\leq K}$ generated by Algorithm \ref{alg}  with $\alpha\leq\varepsilon^{-1}$ satisfies
\begin{center}
$
\bY^{k}\in\operatorname{int}(\mathcal{Y})   
$
\end{center}
for at least $K-\left(\frac{7}{4\beta\xi}+\frac{2K}{\rho \delta^2}\right)(\Phi^0-f_{\min})$ steps.
\end{prop}

\begin{proof}

It follows from Proposition~\ref{prop:primalbd} that there are at least $K - \frac{2K(\Phi^0 - f_{\min})}{\rho\delta^2}$ iterations of $\bX^{k+1}$ satisfying $\bX^{k+1} \in \operatorname{int}(\mathcal{X})$ as $\{\bX\in\mathbb{R}^{m\times n}:\|\bX^\top\bX-\bI_n\|_{F}\leq1\}\subseteq\mathcal{X}$ assumed in Assumption \ref{ass:xcompact}. For such $\bX^{k+1}$, 
recall the primal update
$$
\bX^{k+1}=\argmin_{\bX\in\mathcal{X}}  \left\{\mathcal{L}_{\bX^k,\lambda}(\bX,\bY^k)+\frac{r}{2}\|\bX-\bZ^k\|_F^2\right\}=\argmin_{\bX\in\R^{m\times n}}  \left\{\mathcal{L}_{\bX^k,\lambda}(\bX,\bY^k)+\frac{r}{2}\|\bX-\bZ^k\|_F^2\right\}.
$$
From the optimality condition of this subproblem we derive that
\begin{align*}
\bz\in
\nabla \ell(\bX^k)
+\partial g(\bX^{k+1})+2\bX^{k} (\bY^{k}+\rho G(\bX^k))+r(\bX^{k+1}-\bZ^k)+\frac{1}{\lambda}(\bX^{k+1}-\bX^k).
\end{align*}
Hence,
$$
\begin{aligned}
&  2\|\bX^{k} \bY^{k}\|_F\leq \max_{g'\in \partial g(\bX^{k+1})}
\left\|\nabla \ell(\bX^k)+g'+2\rho \bX^{k}G(\bX^k)+r(\bX^{k+1}-\bZ^k)+\frac{1}{\lambda}(\bX^{k+1}-\bX^k)\right\|_F
\end{aligned}
$$
and consequently by $ \max_{g'\in\partial g(\bX^{k+1})} \|\nabla \ell(\bX^k) + g'\|_F \leq L_{\ell}+L_g = L$ and $\| \bX \| \leq R^{\operatorname{op}}_{\bX}$ from Assumption \ref{ass:xcompact}, we know that
\begin{equation}\label{eq:dualbound-key23}
2\|\bX^{k}\bY^{k}\|_F \leq L+ 2\rho  R^{\operatorname{op}}_{\bX}\|G(\bX^k)\|_F+r\|\bX^{k+1}-\bZ^k\|_F+ \|\bX^{k+1}-\bX^k\|_F/\lambda.
\end{equation}

From  Proposition \ref{prop:suff-decrease} we note that   
\begin{align*}
&\frac{4\beta}{7}\sum_{k=0}^K\left\{\frac{1}{\lambda}\|\bX^{k}-\bX^{k+1}\|_F^{2}+r\|\bX^{k+1}-\bZ^k\|_F^{2}\right\}\\
\leq\ & \sum_{k=0}^K\left\{\frac{7}{16\lambda}\|\bX^{k}-\bX^{k+1}\|_F^{2}+\frac{4r\beta}{7}\|\bX^{k+1}-\bZ^k\|_F^{2}\right\}\leq \Phi^0-\Phi^{k+1},
\end{align*}
where the first inequality is from $\beta\leq1/28$. Therefore, we conclude that the iterates must satisfy
$\max\{\|\bX^{k+1}-\bX^k\|_F^2/\lambda,r\|\bX^{k+1}-\bZ^k\|_F^2\}>\xi$ for at most
$7(\Phi^0-\underline{\Phi}) /4\beta\xi$ steps, where $\underline{\Phi}>-\infty$ is the lower bound of the potential function $\Phi$. Here $\underline{\Phi}$ is independent of $\mathcal{Y}$ as
\[
\underline{\Phi}\geq \min_{\bZ} p(\bZ)=\min\limits_{\bZ}\max\limits_{\bY\in\mathcal{Y}}\min\limits_{\bX\in\mathcal{X}}F(\bX,\bY,\bZ)\geq \min\limits_{\bX\in\mathcal{X},\bZ}F(\bX,\bY^0,\bZ)\geq f_{\min}.
\]
Then we have at least $K-7(\Phi^0-f_{\min}) /4\beta\xi$ steps that
$$
\begin{aligned}
\max\left\{\frac{1}{\lambda}\|\bX^{k+1}-\bX^k\|_F^2,r\|\bX^{k+1}-\bZ^k\|_F^2\right\}  \leq \xi.
\end{aligned}
$$
This together with Proposition \ref{prop:primalbd} ($\delta< 1$)  and \eqref{eq:dualbound-key23} implies that at least $K-7(\Phi^0-f_{\min}) /4\beta\xi -2K(\Phi^0 - f_{\min})/\rho \delta^2$ steps
$$
\|\bY^{k}\|_F\leq \frac{1}{\sqrt{1-\delta}}\cdot \|\bX^{k}\bY^{k}\|_F \leq \frac{L+2\rho R^{\operatorname{op}}_{\bX} + \max\{ \sqrt{r}, \sqrt{\lambda^{-1}} \}\sqrt{\xi}}{2\sqrt{1-\delta}},
$$
where the first inequality is from \eqref{eq:uniformcq}. It implies that the projection of the iterates onto $\mathcal{Y}$ will be inactive at least $K-7(\Phi^0-f_{\min}) /4\beta\xi -2K(\Phi^0 - f_{\min})/\rho \delta^2$ steps since the algorithm operates in the regime $\mathcal{Y}\supseteq\{\bY\in\mathbb{S}^n:\|\bY\|_F\leq \bar{R}_{\bY}\}$ and $R_{\bY}\ge\bar{R}_{\bY}>\frac{L+2\rho R^{\operatorname{op}}_{\bX} + \max\{ \sqrt{r}, \sqrt{\lambda^{-1}} \}\sqrt{\xi}}{2\sqrt{1-\delta}}$. The proof is complete.
\end{proof}

\begin{rmk}\label{rmk:feasibility-condition}
In Proposition~\ref{prop:dual-bd-ybd-linearized}, since the parameter $\rho$ is independent of $R_{\bm{Y}}$,  the requirement on $R_{\bm{Y}}$ in \eqref{eq:ryrequire} is $R_{\bY}\ge\Omega(\sqrt{r})$, which is acceptable as it only needs to satisfy $r \geq \Omega(R_{\bY})$ from Proposition~\ref{prop:suff-decrease}. Therefore, we can choose $r$ sufficiently large to ensure that $R_{\bm{Y}}$ meets this condition. By setting  
\[
R_{\bm{Y}} = \frac{L + 2\rho R^{\operatorname{op}}_{\bX} + \sqrt{r\xi}}{\sqrt{1 - \delta}}
\]
and $r \geq \max \left\{\lambda^{-1}, c_1, c_2\right\}$ with
    \begin{align*} 
    &c_1:=\frac{4}{3}\left(L+4R^{\operatorname{op}}_{\bX}  + \frac{2(L + 2\rho R^{\operatorname{op}}_{\bX})\sqrt{1-\delta}+4\xi}{1-\delta} + 6\rho (R^{\operatorname{op}}_{\bX})^2 + 2\rho \right), \\
        &c_2:=4L + \frac{8(L + 2\rho R^{\operatorname{op}}_{\bX})\sqrt{1-\delta} + 48\xi}{1-\delta} + 24 \rho (R^{\operatorname{op}}_{\bX})^2 + 8 \rho, 
    \end{align*}
    we can ensure that $r \geq \max\{ L_{\rho} + L_g + 4R^{\operatorname{op}}_{\bX},  3(L_{\rho}+L_g) \}$  and $R_{\bY} >\frac{L + 2\rho R^{\operatorname{op}}_{\bX} +  \max\{ \sqrt{r}, \sqrt{\lambda^{-1}} \}\sqrt{\xi}}{2\sqrt{1-\delta}} $ as required in Proposition~\ref{prop:suff-decrease} and Proposition~\ref{prop:dual-bd-ybd-linearized}, respectively.
\end{rmk}

As Proposition~\ref{prop:dual-bd-ybd-linearized} ensures that $ \bY^{k+1}\in\operatorname{int}(\mathcal{Y}) $, the LSALM dual update reduces to
\[
\bY^{k+1} :=  \bY^k + \alpha \cdot ( G(\bX^{k+1}) - \varepsilon \bY^k),
\]
with the projection onto $ \mathcal{Y} $ not activated. This allows us to directly relate the feasibility measure $ \| G(\bX^{k+1}) \|_F $ to the relative error $ \| \bY^{k+1} - \bY^k \|_F $, which is controlled via the sufficient descent property.

\subsection{Iteration Complexity of LSALM}\label{subsec:itercomplex}
We now have all the necessary preparations in place. To prove the main global convergence theorem, we first establish a connection between the descent quantities and an $\epsilon$-KKT point.

\begin{lemma} \label{lemma-episolcol}
        Let $\epsilon\ge 0$ and $k\in\mathbb{N}$. Suppose that $(\bX^{k+1},\bY^{k+1},\bZ^{k+1})$ generated by Algorithm \ref{alg} satisfies $\bX^{k+1}\in\operatorname{int}(\mathcal{X})$, $\bY^{k+1}\in\operatorname{int}(\mathcal{Y})$ and
        \[
        \max\left\{\varepsilon,\|\bX^{k+1}-\bX^{k}\|_F,\|\bY^{k}_+(\bZ^{k}) - \bY^k\|_F,\|\bX^{k+1}-\bZ^k\|_F\right\} \leq  \epsilon.
        \]
        Then $(\bX^{k+1}, \bY^{k+1})$ is an $\mathcal{O}(\epsilon)$-KKT point.
        \end{lemma}
\begin{proof}
In accordance with Definition \ref{defi:KKT}, it is necessary to evaluate the two expressions, namely, $\|G(\bX^{k+1})\|_F$ and
$
\dist(-  2\bX^{k+1}\bY^k,\partial f(\bX^{k+1}))
$.
First, from Lemma \ref{lemma:lip} we know that
 \begin{equation}\label{eq:pd-impor}
 \begin{aligned}
& \|\bY^{k+1}-\bY^k_{+}(\bZ^k)\|_F \\
=\ & \|\operatorname{proj}_{\mathcal{Y}}(\bY^k + \alpha \cdot ( G(\bX^{k+1}) - \varepsilon \bY^k))-\operatorname{proj}_{\mathcal{Y}}(\bY^k+\alpha\cdot (G(\bX(\bY^k, \bZ^k))-\varepsilon\bY^k))\|_F \\
\leq\ &   2\alpha R^{\operatorname{op}}_{\bX}\|\bX^{k+1}-\bX(\bY^k, \bZ^k)\|_F \\
\leq\ &   2\alpha R^{\operatorname{op}}_{\bX} \zeta\|\bX^k-\bX^{k+1}\|_F.
\end{aligned}
\end{equation}
Then from $\bY^{k+1}\in\operatorname{int}(\mathcal{Y})$ we have
\begin{equation}
\begin{aligned}
\|G(\bX^{k+1})\|_F
\leq\ &\|G(\bX^{k+1})-\varepsilon\bY^k\|_F+R_{\bY}\varepsilon\\
=\ &\frac{1}{\alpha}\|\bY^{k+1}-\bY^k\|_F+R_{\bY}\varepsilon\\
\leq\ & \frac{1}{\alpha}\|\bY^{k}-\bY_{+}^k(\bZ^k)\|_F+\frac{1}{\alpha}\|\bY^{k+1}-\bY_{+}^k(\bZ^k)\|_F + R_{\bY}\varepsilon\\
\leq\ & \frac{1}{\alpha}\|\bY^{k}-\bY_{+}^k(\bZ^k)\|_F+ 2R^{\operatorname{op}}_{\bX} \zeta\|\bX^k-\bX^{k+1}\|_F + R_{\bY}\varepsilon \leq \left(\frac{1}{\alpha}+ 2R^{\operatorname{op}}_{\bX}\zeta+R_{\bY}\right)\epsilon.\label{eq:dualmeasure}
\end{aligned}
\end{equation}

On the other hand, it follows from $\bX^{k+1}\in\operatorname{int}(\mathcal{X})$ that $\bX^{k+1}$ satisfies the optimality condition 
  \begin{equation}\label{eq:X-optimality-condition}
\bz \in \nabla \ell(\bX^k) + \partial g(\bX^{k+1}) + 2\bX^k(\bY^k + \rho G(\bX^k)) + \frac{1}{\lambda} (\bX^{k+1} - \bX^k) + r(\bX^{k+1} - \bZ^k).
\end{equation}
Plugging the above equation into the primal stationary measure, 
we obtain
\begin{align*}
 &\dist(-  2\bX^{k+1}\bY^k,\partial f(\bX^{k+1}))\\
  =\ &  \dist(\bz, \nabla \ell(\bX^{k+1}) + \partial g(\bX^{k+1}) + 2\bX^{k+1}\bY^k )) \\
  \leq\ &  \left\| \nabla \ell (\bX^{k+1}) - \nabla \ell(\bX^k) + 2 (\bX^{k+1} - \bX^k) \bY^k - 2\rho \bX^kG(\bX^k)  - \frac{1}{\lambda} (\bX^{k+1} - \bX^k) - r (\bX^{k+1} - \bZ^k) \right\|_F\\
  \leq\ &  (L_{\ell}+2R_{\bY}+ \lambda^{-1}) \| \bX^{k+1} - \bX^k \|_F+ 2\rho  R^{\operatorname{op}}_{\bX}\|G(\bX^k)\|_F + r \| \bX^{k+1} - \bZ^k \|_F\\
\leq\ &  (L_{\ell}+2R_{\bY}+ \lambda^{-1}+4\rho (R^{\operatorname{op}}_{\bX})^2 + r)\epsilon + 2\rho  R^{\operatorname{op}}_{\bX}\|G(\bX^{k+1})\|_F\\
\leq\ & (L_{\ell}+2R_{\bY}+ \lambda^{-1}+4\rho (R^{\operatorname{op}}_{\bX})^2 +r)\epsilon + 2\rho  R^{\operatorname{op}}_{\bX}\left(\frac{1}{\alpha}+ 2R^{\operatorname{op}}_{\bX}\zeta+R_{\bY}\right)\epsilon,
\end{align*}
where the third inequality is from 
\begin{align*}
    \|G(\bX^k)\|_F &\leq \|G(\bX^{k+1})\|_F + \|(\bX^{k})^\top\bX^k - (\bX^{k+1})^\top\bX^{k+1}\|_F \\
    &\leq \|G(\bX^{k+1})\|_F + 2 R^{\operatorname{op}}_{\bX}\|\bX^k - \bX^{k+1}\|_F,
\end{align*}
and the last inequality is from \eqref{eq:dualmeasure}. The proof is complete.
\end{proof}

Armed with Propositions \ref{prop:suff-decrease}, \ref{prop:primalbd}, \ref{prop:dual-bd-ybd-linearized} and Lemma \ref{lemma-episolcol}, we present the main theorem concerning the iteration complexity of LSALM.

\begin{thm}[Iteration complexity]
\label{thm:general}
Suppose that the assumptions of Propositions \ref{prop:suff-decrease} and \ref{prop:dual-bd-ybd-linearized} hold and the sequence $\{(\bX^k,\bY^k,\bZ^k)\}_{0\leq k\leq K}$ is generated by Algorithm \ref{alg}. Then there exists a $k \in\{0,1, \ldots, K-1\}$ such that $(\bX^{k+1}, \bY^{k+1})$ is an $\mathcal{O}(K^{-1/3})$-KKT point
  if    $\varepsilon= \mathcal{O}(K^{-1/3})$ and $\beta  =\mathcal{O}(K^{-1/3})$.
\end{thm}

\begin{proof}
Denote the index set $\mathcal{J}$ being the inactive set in Proposition \ref{prop:dual-bd-ybd-linearized} with 
\[
|\mathcal{J}|=K-\left(\frac{7}{4\beta\xi}+\frac{2K}{\rho \delta^2}\right)(\Phi^0-f_{\min}).
\]
From Proposition \ref{prop:suff-decrease} we have for any $k \in\{0,1, \ldots, K-1\}$ that
\begin{equation*}
\begin{aligned}
\Phi^k-\Phi^{k+1}\ge\ & \frac{7}{16\lambda}\|\bX^{k}-\bX^{k+1}\|_F^{2}+\frac{1}{16\alpha}\|\bY^{k}-\bY_+^{k}(\bZ^k)\|_F^2 +\frac{4r}{7\beta}\|\bZ^{k}-\bZ^{k+1}\|_F^{2}.
\end{aligned}
\end{equation*}
This together with
\begin{align*}
\Phi(\bX, \bY, \bZ) &=p(\bZ)+(F(\bX,\bY,\bZ)-d(\bY, \bZ))+(p(\bZ)-d(\bY, \bZ))\geq p(\bZ) \geq \underline{\Phi}\ge f_{\min}>-\infty
\end{align*}
implies that
    \begin{align*}
    \Phi^0-f_{\min} \geq\ & \sum_{k=0}^{K-1}\Phi(\bX^k,\bY^k,\bZ^k)-\Phi(\bX^{k+1},\bY^{k+1},\bZ^{k+1}) \\
    \geq\ & \sum_{k+1\in\mathcal{J}}\Phi(\bX^k,\bY^k,\bZ^k)-\Phi(\bX^{k+1},\bY^{k+1},\bZ^{k+1}) \\
    \geq\ & \sum_{k+1\in\mathcal{J}} \min\left\{\frac{7}{16\lambda},\frac{1}{16\alpha},\frac{4r\beta}{7}\right\} \left(\|\bX^{k}-\bX^{k+1}\|_F^{2}+\|\bY^{k}-\bY_{+}^{k}(\bZ^{k})\|_F^2 +\|\bZ^{k}-\bX^{k+1}\|_F^{2} \right).
    \end{align*}
    Therefore, there exists a $k\in\{0,1, \ldots, K-1\}$ such that $k+1 \in\mathcal{J}$ and
    \[
    \max \left\{\|\bX^{k}-\bX^{k+1}\|_F^{2},\|\bY^{k}-\bY_{+}^{k}(\bZ^{k})\|_F^2,  \|\bZ^{k}-\bX^{k+1}\|_F^{2}\right\} \leq \frac{\Phi^0-f_{\min}}{\min\left\{\frac{7}{16\lambda},\frac{1}{16\alpha},\frac{4r\beta}{7}\right\} |\mathcal{J}|}.
    \]
Since Proposition \ref{prop:suff-decrease} requires $\beta  \leq \frac{ 1}{448  r\omega^2\alpha}=\mathcal{O}(\varepsilon)$ as $w=\mathcal{O}(1/\sqrt{\varepsilon})$ known from Lemma \ref{prop:dual_eb_KL}, we set $\varepsilon= \mathcal{O}(K^{-1/3})$ and
  $\beta  =\mathcal{O}(K^{-1/3})$. This indicates that 
\[
    \max \left\{\varepsilon,\|\bX^{k}-\bX^{k+1}\|_F,\|\bY^{k}-\bY_{+}^{k}(\bZ^{k})\|_F,  \|\bZ^{k}-\bX^{k+1}\|_F\right\} = \mathcal{O}(K^{-1/3}).
    \]
As established in Propositions~\ref{prop:primalbd} and \ref{prop:dual-bd-ybd-linearized}, we have $ \bX^{k+1} \in \operatorname{int}(\mathcal{X}) $ and $ \bY^{k+1} \in \operatorname{int}(\mathcal{Y}) $ for $k+1 \in\mathcal{J}$. Therefore, the proof is completed by invoking Lemma~\ref{lemma-episolcol}.
\end{proof}

\section{Sequential Convergence Analysis}\label{sec:sequential}
We are going to further investigate the asymptotic convergence properties of the iterates generated by our proposed LSALM. Based on the classical mild Kurdyka-\L ojasiewicz (K\L) property assumption guaranteed by the semi-algebraic structure, we have the  sequential convergence results. To begin with, we show in the following lemma that the sequence $\{\bZ^k\}_{k \in \mathbb{N}}$ is bounded. Consequently, the sequence $\{(\bX^k, \bY^k, \bZ^k)\}_{k \in \mathbb{N}}$ generated by LSALM admits a cluster point.

\begin{lemma}\label{lem:zbound}
The  sequence $\{\bZ^k\}_{k \in \mathbb{N}}$ generated by Algorithm \ref{alg} satisfies $\|\bZ^k\|\leq  R^{\operatorname{op}}_{\bX}$ for each $k\ge0$.
\end{lemma}
\begin{proof}
Since $\bZ^{k+1}:=\bZ^k+\beta(\bX^{k+1}-\bZ^k)$ for each $k\ge0$, we know that
\begin{align*}
\|\bZ^{k+1}\|
&=\|(1-\beta)\bZ^k+\beta\bX^{k+1}\|\notag\\
&=\|(1-\beta)^2\bZ^{k-1}+(1-\beta)\beta\bX^k+\beta\bX^{k+1}\|\notag\\
&\ \ \ \ \ldots\notag\\
&=\left\|(1-\beta)^{k+1}\bZ^{0}+\sum_{i=0}^k(1-\beta)^{i}\beta\bX^{k-i+1}\right\|\notag\\
&\leq (1-\beta)^{k+1}\|\bZ^0\|+\frac{1-(1-\beta)^{k+1}}{1-(1-\beta)}\cdot\beta  R^{\operatorname{op}}_{\bX}\leq R^{\operatorname{op}}_{\bX},
\end{align*}
where the last inequality is from $\|\bZ^0\|=\|\bX^0\|\leq R^{\operatorname{op}}_{\bX}$ and $\beta\in(0,1)$. The proof is complete.
\end{proof}

We can now derive the sequential convergence of the algorithm under the standard semi-algebraic setting. Moreover, with the parameter $\rho$ and the set $\mathcal{Y}$ chosen sufficiently large as in Proposition~\ref{prop:dual-bd-ybd-linearized}, the limiting point is guaranteed to be an $\mathcal{O}(\varepsilon)$-KKT point.

\begin{thm}[Sequential convergence]\label{thm:sequential-convergence}
  Suppose that the function $ f $ is semi-algebraic, $ g $ is continuous on $\mathcal{X}$, and $\mathcal{X}, \mathcal{Y}$ are semi-algebraic sets. Then under the assumptions of Proposition~\ref{prop:suff-decrease}, the sequence $\{ (\bX^k, \bY^k) \}_{k\in\mathbb{N}}$ generated by Algorithm \ref{alg} converges to a point $(\bX^*, \bY^*)$. Furthermore, if the assumptions of Proposition~\ref{prop:dual-bd-ybd-linearized}
  hold, then $(\bX^*, \bY^*)$ is an $ \mathcal{O}(\varepsilon)$-KKT point.
\end{thm}
\begin{proof}
Denote the function $\bar{\Phi}:\R^{m\times n}\times \mathbb{S}^{n}\times \R^{m\times n}\rightarrow\overline{\R}$:
\[
\bar{\Phi}(\bX, \bY, \bZ) := \Phi(\bX, \bY, \bZ) + \iota_{\mathcal{X}}(\bX) + \iota_{\mathcal{Y}}(\bY).
\]
Since $ \bX^k \in \mathcal{X} $ and $ \bY^k \in \mathcal{Y} $ for all $ k \geq 0 $, the sequence $\{\bar{\Phi}(\bX^k, \bY^k, \bZ^k)\}_{k\in\mathbb{N}}$ has the same sufficient decrease property as Proposition~\ref{prop:suff-decrease}. Based on the classical K\L~framework \cite{attouch2013convergence}, to prove the sequential convergence, we need to show the {\it relative error condition} also holds, i.e., the distance to the following subdifferential sets can be upper bounded by the relative change in the algorithmic iterates:
\begin{align}
  \partial_{\bX} \bar{\Phi}(\bX^{k+1}, \bY^{k+1}, \bZ^{k+1})  &=  \partial (F(\cdot , \bY^{k+1}, \bZ^{k+1}) + \iota_{\mathcal{X}}(\cdot))(\bX^{k+1}) ,\label{eq:partial-X-Phi}  \\
   \partial_{\bY} \bar{\Phi}(\bX^{k+1}, \bY^{k+1}, \bZ^{k+1})&=  \nabla_{\bY} F(\bX^{k+1}, \bY^{k+1}, \bZ^{k+1}) - 2 \nabla_{\bY} d(\bY^{k+1}, \bZ^{k+1}) + \partial \iota_{\mathcal{Y}}(\bY^{k+1}) ,\label{eq:partial-Y-Phi} \\
\nabla_{\bZ} \bar{\Phi}(\bX^{k+1}, \bY^{k+1}, \bZ^{k+1}) &=   \nabla_{\bZ}F(\bX^{k+1}, \bY^{k+1}, \bZ^{k+1}) - 2 \nabla_{\bZ} d(\bY^{k+1}, \bZ^{k+1}) + 2 \nabla_{\bZ} p(\bZ^{k+1}) .\label{eq:partial-Z-Phi}
\end{align}
First, it follows from the update rule of $ \bX $ in~\eqref{eq:primal_update} and the definition of $ F $ that
  \[
  \bz \in \nabla  \ell(\bX^k) + \partial (g+\iota_{\mathcal{X}})(\bX^{k+1}) + 2\bX^k(\bY^k+\rho G(\bX^k)) + \frac{1}{\lambda}(\bX^{k+1} - \bX^k) + r(\bX^{k+1} - \bZ^k).
  \]
  This together with 
  \begin{equation}\label{eq:seq-key1}
  \partial_{\bX} F(\bX^{k+1}, \bY^k, \bZ^k) = \nabla \ell(\bX^{k+1}) + \partial g (\bX^{k+1}) + 2\bX^{k+1}(\bY^k + \rho G(\bX^{k+1})) + r(\bX^{k+1} - \bZ^k)    
  \end{equation}
  and the Lipschitz constant computed in Lemma~\ref{lemma:laglipschitz} implies that
  \begin{align}\label{eq:seq-key0}
    &\ \dist(\bz, \partial (F( \cdot, \bY^k, \bZ^k) + \iota_{\mathcal{X}}(\cdot))(\bX^{k+1}))\notag\\
    \leq & \ \left\| \nabla \ell(\bX^{k+1}) - \nabla \ell(\bX^{k}) + 2(\bX^{k+1} - \bX^{k}) \bY^k + 2 \rho (\bX^{k+1}G(\bX^{k+1}) - \bX^k G(\bX^k)) + \frac{1}{\lambda}(\bX^{k} - \bX^{k+1}) \right\|_F\notag\\
    \leq & \ (L_{\rho}+\lambda^{-1}) \| \bX^{k+1} - \bX^k \|_F.
\end{align}
On the other hand, from \eqref{eq:seq-key1} we know
  \[
  \partial (F( \cdot, \bY^k, \bZ^k) + \iota_{\mathcal{X}}(\cdot))(\bX^{k+1}) = \partial (F( \cdot, \bY^{k+1}, \bZ^{k+1}) + \iota_{\mathcal{X}}(\cdot))(\bX^{k+1}) + 2\bX^{k+1}(\bY^k - \bY^{k+1}) + r(\bZ^{k+1}- \bZ^k),
  \]
which combined with \eqref{eq:partial-X-Phi} and \eqref{eq:seq-key0} implies that
\begin{align}\label{eq:seq-converge-partial-X-ineq}
  &\ \dist(\bz, \partial_{\bX} \bar{\Phi}( \bX^{k+1}, \bY^{k+1}, \bZ^{k+1}))\notag\\
  =&\ \dist(\bz, \partial (F( \cdot, \bY^{k+1}, \bZ^{k+1}) + \iota_{\mathcal{X}}(\cdot))(\bX^{k+1}))\notag\\
  \leq& \ \dist(\bz, \partial (F( \cdot, \bY^k, \bZ^k) + \iota_{\mathcal{X}}(\cdot))(\bX^{k+1})) +  2 \| \bX^{k+1} \| \| \bY^{k+1} - \bY^k \|_F + r \| \bZ^{k+1} - \bZ^k \|_F\notag\\
\leq&\ (L_{\rho}+\lambda^{-1}) \| \bX^{k+1} - \bX^k \|_F + 2 R^{\operatorname{op}}_{\bX} \| \bY^{k+1} - \bY^k \|_F + r \| \bZ^{k+1} - \bZ^k \|_F.
\end{align}

Next, it follows from the update rule of $ \bY $, the definition of $ F $, and the fact that $ \alpha \partial \iota_{\mathcal{Y}}(\bY^{k+1}) = \partial \iota_{\mathcal{Y}}(\bY^{k+1}) $ that
\begin{align*}
   &\ (1 - \alpha \varepsilon) (\bY^{k+1} - \bY^k) - \alpha \nabla_{\bY} F(\bX^{k+1}, \bY^{k+1}, \bZ^{k+1}) + \alpha \partial \iota_{\mathcal{Y}}(\bY^{k+1})\\
  = & \ \bY^{k+1} - (\bY^k + \alpha \nabla_{\bY} F(\bX^{k+1}, \bY^k, \bZ^{k+1})) + \partial \iota_{\mathcal{Y}}(\bY^{k+1}) \ni \bz.
\end{align*}
Substituting the above equation into \eqref{eq:partial-Y-Phi} and using Lemma~\ref{lemma-dualdiff}, we have
\begin{align}\label{eq:seq-converge-partial-Y-ineq}
  &\ \dist(\bz, \partial_{\bY} \bar{\Phi}( \bX^{k+1}, \bY^{k+1}, \bZ^{k+1}))\notag\\
  \leq &\ \left(\frac{1}{\alpha} - \varepsilon\right)\| \bY^{k+1} - \bY^k \|_F + 2 \| \nabla_{\bY} F(\bX^{k+1}, \bY^{k+1}, \bZ^{k+1}) - \nabla_{\bY} d(\bY^{k+1}, \bZ^{k+1}) \|_F\notag\\
  = &\ \left(\frac{1}{\alpha} - \varepsilon\right)\| \bY^{k+1} - \bY^k \|_F + 2 \| G(\bX^{k+1}) -  G(\bX(\bY^{k+1}, \bZ^{k+1})) \|_F\notag\\
  \leq &\ \left(\frac{1}{\alpha} - \varepsilon\right)\| \bY^{k+1} - \bY^k \|_F 
  +  4R^{\operatorname{op}}_{\bX} \| \bX(\bY^{k+1}, \bZ^{k+1}) - \bX(\bY^{k}, \bZ^{k}) \|_F 
  +  4R^{\operatorname{op}}_{\bX}\| \bX^{k+1} - \bX(\bY^{k}, \bZ^{k}) \|_F\notag\\
  \leq &\ \left(\frac{1}{\alpha} - \varepsilon+4R^{\operatorname{op}}_{\bX}\sigma_2\right)\| \bY^{k+1} - \bY^k \|_F +  4R^{\operatorname{op}}_{\bX}\zeta \| \bX^{k+1} - \bX^k \|_F+4R^{\operatorname{op}}_{\bX}\sigma_1 \| \bZ^{k+1} - \bZ^k \|_F,
\end{align}
where the second inequality follows from the Lipschitz continuity of $ G(\cdot) $ on $\mathcal{X}$ and the last inequality follows from Lemmas~\ref{lemma:lip} and \ref{lemma-sollip}.

Finally, using Lemmas~\ref{lemma-dualdiff} and~\ref{lem-proximal-smooth}, we have
\begin{align}\label{eq:seq-converge-partial-Z-ineq}
  &\ \| \nabla_{\bZ} \bar{\Phi}(\bX^{k+1}, \bY^{k+1}, \bZ^{k+1}) \|_F \notag\\
  \leq& \ r \| \bX^{k+1} - \bZ^{k+1} \|_F + 2 \| \nabla p(\bZ^{k+1}) - \nabla_{\bZ} d(\bY^{k+1}, \bZ^{k+1}) \|_F\notag \\
  \leq & \  r (\| \bX^{k+1} - \bZ^k \|_F + \| \bZ^k - \bZ^{k+1} \|_F) + 2r \| \bX(\bY^{k+1}, \bZ^{k+1}) - \bX(\bZ^{k+1}) \|_F\notag\\
  \leq & \  r\left(\frac{1}{\beta} + 1 + 4 \sigma_1\right) \| \bZ^{k+1} - \bZ^k \|_F 
  % + 2r\sigma_2 \| \bY^{k+1} - \bY^k \|_F 
  + 2r \| \bX(\bY^{k+1}, \bZ^{k}) - \bX(\bZ^{k}) \|_F\notag\\
  \leq & \ r\left(\frac{1}{\beta} + 1 + 4 \sigma_1\right) \| \bZ^{k+1} - \bZ^k \|_F 
  + 2r\omega \| \bY^{k+1} - \bY^k \|_F 
  + 4r\alpha R^{\operatorname{op}}_{\bX} \zeta(\omega + \sigma_{2}) \|\bX^{k+1} - \bX^k\|_F,
\end{align}
where $\bX(\bZ):=\mathop{\argmin}\limits_{\bX\in\mathcal{X}}F(\bX,\bY(\bZ),\bZ)$.
Here the last inequality follows from
\begin{align*}
  \| \bX(\bY^{k+1}, \bZ^k) - \bX(\bZ^k) \| & \leq \| \bX(\bY_+^k(\bZ^k), \bZ^k) - \bX(\bY(\bZ^k), \bZ^k) \|_F + \| \bX(\bY^{k+1}, \bZ^k) - \bX(\bY_+^k(\bZ^k), \bZ^k) \|_F \\
                             & \leq \omega \| \bY^k_+(\bZ^k) - \bY^k \|_F + \sigma_2 \| \bY_+^k(\bZ^k) - \bY^{k+1} \|_F\\
  &\leq   2\alpha R^{\operatorname{op}}_{\bX} \zeta(\omega+\sigma_2) \| \bX^{k+1} - \bX^k \|_F+\omega \| \bY^{k+1} - \bY^k \|_F,
\end{align*}
where the second inequality is due to the dual error bound Proposition~\ref{prop:dual_eb_KL} and \eqref{lip-y}, and the last inequality is due to \eqref{eq:pd-impor}.

Note that from \eqref{eq:pd-impor} we have
\begin{align}\label{eq:seq-converge-Yk-Yplus}
  \| \bY^{k+1} - \bY^k \|_F & \leq \| \bY^k_+(\bZ^k) - \bY^k \|_F + \| \bY^{k+1} - \bY^k_+(\bZ^k) \|_F\notag\\
  & \leq  \| \bY^k_+(\bZ^k) - \bY^k \|_F +  2\alpha R^{\operatorname{op}}_{\bX} \zeta\| \bX^{k+1} - \bX^k\|_F.
\end{align}
Combining \eqref{eq:seq-converge-partial-X-ineq}, \eqref{eq:seq-converge-partial-Y-ineq}, \eqref{eq:seq-converge-partial-Z-ineq} and \eqref{eq:seq-converge-Yk-Yplus}, we know that there exists some constant $ \tilde{c} > 0 $ such that the following {\it relative error condition} property holds:
 \[ \dist(\bz, \partial \bar{\Phi}(\bX^{k+1}, \bY^{k+1}, \bZ^{k+1})) \leq \tilde{c}\cdot (\| \bX^{k+1} - \bX^k \|_F + \| \bY_+^{k}(\bZ^{k}) - \bY^k \|_F + \|\bZ^{k+1} - \bZ^k\|_F). \]
By Lemma \ref{lem:zbound} we know $\bZ^k\in\mathcal{X}$, and then the sequence $\{ (\bX^k, \bY^k, \bZ^k) \}_{k\in\mathbb{N}}$ is bounded and has a cluster point. Also, by our assumption $ F $ is continuous on $ \mathcal{X} \times \mathcal{Y} \times \mathcal{X} $.
According to the results in \cite[Example 2]{bolte2014proximal} and our assumption that $ f $ is semi-algebraic and the sets $ \mathcal{X}, \mathcal{Y} $ are semi-algebraic, we know that $\bar{\Phi}$ is semi-algebraic, and consequently by \cite[Theorem 3.1 and Remark 3.2]{bolte2007lojasiewicz} (noting that a semi-algebraic function is subanalytic) we know $\bar{\Phi}$ is a K\L{} function.  Building on the unified convergence analysis framework in \cite[Theorem 2.9]{attouch2013convergence}, the sequence $ \{ (\bX^k, \bY^k, \bZ^k) \}_{k \in \mathbf{N}} $ is convergent.

We now consider the case where the assumptions of Proposition~\ref{prop:dual-bd-ybd-linearized} are satisfied. It follows from the discussion therein that for any $ K > 0 $, there are at least
\[
K - \left( \frac{7}{4\beta\xi} + \frac{2K}{\rho \delta^2} \right)(\Phi^0 - f_{\min})
\]  
iterations of the sequence $ \{ (\bm{X}^k, \bm{Y}^k) \}_{0 \leq k \leq K} $ lying in $ \operatorname{int}(\mathcal{X}) \times \operatorname{int}(\mathcal{Y}) $. Since $ K $ can be chosen arbitrarily, we conclude that there exists a subsequence of $ \{ (\bm{X}^k, \bm{Y}^k) \}_{k \in \mathbb{N}} $ in $ \operatorname{int}(\mathcal{X}) \times \operatorname{int}(\mathcal{Y}) $. Combining this with the fact that the entire sequence $ \{ (\bm{X}^k, \bm{Y}^k) \}_{k \in \mathbb{N}} $ converges, we know that there exists $ K_0 > 0 $ such that
\[
(\bm{X}^k, \bm{Y}^k) \in \operatorname{int}(\mathcal{X}) \times \operatorname{int}(\mathcal{Y}) \quad \text{for all } k \geq K_0.
\]
From the update rule of $ \bm{Z}^k $, we obtain $ \bm{Z}^* = \lim_{k \to \infty} \bm{Z}^k = \bm{X}^* $.
Also, for $ k \geq K_0 $, since $ \bm{Y}^k \in \operatorname{int}(\mathcal{Y}) $, it follows from the update rule of $ \bm{Y}^k $ that
\[
\| G(\bm{X}^{k+1}) \|_F \leq \frac{\| \bm{Y}^{k+1} - \bm{Y}^k \|_F}{\alpha} + \varepsilon \| \bm{Y}^k \|_F.
\]
Letting $ k \to \infty $, we obtain $ \| G(\bm{X}^*) \|_F \leq \varepsilon R_{\bm{Y}} $. Substituting this into~\eqref{eq:X-optimality-condition}, and using $ \bm{Z}^* = \bm{X}^* $ along with the definition of the limiting subdifferential, we conclude that
\[
\operatorname{dist}(-2 \bm{X}^* \bm{Y}^*, \partial f(\bm{X}^*)) = \lim_{k \to \infty} \operatorname{dist}(-2 \bm{X}^k \bm{Y}^k, \partial f(\bm{X}^k)) \leq 2\rho  R^{\operatorname{op}}_{\bX} R_{\bm{Y}} \varepsilon.
\]
The proof is complete.
\end{proof}

Unlike the complexity result in Theorem~\ref{thm:general}, which provides only a finite-step guarantee with the number of iterations determined by both $\beta$ and $\varepsilon$, Theorem~\ref{thm:sequential-convergence} guarantees asymptotic convergence over infinite steps, independent of the specific choice of $\beta$. Specifically, for any $\beta$ satisfying the conditions in Proposition~\ref{prop:dual-bd-ybd-linearized}, the entire sequence generated by LSALM converges asymptotically to an $\mathcal{O}(\varepsilon)$-KKT point.

\section{Numerical Results}
\label{sec:numerical}
In this part, we conduct numerical experiments to evaluate the performance of our LSALM on randomly generated nonsmooth quadratic problems. We also compare its performance on the nonsmooth problem (sparse PCA) and the smooth problem (graph matching) with state-of-the-art algorithms, respectively. All experiments are implemented in MATLAB 2025b and run on a machine with an Intel i5-14500 CPU (14 cores) and 32 GB of RAM.

\subsection{Randomly Generated Quadratic Problems}
Firstly, we demonstrate the robustness of LSALM regarding the algorithm parameters $(\rho, \lambda, r, \alpha, \beta)$ via the following quadratic problem (QP) with nonsmooth $\ell_1$ norm:
\begin{equation}\label{eq:prob-QP}
\begin{array}{c@{\quad}l}
\min\limits_{\bX\in\R^{m\times n}} &f(\bX):=\frac{1}{2} \tr(\bX^\top \bm{A} \bX) + \tr(\bm{G}^\top \bX) + \mu \| \bX \|_1\\
{\rm s.t.} &\bX^\top\bX=\bI_n,
\end{array}
\end{equation}
where the $\ell_1$ norm is defined as $\|\bX\|_1:=\sum_{i j}|\bX_{i j}|$.
The smooth case where \(\mu = 0\) is firstly used in \cite{gao2019parallelizable}.
The matrices \(\bm{A} \in \R^{m \times m}\) and \(\bm{G} \in \R^{m \times n}\) are generated as follows: \(\bm{A} = \bm{P} \bm{L} \bm{P}^\top\), where \(\bm{L} \in \R^{m \times m}\) is a diagonal matrix with entries \(\bm{L}_{ii} = 1.01^{1-i}\) for \(1 \leq i \leq m\), and \(\bm{P} \in \R^{m \times m}\) is an orthogonal matrix obtained from the QR decomposition of a random matrix, i.e., \texttt{qr(rand(m,m))}. The matrix \(\bm{G} = \bm{Q}\bm{D}\), where \(\bm{Q} \in \R^{m \times n}\) consists of columns \(\bm{Q}_i = \tilde{\bm{Q}}_i / \| \tilde{\bm{Q}}_i \|_2\) for \(1 \leq i \leq n\), with \(\tilde{\bm{Q}}_i\) being the \(i\)-th column of a random matrix \(\tilde{\bm{Q}} = \texttt{rand(m,n)}\). Additionally, \(\bm{D} \in \R^{n \times n}\) is a diagonal matrix with entries \(\bm{D}_{ii} = 1.01^{i-1}\) for \(1 \leq i \leq n\).
Our goal is to demonstrate the impact of the parameters of our algorithm through this problem.

We demonstrate the robustness of our algorithm by examining its performance on the above problem with $(m, n, \mu) = (20, 2, 0.35)$ across various parameter settings. For the LSALM, we set a baseline set of parameters as: $\rho = 0.15$, $\lambda = 1.35$, $\varepsilon = 10^{-8}$, $R_{\bX}^{\operatorname{op}} = 10$, $R_{\bY} = 5$, $r = 1.25$, $\alpha = 0.1$, and $\beta = 0.44$.  Then we individually adjust the parameters $r$, $\lambda$, $\rho$, $\alpha$, and $\beta$, with other parameters remained at their baseline values, to determine their respective ranges for convergence.

\begin{rmk}\label{rmk:beta-gap}
While Lemma \ref{prop:dual_eb_KL} dictates a conservative theoretical bound $\beta \le \mathcal{O}(\varepsilon)$ to safeguard against global rank-deficient regions, we empirically use a much larger $\beta = \mathcal{O}(1)$. This discrepancy stems from the favorable local geometry of the orthogonality constraints. In practice, iterates rapidly approach the Stiefel manifold where $\bX$ remains full rank and naturally satisfies the LICQ. This local LICQ provides an inherent $\mathcal{O}(1)$ strong concavity for the dual function, governing the algorithm's practical behavior and rendering the pessimistic global parameter $\varepsilon$ unnecessary.
\end{rmk}

We conduct each experiment 10 times, where the objective function and the initial point are randomly generated, and stop LSALM when  $\|\bX^{k} - \bX^{k-1}\|_F + \|\bX^{k} - \bZ^{k-1}\|_F \leq 10^{-3}$ and $\| (\bX^k)^\top \bX^k - \bI_n \|_F \leq 10^{-5}$ in each random experiment.
Figure~\ref{fig:nonsmoothQP-range} illustrates the tested parameter ranges. For each parameter, the first column indicates the interval where LSALM converged in all 10 experiments, while the second column highlights ranges where the average number of iterations was less than $110\%$ of the baseline algorithm's average. Our results confirm the algorithm's robustness, demonstrating convergence across a wide range of parameters. This also suggests that, despite potentially conservative theoretical assumptions, the algorithm performs effectively in practice, even when it doesn't strictly satisfy all assumptions from our convergence analysis. A similar phenomenon has also been observed in augmented Lagrangian methods for smooth problems on the Stiefel manifold \cite{gao2019parallelizable}. Consequently, we slightly extend the parameter selection beyond the theoretical requirements to achieve better empirical performance.

We further investigate the effect of the parameter $\beta$ on the convergence speed of LSALM. In Figure~\ref{fig:nonsmoothQP-beta-iter}, we plot the relationship between $\beta$ and the average number of iterations, using the same settings as in the previous experiment. The average is computed only over those instances in which the algorithm successfully converged in all 10 random trials for the given value of $\beta$. As observed, when convergence is achieved, a larger $\beta$ generally leads to faster convergence, as indicated by the reduced number of iterations. However, $\beta$ appears to have an upper bound beyond which LSALM may fail to converge. Specifically, when $\beta$ exceeds this threshold (empirically observed to be 0.49 in this experiment), the algorithm fails to converge in all instances.

\begin{figure}[htbp]
  \centering
  \begin{subfigure}{0.49\linewidth}
    \centering
    \includegraphics[width=\textwidth]{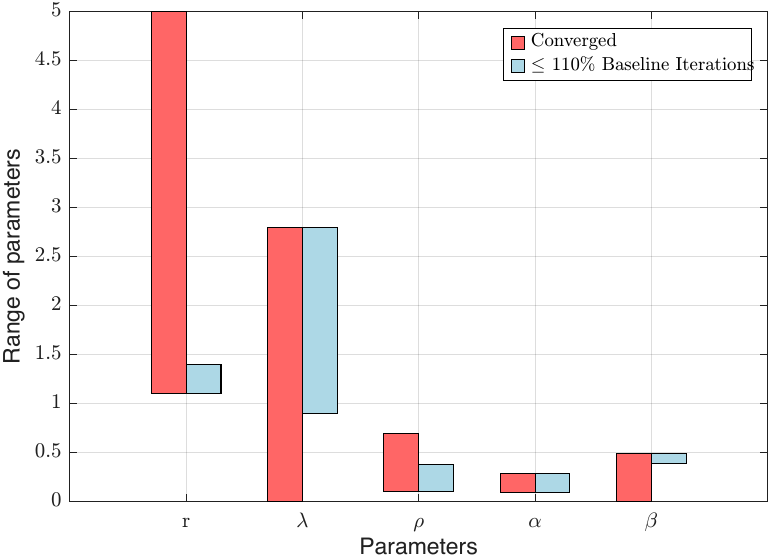}
  \caption{Feasible Range of Parameters}\label{fig:nonsmoothQP-range}
  \end{subfigure}
  \hfill
  \begin{subfigure}{0.49\linewidth}
    \centering
    \includegraphics[width=\textwidth]{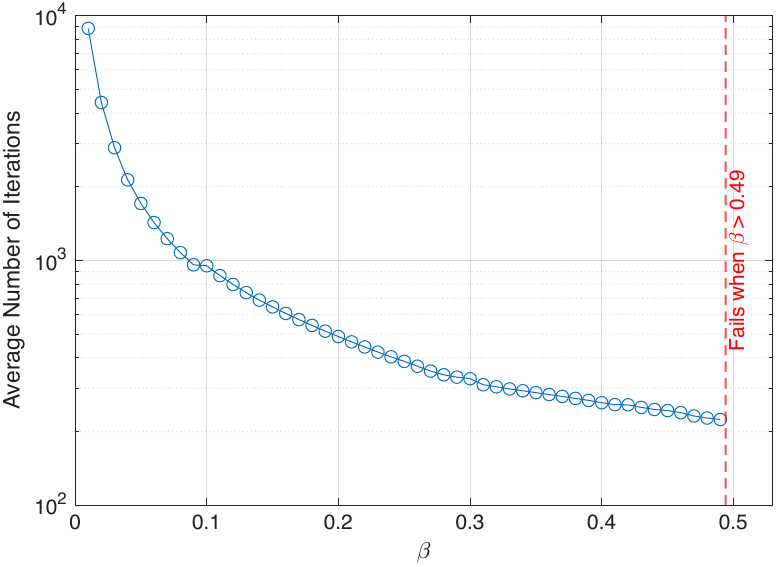}
  \caption{Effect of $\beta$ on the Average Number of Iterations}\label{fig:nonsmoothQP-beta-iter}
  \end{subfigure}
  \caption{Performance of LSALM in the Nonsmooth QP Experiment}
\end{figure}

We also demonstrate the effect of the parameter $\rho$ on the feasibility violation $\| (\bX^k)^\top \bX^k - \bI_n\|_F$, and the update of primal variables $\|\bX^k - \bX^{k-1}\|_F + \| \bX^k - \bZ^{k-1} \|_F$ in Figure~\ref{fig:nonsmoothQP-rho}. The experiment is conducted under the problem setting $(m,n, \mu) = (20, 20, 0.5)$. We vary $\rho \in \{ 0.1, 0.2, 0.5, 1, 50 \}$ while fixing the other algorithmic parameters as follows: $\lambda = 0.01$, $\varepsilon = 10^{-8}$, $R_{\bX}^{\operatorname{op}} = 10$, $R_{\bY} = 20$, $r = 30$, $\alpha = 0.1$, and $\beta = 0.3$. The figures show that the algorithm fails to converge when $\rho = 0.1$, while convergence becomes faster as $\rho$ increases. However, comparing the curves corresponding to $\rho = 1$ and $\rho = 50$, we observe that when $\rho$ is too large, although the feasibility violation decreases fast at start, the convergence of the algorithm becomes slower in later iterations. This indicates that an appropriately chosen $\rho$ is important for achieving efficient convergence in practice. Furthermore, Figure~\ref{fig:nonsmoothQP-feasibility} shows that the limiting feasibility violation attained by the algorithm is governed by the perturbation parameter $\varepsilon = 10^{-8}$.

\begin{figure}[htbp]
  \centering
  \begin{subfigure}{0.49\linewidth}
    \centering
    \includegraphics[width=\textwidth]{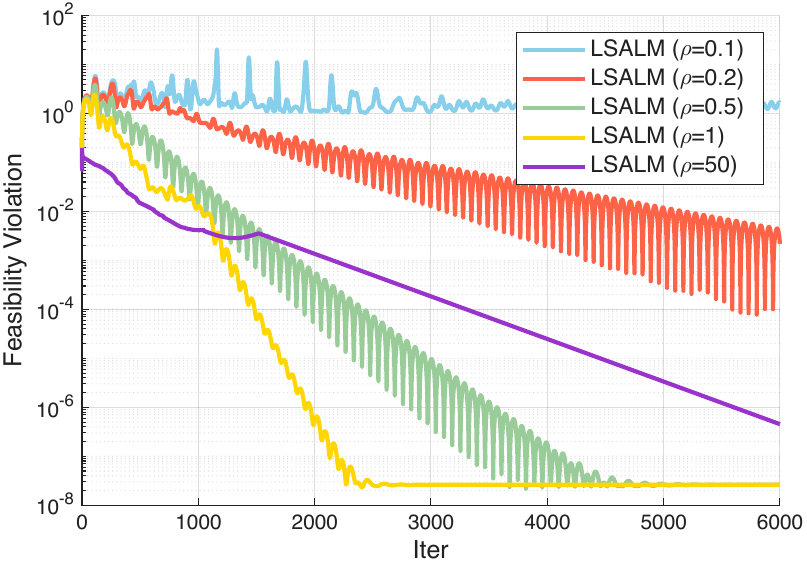}
  \caption{Feasibility Violation}\label{fig:nonsmoothQP-feasibility}
  \end{subfigure}
  \hfill
  \begin{subfigure}{0.49\linewidth}
    \centering
    \includegraphics[width=\textwidth]{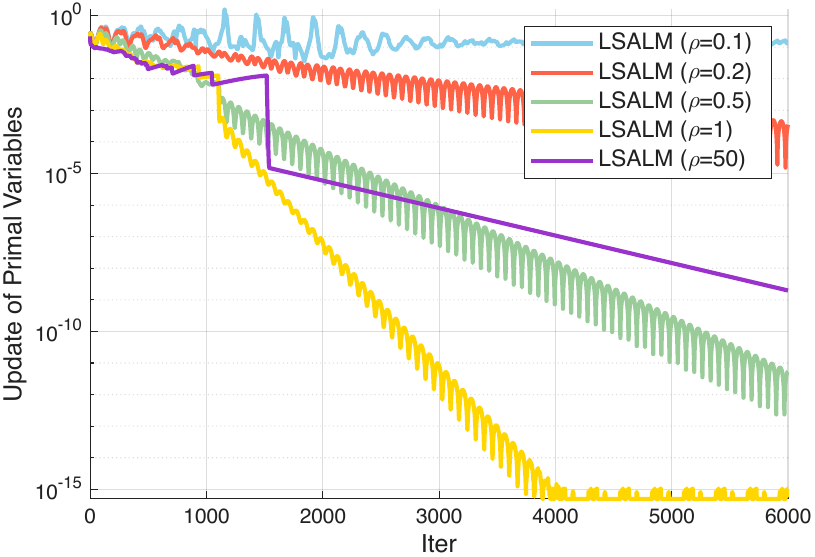}
  \caption{Update of Primal Variables}\label{fig:nonsmoothQP-primal-updates}
  \end{subfigure}
  \caption{Feasibility and Update of Primal Variables with Different $\rho$}\label{fig:nonsmoothQP-rho}
\end{figure}

We now investigate the relationship between the smoothing parameter $r$ and $\beta$ guided by our convergence theory. 
According to Lemma~\ref{prop:dual_eb_KL} and Proposition~\ref{prop:suff-decrease}, the theoretical upper bound for $\beta$ can be explicitly characterized as:
\begin{equation*}
    \beta \le \mathcal{O}\Bigg( \varepsilon\alpha \cdot \underbrace{\left( 1 + \frac{4(R_{\bX}^{\operatorname{op}})^2 \alpha}{r-\mu_{\rho}} \right)^{-2}}_{\text{pre-asymptotic penalty factor}} \Bigg).
\end{equation*}
While this bound becomes asymptotically independent of $r$ as $r \to \infty$ (converging to $\mathcal{O}(\varepsilon\alpha)$), the pre-asymptotic penalty factor plays a dominant role in the practical regime of moderate $r$ values. 
Specifically, as $r$ increases in this finite regime, the term $4(R_{\bX}^{\operatorname{op}})^2 \alpha / (r-\mu_{\rho})$ decreases rapidly. This significantly relaxes the penalty factor towards $1$, thereby expanding the allowable upper bound for $\beta$. To validate this theoretical relationship, we conduct a numerical experiment with problem setting $(m,n,\mu)=(5, 2, 1)$. The LSALM parameters were uniformly set as $\rho=0.1$, $\lambda=1$, $\varepsilon=10^{-8}$,  $\alpha=0.1$, $R_{\bX}^{\operatorname{op}}=10$, and $R_{\bY}=10$ uniformly. We then vary $r$ from $7$ to $66$ with a gap of $0.1$, and $\beta$ from $0.1$ to $0.25$ with a gap of $0.01$ simultaneously. Figure~\ref{fig:nonsmoothQP-r-beta-rel} visualizes the convergence results for each combination of $r$ and $\beta$ across 10 random instances. In Figure~\ref{fig:nonsmoothQP-converge-all}, green means LSALM converges in all 10 random instances, and blue means at least one failure to converge. Figure~\ref{fig:nonsmoothQP-converge-case} shows the number of convergent instance for each combination of $r$ and $\beta$.
The figure aligns with our theoretical upper bound of $\beta$. Note that we have verified in numerical experiment larger values of $\beta$ lead to faster convergence when LSALM converges. However, since $r$ is the proximal parameter, increasing $r$ generally slows down the algorithm.
Thus, to accelerate convergence, we need to find a balance between the values of $r$ and $\beta$.

\begin{figure}[h]
  \centering
  \begin{subfigure}{0.49\linewidth}
    \centering
    \includegraphics[width=\textwidth]{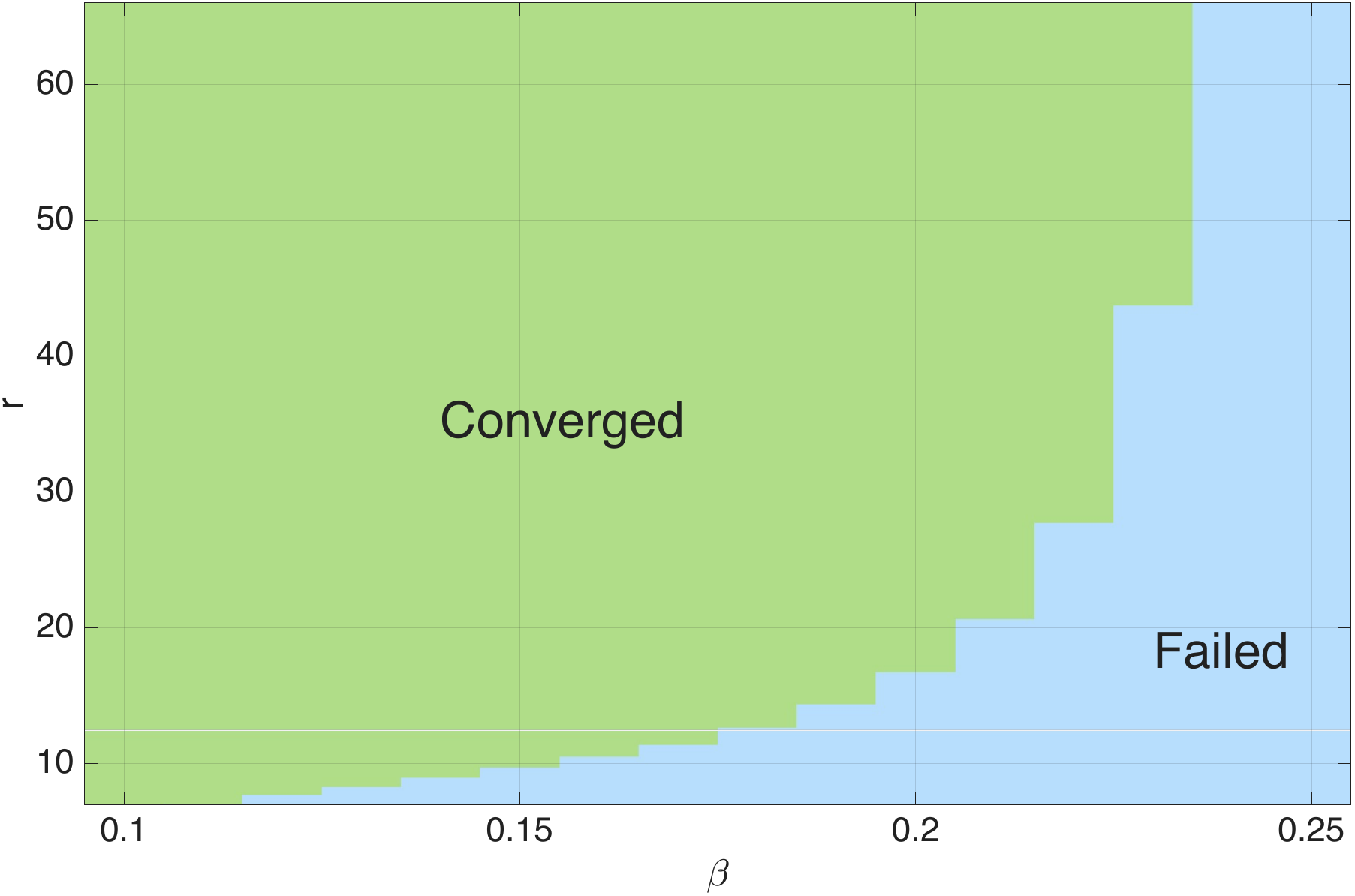}
  \caption{Region of $r$ and $\beta$ Convergent in All Instances}\label{fig:nonsmoothQP-converge-all}
  \end{subfigure}
  \hfill
  \begin{subfigure}{0.49\linewidth}
    \centering
    \includegraphics[width=\textwidth]{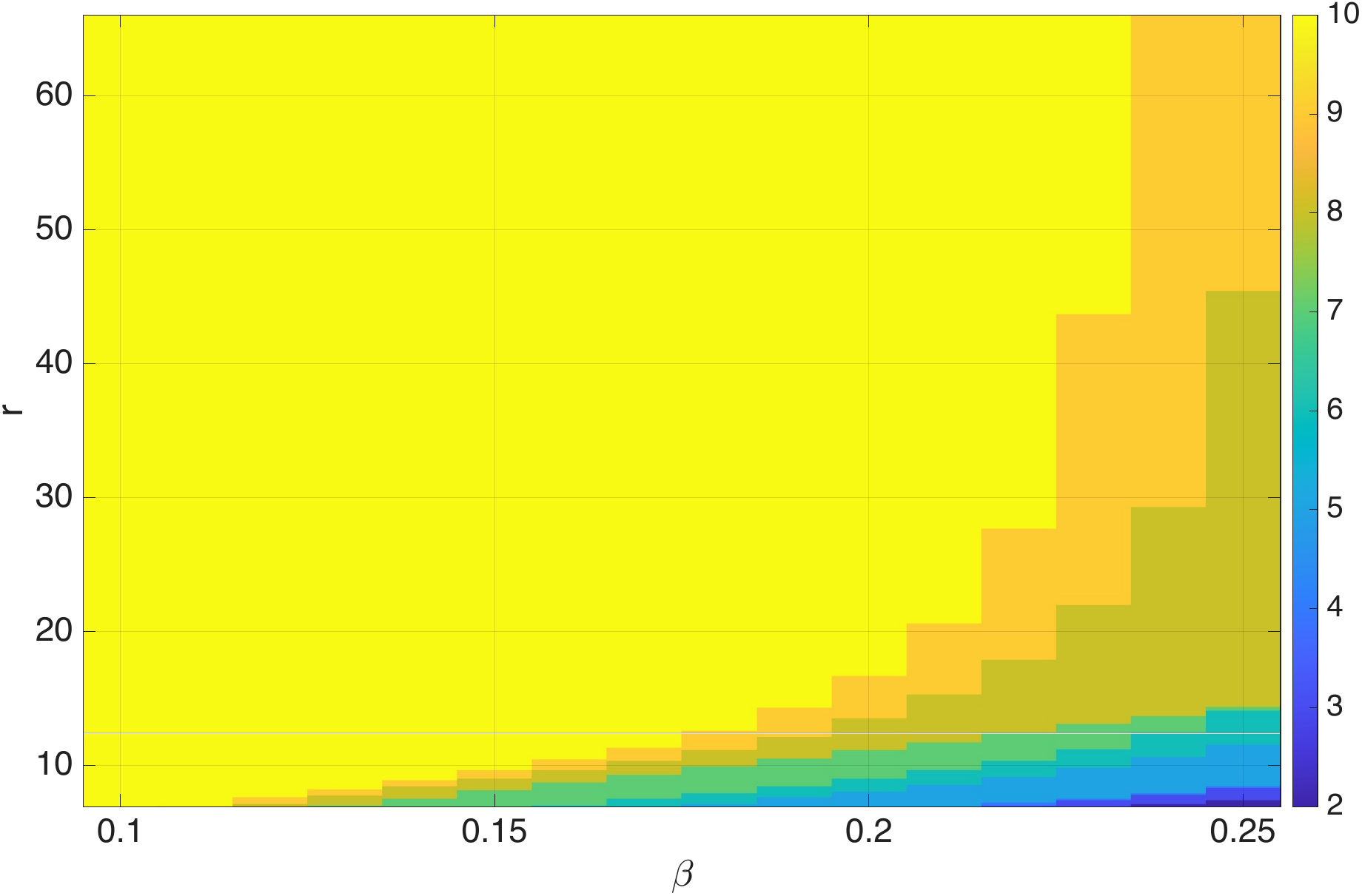}
  \caption{Number of Convergent Instances of LSALM}\label{fig:nonsmoothQP-converge-case}
  \end{subfigure}
  \caption{Convergence of LSALM with varying $r$ and $\beta$}\label{fig:nonsmoothQP-r-beta-rel}
\end{figure}

\subsection{Sparse Principal Component Analysis}
Principal Component Analysis (PCA) is a key method for analyzing high-dimensional data.
Sparse PCA is a variant of PCA which improves interpretability by finding principal components with very few non-zero entries.
For a given data matrix $\bm{A} \in \mathbb{R}^{p \times m}$, the goal of sparse PCA is to find the top $n$ sparse loading vectors, where $n < \min\{ p, m \}$. This problem is formulated as:
\begin{equation}\label{eq:sparse_PCA}
\begin{array}{c@{\quad}l}
\min\limits_{\bX\in\R^{m\times n}} &-\operatorname{tr}(\bX^{\top} \bm{A}^{\top} \bm{A} \bX)+\mu\|\bX\|_1\\
{\rm s.t.} &\bX^\top\bX=\bI_n.
\end{array}
\end{equation}
Here $\mu$ is a regularization parameter.
When $\mu=0$, this problem reduces to standard PCA, which involves finding the leading $n$ eigenvectors of $\bm{A}^{\top} \bm{A}$. When $\mu>0$, the $\ell_1$ norm $\|\bX\|_1$ encourages the loading vectors to be sparse. Solving \eqref{eq:sparse_PCA} is relatively simple when $n=1$.
However, for $n>1$, the problem is more complex because it requires enforcing both sparsity and orthogonality simultaneously. LSALM is designed to solve this more challenging case for larger values of $n$.

To evaluate the performance of LSALM, we compare it against the following algorithms: ManPG-Ada \cite{chen2020proximal}, PAMAL \cite{chen2016augmented}, SOC \cite{lai2014splitting}, and RADMM \cite{li2025riemannian}.
The experimental setup is as follows. We fix $p = 1000$ and $\mu = 0.5$ uniformly. The sparse PCA problem is solved across varying values of $m \in \{ 300, 400, 500, 600, 700, 800 \}$ with $n = m / 2$. The data matrix $\bm{A}$ is synthetically generated following the procedure in \cite{huang2022riemannian}. Specifically, $\bm{A}$ is constructed from five principal components, with each component repeated $p/5$ times (refer to \cite[Figure 4]{huang2022riemannian} for component details). Gaussian noise $\mathcal{N}(0,0.25)$ is then added to each entry of this matrix.
A common initial point is generated by projecting a randomly sampled matrix with standard Gaussian entries onto the Stiefel manifold, i.e. $\operatorname{proj}_{\St(m, n)}(\texttt{randn}(m, n))$ in MATLAB.

The parameter settings for each algorithm are as follows. For ManPG-Ada and PAMAL, we adopt the settings used in \cite[Section 6]{chen2020proximal}. For SOC, the penalty parameter is set to $\beta = 1.5 \cdot L_{\ell} $. For RADMM, we choose the smoothing parameter $\gamma = 10^{-12}$, the penalty parameter $ \rho = L_{\ell} $, and a fixed step size $\eta_k = \eta = 1 / (2L_{\ell}) $. The definitions of these parameters can be found in \cite{chen2020proximal,li2025riemannian}. For LSALM, we set $\rho = 10 $, $\alpha = \operatorname{round}(0.07 \cdot \sqrt{mn})$, $\beta = 0.5$, $\varepsilon = 10^{-10}$, $r = 15 $, $\lambda = 1 / L_{\ell}$, $R^{\operatorname{op}}_{\bX} = 10$, and $R_{\bY} = 10^3$.
All algorithms terminate when either the number of iterations reaches $30000$, or both the variable update condition $\| \bX^{k} - \bX^{k-1} \|_F \leq 10^{-4}$ and the respective constraint violation condition are satisfied:
\begin{equation}\label{eq:constraint-violation}
\begin{aligned}
&\text{SOC and PAMAL: } &&\frac{\| \bm{Q}^k - \bm{P}^k \|_F}{\max\{ 1, \| \bm{Q}^k \|_F, \| \bm{P}^k \|_F \}} + \frac{\| \bX^k - \bm{P}^k \|_F}{\max\{ 1, \| \bX^k \|_F, \| \bm{P}^k \|_F \}} \leq 10^{-4}, \\
&\text{RADMM: } &&\frac{\| \bX^k - \bZ^k \|_F}{\max\{ 1, \| \bX^k \|_F, \| \bZ^k \|_F \}} \leq 10^{-4}, \\
&\text{LSALM: } &&\| (\bX^k)^\top \bX^k - \bI_n \|_F \leq 10^{-4} .
\end{aligned}
\end{equation}

Each experiment is repeated 10 times, and all algorithms successfully converge across all test instances. For each algorithm, we report the following statistics: average CPU time (``T''), average number of iterations (``\#I''), average time per iteration (``T/I''), average final objective value (``Obj''), and average sparsity of the returned solution (``S''), as summarized in Table~\ref{table:sparse-pca-scaling}. Here, sparsity is defined as the proportion of entries in the solution whose absolute values are smaller than \( 10^{-5} \).
PAMAL is excluded from the table due to its significantly slower performance. For example, it requires an average CPU time of 2242 seconds for the case \( (m, n) = (800, 400) \). As shown in the table, our algorithm consistently outperforms the others in terms of CPU time, and the performance advantage becomes increasingly pronounced as the problem dimension grows.

Furthermore, Figure~\ref{fig:sparsepca-scaling} presents the relationship between \( m \) and both the average CPU time and the average time per iteration (in log scale), with fitted lines illustrating the growth trend. This figure provides a visual comparison of the practical scaling behavior of the algorithms with respect to \( m \). As shown, LSALM consistently demonstrates lower empirical complexity and outperforms the competing algorithms in both average CPU time and per-iteration efficiency, highlighting its superior scalability.

\begin{table}[htbp]\tiny
	\centering
\resizebox{1\textwidth}{!}{%
	\begin{tabular}{c|ccccc|ccccc|ccccc|ccccc}
		\hline
		\multicolumn{1}{c}{} & \multicolumn{5}{c}{ManPG-Ada} & \multicolumn{5}{c}{SOC}   & \multicolumn{5}{c}{RADMM} & \multicolumn{5}{c}{LSALM}  \bigstrut\\
      \hline
$m,n$&T(s)& \#I &  T/I(s) & Obj & S(\%) & T(s)&  \#I &  T/I(s) & Obj & S(\%)&T(s)& \#I &  T/I(s) & Obj & S(\%)&T(s)& \#I &  T/I(s) & Obj & S(\%) \bigstrut\\
      \hline
      300, 150 & 33.5  & 504 & 0.066 & -117.9 & 99.32 & 10.9  & 1694 & 0.006 & -118.8 & 99.31 & 4.4   & 1235 & 0.004 & -118.9 & 99.31 & 3.2  & 1101 & 0.003 & -118.7 & 99.32 \\
      400, 200 & 65.9  & 398 & 0.166 & -159.1 & 99.47 & 25.1  & 2287 & 0.011 & -160.0 & 99.46 & 9.7   & 1570 & 0.006 & -159.8 & 99.47 & 8.0  & 1480 & 0.005 & -159.7 & 99.47 \\
      500, 250 & 134.2 & 442 & 0.303 & -201.0 & 99.56 & 48.3  & 2885 & 0.017 & -201.5 & 99.55 & 18.4  & 1964 & 0.009 & -201.5 & 99.55 & 12.4 & 1562 & 0.008 & -200.9 & 99.56 \\
      600, 300 & 301.9 & 536 & 0.563 & -242.5 & 99.62 & 91.2  & 3249 & 0.028 & -242.9 & 99.62 & 43.7  & 2554 & 0.017 & -242.6 & 99.61 & 19.7 & 1508 & 0.013 & -242.9 & 99.62 \\
      700, 350 & 440.0 & 546 & 0.805 & -283.9 & 99.66 & 140.9 & 3713 & 0.038 & -286.4 & 99.66 & 59.6  & 2623 & 0.023 & -286.2 & 99.66 & 31.5 & 1783 & 0.018 & -284.7 & 99.66 \\
      800, 400 & 763.2 & 642 & 1.189 & -324.6 & 99.70 & 231.1 & 4560 & 0.051 & -324.8 & 99.70 & 106.0 & 3356 & 0.032 & -324.8 & 99.70 & 41.0 & 1731 & 0.024 & -325.0 & 99.70 \\
\hline
	\end{tabular}%
}
	\caption{Average Performance of the Algorithms for Different $(m, n)$}\label{table:sparse-pca-scaling}
\end{table}
\begin{figure}[htbp]
  \centering
  \begin{subfigure}{0.49\linewidth}
    \centering
    \includegraphics[width=\textwidth]{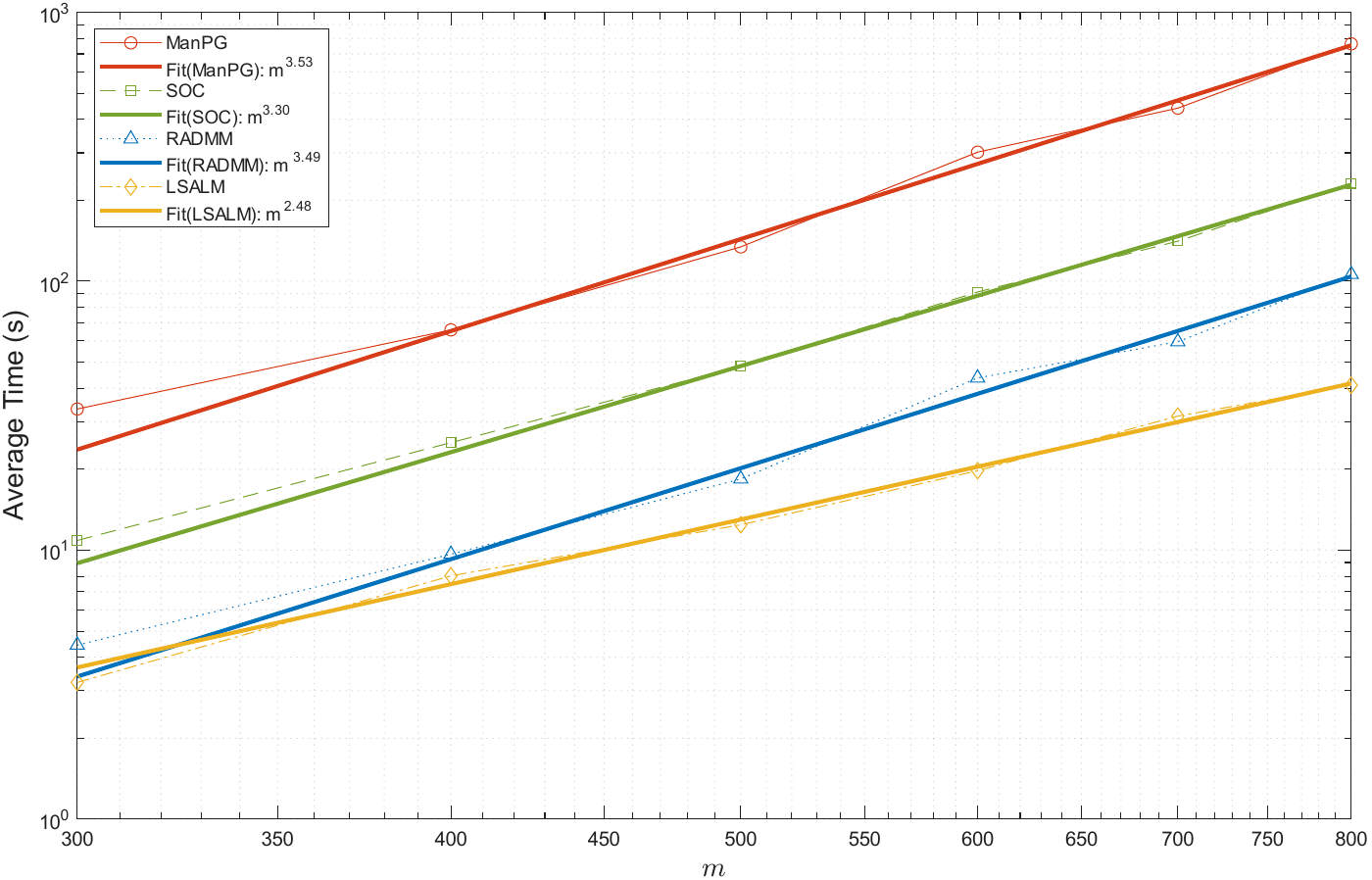}
    \caption{Average Time}
  \end{subfigure}
  \hfill
  \begin{subfigure}{0.49\linewidth}
    \centering
    \includegraphics[width=\textwidth]{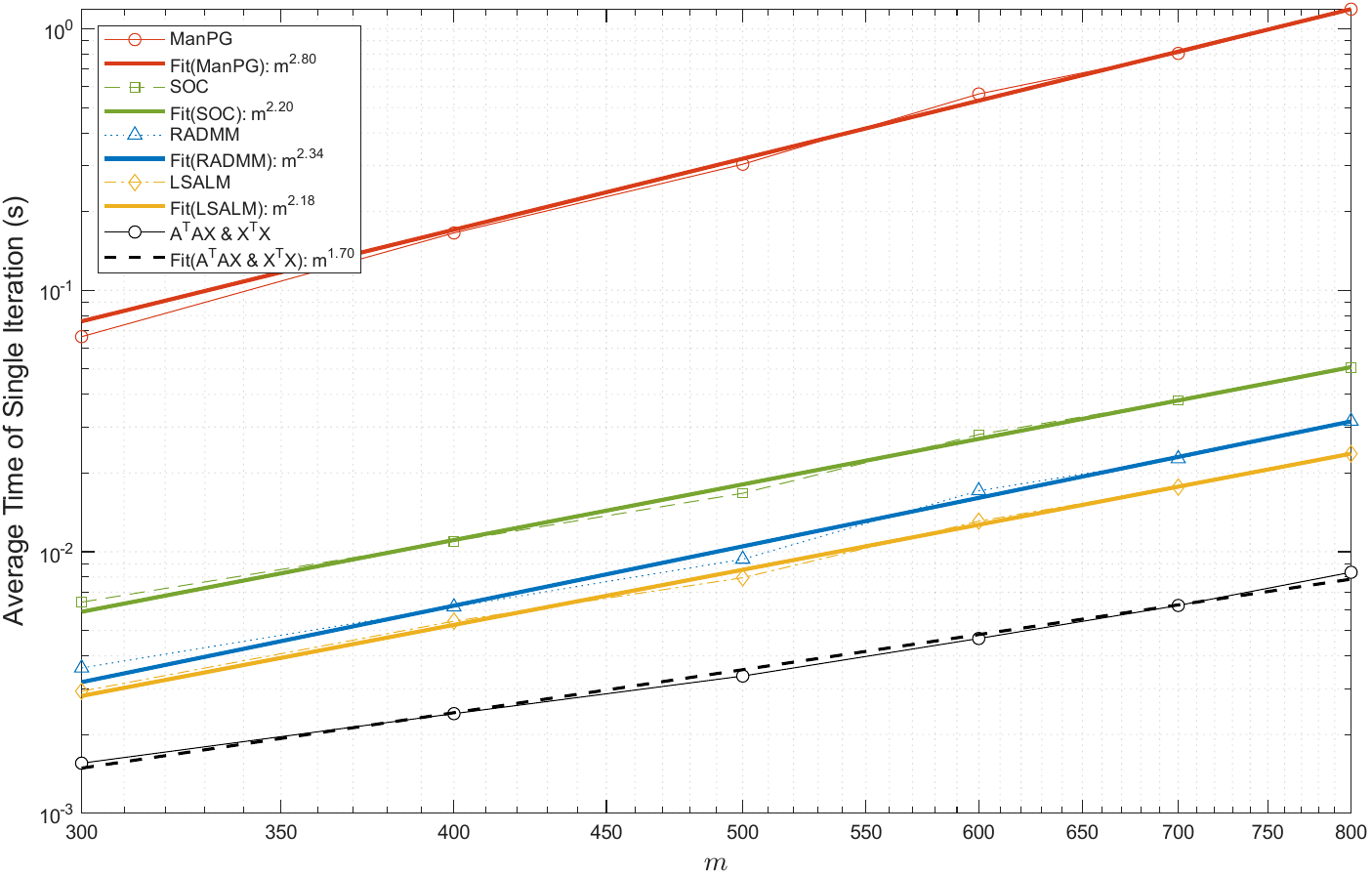}
    \caption{Average Time of Single Iteration}
  \end{subfigure}
  \caption{Comparison of Average CPU time and Per-Iteration Efficiency of the Algorithms}\label{fig:sparsepca-scaling}
\end{figure}

In the previous experiment, different algorithms often converged to different solutions, making it difficult to fairly compare their convergence speeds. To address this, we adopt a unified initialization strategy designed to encourage convergence to a common solution across all algorithms. Specifically, we adopt the initialization procedure proposed in \cite{chen2020proximal}. In each instance, we first generate a random point as in the last experiment and then run the Riemannian subgradient method (RSM) \cite{li2021weakly} for 250 iterations using a diminishing stepsize \( 1 / k^{3/4} \) at iteration \( k \). The resulting point is then used as the common starting point for all algorithms.

All other settings remain the same as in the previous experiment, except for the stopping criteria. In this experiment, ManPG-Ada is used as the baseline algorithm. It is run until its stopping criterion \( \| \bm{V}^k / t \|_F^2 \leq 10^{-8}mn \) is satisfied, yielding the baseline solution \( \bX_M \). The other algorithms are terminated when they satisfy both \( f(\bX^k) - f(\bX_M) \leq 10^{-4} \) and the corresponding constraint violation condition specified in \eqref{eq:constraint-violation}.

Each experiment is repeated 10 times. For instances where all algorithms successfully converge to the baseline solution, that is, the returned solution $\bX_r$ satisfies $\| \bX_r-\bX_M \|_F \leq 10^{-2}$ for every algorithm, we report the average performance metrics in Table~\ref{table:sparse-pca-scaling-local}.  The average final objective value and sparsity are excluded, as all algorithms converge to the baseline solution $\bX_M$ in these cases.  As shown in the table, our algorithm continues to outperform the competing methods in terms of convergence speed.

\begin{table}[htbp]\footnotesize
	\centering
	\begin{tabular}{c|ccc|ccc|ccc|ccc}
		\hline
		\multicolumn{1}{c}{} & \multicolumn{3}{c}{ManPG-Ada} & \multicolumn{3}{c}{SOC}   & \multicolumn{3}{c}{RADMM} & \multicolumn{3}{c}{LSALM}  \bigstrut\\
      \hline
      $m,n$    & T(s)  & \#I & T/I(s) & T(s)  & \#I  & T/I(s) & T(s) & \#I  & T/I(s) & T(s) & \#I  & T/I(s)  \bigstrut \\
      \hline
      300, 150 & 1.2  & 75  & 0.016 & 2.4  & 368  & 0.007 & 1.0  & 257 & 0.004 & 0.5  & 280 & 0.003 \\
      400, 200 & 3.0  & 102 & 0.030 & 6.5  & 588  & 0.011 & 2.5  & 400 & 0.007 & 1.3  & 266 & 0.005 \\
      500, 250 & 6.3  & 90  & 0.070 & 8.3  & 491  & 0.017 & 3.2  & 332 & 0.010 & 1.5  & 207 & 0.007 \\
      600, 300 & 30.0 & 105 & 0.284 & 17.1 & 606  & 0.028 & 7.2  & 410 & 0.018 & 3.1  & 248 & 0.012 \\
      700, 350 & 42.0 & 119 & 0.352 & 26.7 & 696  & 0.038 & 10.8 & 469 & 0.023 & 4.6  & 282 & 0.016 \\
      800, 400 & 86.8 & 195 & 0.445 & 71.2 & 1423 & 0.050 & 28.8 & 961 & 0.030 & 12.3 & 564 & 0.022 \\
\hline
	\end{tabular}%
	\caption{Average Performance of the Algorithms with RSM initialization for Different $(m, n)$}\label{table:sparse-pca-scaling-local}
\end{table}

\subsection{Graph Matching}
Although LSALM is primarily motivated by nonsmooth objective functions, we also investigate its performance on the graph matching problem. Our results will show that LSALM remains comparable even when dealing with smooth objective functions.

In the graph matching problem between a pair of graphs \((\mathcal{G}_1, \mathcal{G}_2)\), we set the node-affinity matrix \(\bm{K}^p = \bm{0}\). For each edge, we set the feature \(q_c\) as the distance between the two incident nodes. Then we define the edge-affinity matrix \(\bm{K}^q\) by \(\bm{K}_{c_i c_j}^q=\exp (-(q_{c_i}^1-q_{c_j}^2)^2/2500)\), which quantifies the similarity between the \(c_i\)th edge of \(\mathcal{G}_1\) and the \(c_j\)th edge of \(\mathcal{G}_2\). The graph matching problem is formulated as
\begin{equation}\label{eq:graph_matching}
\begin{array}{c@{\quad}l}
\max\limits_{\bX\in\R^{n\times n}} &\operatorname{vec}(\bX)^{\top} \bm{K} \operatorname{vec}(\bX)= \sum_{i=1}^n \bX_{:,i}^\top\bm{K}\bX_{:,i}\\
{\rm s.t.} &\bX^\top\bX=\bI_n,\ \bX\ge\bm{0},
\end{array}
\end{equation}
where the data matrix \(\bm{K}\) is non-negative. We solve the following penalized version:
\begin{equation}\label{eq:penalized_graph_mathcing}
\begin{array}{c@{\quad}l}
\min\limits_{\bX\in\R^{m\times n}} & \ell(\bX) :=-\sum_{i=1}^n \bX_{:,i}^\top\bm{K}\bX_{:,i}+\mu\|\max\{\bz,-\bX\}\|_F^2\\
{\rm s.t.} &\bX^\top\bX=\bI_n.
\end{array}
\end{equation}
The exact penalty properties are analyzed in \cite{qian2024error}. For our experiments, we use the CMU House dataset\footnote{The dataset is downloaded from \url{https://github.com/zhfe99/fgm}.} \cite{zhou2015factorized}, which contains 111 frames of a house, each annotated with 30 landmarks. Consequently, the graph matching problem has dimensions $m = n = 30$. Empirically, larger values of \(\mu\) degrade solution quality because the penalty term begins to dominate the object. Therefore, we use $\mu = 2$, which is relatively small. The drawback of a small \(\mu\) is that the resulting solution from \eqref{eq:penalized_graph_mathcing} may violate non-negativity constraints. To fix this issue, we employ a simple heuristic rounding scheme to obtain an assignment matrix. Specifically, suppose that we obtained a solution $\bX$ from \eqref{eq:penalized_graph_mathcing}. We first generate a matrix $\bX_{R}$, by setting $(\bX_R)_{ij^*} = 1 $  and $(\bX_R)_{ij} = 0$ for $j \neq j^*$ for each row \( i \), where $ j^* = \argmax_j \bX_{ij}$. In most cases, $\bX_{R}$ already satisfies the requirements. Then we use it in that case. In the rare cases where there exists an column $c_1$ has more than one 1 in $\bX_{R}$, we perform the following operations: Identify the two conflicting rows \(r_1\), \(r_2\), and a column $c_2$ with all 0. Then we reassign the entries, i.e. $(\bX_R)_{r_1, c_1} = 0$,  $(\bX_R)_{r_1, c_2} = 1$ or $(\bX_R)_{r_2, c_1} = 0$,  $(\bX_R)_{r_2, c_2} = 1$ according to a certain rule.

In the experiment, we first generate an initial point by $\proj_{\St(m,n)}(\texttt{randn(m,n)})$. After that, we run RGD for 10 iterations with a fixed step size of \(0.1\) starting from this point. Then we run all algorithms from this point. We use image No.1 to match image No.30, No.60, No.90, respectively. We repeat the algorithm for 10 times, and report the average performance in Table~\ref{tab:graph_matching}, where ``Obj'' means the average objective value, and ``F-mea'' denotes the average F-measure scores between the obtained solution and the ground truth among the 10 random experiments.

The parameters of the algorithms are set as follows: For the graph matching problem, we set RGD with trial step size \(\eta = 0.1 \), backtracking coefficient $\gamma = 0.5$, and sufficient decrease parameter \(\delta = 0.5\). The penalty parameter of PCAL is \( \beta = 20 \), with a fixed step size \(1 / \eta = 0.013\). The parameters of Landing are \(\lambda = 10\) with step size \(\eta = 0.015 \). For LSALM, the parameters are \(\alpha = 6 \), \(\beta = 0.2 \), \(\rho = 1 \), \(\varepsilon=10^{-9}\), \(r = 1\), \(\lambda = 0.025 \), \(R^{\operatorname{op}}_{\bX} = 10\), and \(R_{\bY} = 10^3\).
The meanings of the parameters for PCAL and Landing can be found in their respective papers \cite{gao2019parallelizable, ablin2022fast}.

The algorithms are stopped when the following criteria for stationarity are satisfied:
 \begin{align*}
   & \text{RGD:} && \|\grad \ell(\bX^k)\|_F \leq 10^{-4},\\
   & \text{PCAL and Landing:} && \|\nabla \ell(\bX^k) - \bX^k(\nabla \ell(\bX^k)^\top \bX^k + (\bX^k)^\top \nabla \ell(\bX^k)) / 2 \|_F \leq 10^{-4}\\
   & && \text{and } \|(\bX^k)^\top \bX^k - \bI_n\|_F \leq 10^{-6},\\
   & \text{LSALM:} && \|\nabla \ell(\bX^k) + 2\bX^k\bY^k\|_F \leq 10^{-4} \text{ and } \|(\bX^k)^\top \bX^k - \bI_n\|_F \leq 10^{-6}.\\
 \end{align*}

\begin{table}[htbp]\tiny
	\centering
    \resizebox{1\textwidth}{!}{%
	\begin{tabular}{c|cccc|cccc|cccc|cccc}
		\hline
		\multicolumn{1}{c}{} & \multicolumn{4}{c}{RGD} & \multicolumn{4}{c}{PCAL}   & \multicolumn{4}{c}{Landing} & \multicolumn{4}{c}{LSALM}  \bigstrut\\
		\hline
		 &Time(s)& \#Iter & Obj & F-mea & Time(s) & \#Iter & Obj & F-mea &Time(s)& \#Iter & Obj & F-mea &Time(s)& \#Iter & Obj & F-mea \bigstrut\\
      \hline
      No.30 & 1.37 & 1508 & -142.0 & 0.77 & 1.07 & 4849 & -142.0 & 0.77 & 0.90 & 4849 & -142.0 & 0.77 & 0.73 & 2929 & -142.0 & 0.77 \bigstrut[t] \\
      No.60 & 1.42 & 1595 & -135.8 & 0.82 & 1.12 & 5352 & -135.8 & 0.82 & 0.97 & 5351 & -135.8 & 0.82 & 0.79 & 3251 & -135.8 & 0.83              \\
      No.90 & 1.57 & 1747 & -131.8 & 0.84 & 1.33 & 6243 & -131.8 & 0.84 & 1.13 & 6220 & -131.8 & 0.84 & 0.83 & 3429 & -133.5 & 0.91 \bigstrut[b] \\
      \hline
	\end{tabular}}
	\caption{Average Results for Graph Matching}\label{tab:graph_matching}
\end{table}

\section{Closing Remarks}\label{sec:conclusion}

In this paper, we propose the primal-dual algorithm LSALM for solving nonsmooth and nonconvex optimization problems with orthogonality constraints. Unlike Riemannian optimization methods, which typically require retraction operations onto the Stiefel manifold involving complex matrix manipulations, our iterative scheme is simple and relies only on matrix multiplications. We establish both a competitive $\mathcal{O}(\epsilon^{-3})$ iteration complexity for finding $\epsilon$-KKT points and asymptotic convergence guarantees under mild conditions for our proposed method. Beyond effectively handling nonsmooth problems, LSALM also performs competitively in smooth settings compared to various state-of-the-art methods. Parallelizability is naturally aligned with the design of LSALM. Given a suitable nonsmooth separable structure, a parallel implementation can be further explored as an additional inherent advantage compared to the Riemannian framework. Moreover, the technique we employ to ensure feasibility in nonconvex constrained problems may be of independent interest and could potentially be extended to broader problems that exhibit favorable structure near the feasible region.

% \newpage
\bibliographystyle{alpha}
\bibliography{ref}

@article{zhang2020single,
  title={A single-loop smoothed gradient descent-ascent algorithm for nonconvex-concave min-max problems},
  author={Zhang, Jiawei and Xiao, Peijun and Sun, Ruoyu and Luo, Zhiquan},
  journal={Advances in Neural Information Processing Systems},
  volume={33},
  pages={7377--7389},
  year={2020}
}

@inproceedings{yang2022faster,
  title={Faster single-loop algorithms for minimax optimization without strong concavity},
  author={Yang, Junchi and Orvieto, Antonio and Lucchi, Aurelien and He, Niao},
  booktitle={International Conference on Artificial Intelligence and Statistics},
  pages={5485--5517},
  year={2022},
  organization={PMLR}
}

@article{zhang2020proximal,
  title={A proximal alternating direction method of multiplier for linearly constrained nonconvex minimization},
  author={Zhang, Jiawei and Luo, Zhi-Quan},
  journal={SIAM Journal on Optimization},
  volume={30},
  number={3},
  pages={2272--2302},
  year={2020},
  publisher={SIAM}
}

@article{attouch2013convergence,
  title={Convergence of descent methods for semi-algebraic and tame problems: {P}roximal algorithms, forward--backward splitting, and regularized {G}auss--{S}eidel methods},
  author={Attouch, Hedy and Bolte, J{\'e}r{\^o}me and Svaiter, Benar Fux},
  journal={Mathematical Programming},
  volume={137},
  number={1},
  pages={91--129},
  year={2013},
  publisher={Springer}
}

@article{bolte2017error,
  title={From error bounds to the complexity of first-order descent methods for convex functions},
  author={Bolte, J{\'e}r{\^o}me and Nguyen, Trong Phong and Peypouquet, Juan and Suter, Bruce W},
  journal={Mathematical Programming},
  volume={165},
  number={2},
  pages={471--507},
  year={2017},
  publisher={Springer}
}

@article{li2020understanding,
  title={Understanding notions of stationarity in nonsmooth optimization: {A} guided tour of various constructions of subdifferential for nonsmooth functions},
  author={Li, Jiajin and So, Anthony Man-Cho and Ma, Wing-Kin},
  journal={IEEE Signal Processing Magazine},
  volume={37},
  number={5},
  pages={18--31},
  year={2020},
  publisher={IEEE}
}

@article{qian2024error,
  title={Error bound and exact penalty method for optimization problems with nonnegative orthogonal constraint},
  author={Qian, Yitian and Pan, Shaohua and Xiao, Lianghai},
  journal={IMA Journal of Numerical Analysis},
  volume={44},
  number={1},
  pages={120--156},
  year={2024},
  publisher={Oxford University Press}
}

@article{chen2020proximal,
  title={Proximal gradient method for nonsmooth optimization over the {S}tiefel manifold},
  author={Chen, Shixiang and Ma, Shiqian and So, Anthony Man-Cho and Zhang, Tong},
  journal={SIAM Journal on Optimization},
  volume={30},
  number={1},
  pages={210--239},
  year={2020},
  publisher={SIAM}
}

@article{ablin2024infeasible,
  title={Infeasible deterministic, stochastic, and variance-reduction algorithms for optimization under orthogonality constraints},
  author={Ablin, Pierre and Vary, Simon and Gao, Bin and Absil, Pierre-Antoine},
  journal={Journal of Machine Learning Research},
  volume={25},
  number={389},
  pages={1--38},
  year={2024}
}

@inproceedings{ablin2022fast,
  title={Fast and accurate optimization on the orthogonal manifold without retraction},
  author={Ablin, Pierre and Peyr{\'e}, Gabriel},
  booktitle={International Conference on Artificial Intelligence and Statistics},
  pages={5636--5657},
  year={2022},
  organization={PMLR}
}

@article{gao2019parallelizable,
  title={Parallelizable algorithms for optimization problems with orthogonality constraints},
  author={Gao, Bin and Liu, Xin and Yuan, Ya-xiang},
  journal={SIAM Journal on Scientific Computing},
  volume={41},
  number={3},
  pages={A1949--A1983},
  year={2019},
  publisher={SIAM}
}

@article{hosseini2011generalized,
  title={Generalized gradients and characterization of epi-{L}ipschitz sets in {R}iemannian manifolds},
  author={Hosseini, Seyedehsomayeh and Pouryayevali, MR},
  journal={Nonlinear Analysis: Theory, Methods \& Applications},
  volume={74},
  number={12},
  pages={3884--3895},
  year={2011},
  publisher={Elsevier}
}

@article{yang2014optimality,
  title={Optimality conditions for the nonlinear programming problems on {R}iemannian manifolds},
  author={Yang, Wei Hong and Zhang, Lei-Hong and Song, Ruyi},
  journal={Pacific Journal of Optimization},
  volume={10},
  number={2},
  pages={415--434},
  year={2014}
}

@article{li2021weakly,
  title={Weakly convex optimization over {S}tiefel manifold using {R}iemannian subgradient-type methods},
  author={Li, Xiao and Chen, Shixiang and Deng, Zengde and Qu, Qing and Zhu, Zhihui and Man-Cho So, Anthony},
  journal={SIAM Journal on Optimization},
  volume={31},
  number={3},
  pages={1605--1634},
  year={2021},
  publisher={SIAM}
}

@article{chen2016augmented,
  title={An augmented {L}agrangian method for $\ell_1$-regularized optimization problems with orthogonality constraints},
  author={Chen, Weiqiang and Ji, Hui and You, Yanfei},
  journal={SIAM Journal on Scientific Computing},
  volume={38},
  number={4},
  pages={B570--B592},
  year={2016},
  publisher={SIAM}
}

@article{lai2014splitting,
  title={A splitting method for orthogonality constrained problems},
  author={Lai, Rongjie and Osher, Stanley},
  journal={Journal of Scientific Computing},
  volume={58},
  pages={431--449},
  year={2014},
  publisher={Springer}
}

@inproceedings{kovnatsky2016madmm,
  title={{MADMM}: {A} generic algorithm for non-smooth optimization on manifolds},
  author={Kovnatsky, Artiom and Glashoff, Klaus and Bronstein, Michael M},
  booktitle={Computer Vision--ECCV 2016: 14th European Conference, Amsterdam, The Netherlands, October 11-14, 2016, Proceedings, Part V 14},
  pages={680--696},
  year={2016},
  organization={Springer}
}

@article{li2025riemannian,
  title={A {R}iemannian alternating direction method of multipliers},
  author={Li, Jiaxiang and Ma, Shiqian and Srivastava, Tejes},
  journal={Mathematics of Operations Research},
  volume={50},
  number={4},
  pages={3222--3242},
  year={2025},
  publisher={INFORMS}
}

@article{bolte2014proximal,
  title={Proximal alternating linearized minimization for nonconvex and nonsmooth problems},
  author={Bolte, J{\'e}r{\^o}me and Sabach, Shoham and Teboulle, Marc},
  journal={Mathematical Programming},
  volume={146},
  number={1},
  pages={459--494},
  year={2014},
  publisher={Springer}
}

@article{huang2022riemannian,
  title={Riemannian proximal gradient methods},
  author={Huang, Wen and Wei, Ke},
  journal={Mathematical Programming},
  volume={194},
  number={1},
  pages={371--413},
  year={2022},
  publisher={Springer}
}

@article{deng2023manifold,
  title={A manifold inexact augmented {L}agrangian method for nonsmooth optimization on {R}iemannian submanifolds in {E}uclidean space},
  author={Deng, Kangkang and Peng, Zheng},
  journal={IMA Journal of Numerical Analysis},
  volume={43},
  number={3},
  pages={1653--1684},
  year={2023},
  publisher={Oxford University Press}
}

@article{zhou2023semismooth,
  title={A semismooth {N}ewton based augmented {L}agrangian method for nonsmooth optimization on matrix manifolds},
  author={Zhou, Yuhao and Bao, Chenglong and Ding, Chao and Zhu, Jun},
  journal={Mathematical Programming},
  volume={201},
  number={1},
  pages={1--61},
  year={2023},
  publisher={Springer}
}

@article{vandereycken2013low,
  title={Low-rank matrix completion by {R}iemannian optimization},
  author={Vandereycken, Bart},
  journal={SIAM Journal on Optimization},
  volume={23},
  number={2},
  pages={1214--1236},
  year={2013},
  publisher={SIAM}
}

@article{jolliffe2016principal,
  title={Principal component analysis: {A} review and recent developments},
  author={Jolliffe, Ian T and Cadima, Jorge},
  journal={Philosophical Transactions of the Royal Society A: Mathematical, Physical and Engineering Sciences},
  volume={374},
  number={2065},
  pages={20150202},
  year={2016},
  publisher={the Royal Society publishing}
}

@article{journee2010generalized,
  title={Generalized power method for sparse principal component analysis.},
  author={Journ{\'e}e, Michel and Nesterov, Yurii and Richt{\'a}rik, Peter and Sepulchre, Rodolphe},
  journal={Journal of Machine Learning Research},
  volume={11},
  number={2},
  year={2010}
}

@article{wang2023linear,
  title={Linear Convergence of a Proximal Alternating Minimization Method with Extrapolation for $\ell_1$-Norm Principal Component Analysis},
  author={Wang, Peng and Liu, Huikang and So, Anthony Man-Cho},
  journal={SIAM Journal on Optimization},
  volume={33},
  number={2},
  pages={684--712},
  year={2023},
  publisher={SIAM}
}

@article{liu2023unified,
  title={A unified approach to synchronization problems over subgroups of the orthogonal group},
  author={Liu, Huikang and Yue, Man-Chung and So, Anthony Man-Cho},
  journal={Applied and Computational Harmonic Analysis},
  volume={66},
  pages={320--372},
  year={2023},
  publisher={Elsevier}
}

@article{zhu2023rotation,
  title={Rotation group synchronization via quotient manifold},
  author={Zhu, Linglingzhi and Li, Chong and So, Anthony Man-Cho},
  journal={arXiv preprint arXiv:2306.12730},
  year={2023}
}

@article{ling2022near,
  title={Near-optimal performance bounds for orthogonal and permutation group synchronization via spectral methods},
  author={Ling, Shuyang},
  journal={Applied and Computational Harmonic Analysis},
  volume={60},
  pages={20--52},
  year={2022},
  publisher={Elsevier}
}

@article{cherian2016riemannian,
  title={Riemannian dictionary learning and sparse coding for positive definite matrices},
  author={Cherian, Anoop and Sra, Suvrit},
  journal={IEEE Transactions on Neural Networks and Learning Systems},
  volume={28},
  number={12},
  pages={2859--2871},
  year={2016},
  publisher={IEEE}
}

@article{sun2016completeI,
  title={Complete dictionary recovery over the sphere {I}: Overview and the geometric picture},
  author={Sun, Ju and Qu, Qing and Wright, John},
  journal={IEEE Transactions on Information Theory},
  volume={63},
  number={2},
  pages={853--884},
  year={2016},
  publisher={IEEE}
}

@article{keshavan10matrix,
  author  = {Raghunandan H. Keshavan and Andrea Montanari and Sewoong Oh},
  title   = {Matrix Completion from  Noisy Entries},
  journal = {Journal of Machine Learning Research},
  year    = {2010},
  volume  = {11},
  number  = {69},
  pages   = {2057-2078},
  url     = {http://jmlr.org/papers/v11/keshavan10a.html}
}

@book{absil2009optimization,
  title={Optimization Algorithms on Matrix Manifolds},
  author={Absil, P-A and Mahony, Robert and Sepulchre, Rodolphe},
  year={2009},
  publisher={Princeton University Press}
}

@book{boumal2023introduction,
  title={An Introduction to Optimization on Smooth Manifolds},
  author={Boumal, Nicolas},
  year={2023},
  publisher={Cambridge University Press}
}

@article{beck2023dynamic,
  title={A dynamic smoothing technique for a class of nonsmooth optimization problems on manifolds},
  author={Beck, Amir and Rosset, Israel},
  journal={SIAM Journal on Optimization},
  volume={33},
  number={3},
  pages={1473--1493},
  year={2023},
  publisher={SIAM}
}

@article{peng2023riemannian,
  title={Riemannian smoothing gradient type algorithms for nonsmooth optimization problem on compact {R}iemannian submanifold embedded in {E}uclidean space},
  author={Peng, Zheng and Wu, Weihe and Hu, Jiang and Deng, Kangkang},
  journal={Applied Mathematics \& Optimization},
  volume={88},
  number={3},
  pages={85},
  year={2023},
  publisher={Springer}
}

@book{bertsekas2014constrained,
  title={Constrained Optimization and Lagrange Multiplier Methods},
  author={Bertsekas, Dimitri P},
  year={2014},
  publisher={Academic Press}
}

@article{rockafellar1976augmented,
  title={Augmented {L}agrangians and applications of the proximal point algorithm in convex programming},
  author={Rockafellar, R Tyrrell},
  journal={Mathematics of Operations Research},
  volume={1},
  number={2},
  pages={97--116},
  year={1976},
  publisher={INFORMS}
}

@article{zhang2022global,
  title={A global dual error bound and its application to the analysis of linearly constrained nonconvex optimization},
  author={Zhang, Jiawei and Luo, Zhi-Quan},
  journal={SIAM Journal on Optimization},
  volume={32},
  number={3},
  pages={2319--2346},
  year={2022},
  publisher={SIAM}
}

@article{li2025nonsmooth,
  author  = {Li, Jiajin and Zhu, Linglingzhi and So, Anthony Man-Cho},
  title   = {Nonsmooth Nonconvex-Nonconcave Minimax Optimization: Primal-Dual Balancing and Iteration Complexity Analysis},
  journal = {Mathematical Programming},
  volume  = {214},
  number  = {1},
  pages   = {591--641},
  year    = {2025}
}

@inproceedings{lu2022single,
  title={A single-loop gradient descent and perturbed ascent algorithm for nonconvex functional constrained optimization},
  author={Lu, Songtao},
  booktitle={International Conference on Machine Learning},
  pages={14315--14357},
  year={2022},
  organization={PMLR}
}

@article{koshal2011multiuser,
  title={Multiuser optimization: Distributed algorithms and error analysis},
  author={Koshal, Jayash and Nedi{\'c}, Angelia and Shanbhag, Uday V},
  journal={SIAM Journal on Optimization},
  volume={21},
  number={3},
  pages={1046--1081},
  year={2011},
  publisher={SIAM}
}

@article{hajinezhad2019perturbed,
  title={Perturbed proximal primal--dual algorithm for nonconvex nonsmooth optimization},
  author={Hajinezhad, Davood and Hong, Mingyi},
  journal={Mathematical Programming},
  volume={176},
  number={1},
  pages={207--245},
  year={2019},
  publisher={Springer}
}

@article{bolte2007lojasiewicz,
  title={The {{\L}}ojasiewicz inequality for nonsmooth subanalytic functions with applications to subgradient dynamical systems},
  author={Bolte, J{\'e}r{\^o}me and Daniilidis, Aris and Lewis, Adrian},
  journal={SIAM Journal on Optimization},
  volume={17},
  number={4},
  pages={1205--1223},
  year={2007},
  publisher={SIAM}
}

@article{hu2024constraint,
  title={A constraint dissolving approach for nonsmooth optimization over the {S}tiefel manifold},
  author={Hu, Xiaoyin and Xiao, Nachuan and Liu, Xin and Toh, Kim-Chuan},
  journal={IMA Journal of Numerical Analysis},
  volume={44},
  number={6},
  pages={3717--3748},
  year={2024},
  publisher={Oxford University Press}
}

@article{zhou2015factorized,
  title={Factorized graph matching},
  author={Zhou, Feng and De la Torre, Fernando},
  journal={IEEE Transactions on Pattern Analysis and Machine Intelligence},
  volume={38},
  number={9},
  pages={1774--1789},
  year={2015},
  publisher={IEEE}
}

@inproceedings{saxe2014exact,
  title={Exact solutions to the nonlinear dynamics of learning in deep linear neural networks},
  author={Saxe, Andrew M. and McClelland, James L. and Ganguli, Surya},
  booktitle={International Conference on Learning Representations},
  year={2014}
}

@inproceedings{arjovsky2016unitary,
  title={Unitary evolution recurrent neural networks},
  author={Arjovsky, Martin and Shah, Amar and Bengio, Yoshua},
  booktitle={Proceedings of the 33rd International Conference on Machine Learning},
  pages={1120--1128},
  year={2016},
  organization={PMLR}
}

@article{kingma2018glow,
  title     = {Glow: Generative Flow with Invertible 1x1 Convolutions},
  author    = {Kingma, Diederik P. and Dhariwal, Prafulla},
  journal = {Advances in Neural Information Processing Systems},
  volume={31},
  pages     = {10236--10245},
  year      = {2018}
}

@article{xu2025oracle,
  title={On the oracle complexity of a Riemannian inexact augmented Lagrangian method for nonsmooth composite problems over Riemannian submanifolds: M. Xu et al.},
  author={Xu, Meng and Jiang, Bo and Liu, Ya-Feng and So, Anthony Man-Cho},
  journal={Optimization Letters},
  pages={1--19},
  year={2025},
  publisher={Springer}
}

@article{deng2025oracle,
  title={Oracle Complexities of Augmented Lagrangian Methods for Nonsmooth Composite Optimization on a Compact Submanifold},
  author={Deng, Kangkang and Hu, Jiang and Wu, Jiayuan and Wen, Zaiwen},
  journal={Mathematics of Operations Research},
  year={2025},
  publisher={INFORMS}
}
\appendix 
% \newpage

\section{Subdifferentials and Embedded Geometry}\label{sec:Riesubdiff}

Let $f:\mathbb{R}^{m\times n}\rightarrow \mathbb{R}$ and denote $f_{\mathcal{M}}:=f|_{\mathcal{M}}$ for simplicity. Suppose that $f_{\mathcal{M}}$ is differentiable, then the Riemannian gradient of $f_{\mathcal{M}}$ is a vector field $\operatorname{grad} f_{\mathcal{M}}$ as  the unique element  in $\operatorname{T}_{\bX}\mathcal{M}$ for any $\bX \in \mathcal{M}$ such that
\[
\langle\grad f_{\mathcal{M}}(\bX), \bm{\xi_{X}}\rangle_{\bX}=\operatorname{D} f_{\mathcal{M}}(\bm{X})\left[\bm{\xi_{X}}\right]\quad \text{for each}\ \bm{\xi_{X}} \in \operatorname{T}_{\bm{X}} \mathcal{M},
\]
where $\operatorname{D}f_{\mathcal{M}}$ is the differential of the function $f_{\mathcal{M}}$. For a smooth function $f$, we know that $\operatorname{grad} f_{\mathcal{M}}(\bX)=\operatorname{proj}_{\operatorname{T}_{\bX} \mathcal{M}} (\nabla f(\bX))$ by the given Riemannian metric. When $f_{\mathcal{M}}$ is not necessary smooth, we discuss the following directional derivative \cite{hosseini2011generalized} and subdifferentials.

\begin{defi}[Clarke Directional Derivative \& Subdifferential] For a locally Lipschitz function and lower semicontinuous function $f_{\mathcal{M}}$ on $\mathcal{M}$, the Riemannian Clarke directional derivative of $f_{\mathcal{M}}$ at $\bX \in \mathcal{M}$ in the direction $\bm{d}$ is defined by
$$
f_{\mathcal{M}}^{\circ}(\bX, \bm{d})=\limsup _{\bY \rightarrow \bX, t \rightarrow 0} \frac{f_{\mathcal{M}} \circ \varphi^{-1}(\varphi(\bY)+t \operatorname{D} \varphi(\bX)[\bm{d}])-f_{\mathcal{M}} \circ \varphi^{-1}(\varphi(\bY))}{t},
$$
where $(\varphi, U)$ is a coordinate chart at $\bX$. The Clarke subdifferential of $f_{\mathcal{M}}$ at $\bX \in \mathcal{M}$, denoted by $\partial_C f_{\mathcal{M}}(\bX)$, is given by
$$
\partial_C f_{\mathcal{M}}(\bX)=\left\{\bm{v} \in \mathrm{T}_{\bX} \mathcal{M}:\langle \bm{v}, \bm{d}\rangle \leq f_{\mathcal{M}}^{\circ}(\bX, \bm{d}),\ \text{for each}\ \bm{d} \in \mathrm{T}_{\bm{X}} \mathcal{M}\right\} .
$$
\end{defi}
It can be further characterized by the projection of Clarke subdifferential in the ambient space under certain regular condition holds.
\begin{fact}[{\cite[Theorem 5.1]{yang2014optimality}}]
\label{fact:proj-sub}
The inclusion $\partial_C f_{\mathcal{M}}(\bX)\subseteq\operatorname{proj}_{\operatorname{T}_{\bX} \mathcal{M}}(\partial_C f(\bX))$ holds. Suppose that for all $\bm{d} \in \mathrm{T}_{\bX} \mathcal{M}$,
\begin{equation}\label{eq:regular-euc}
f'(\bX,\bm{d}):=\lim _{t \rightarrow 0} \frac{f(\bX+t \bm{d})-f(\bX)}{t}=\limsup _{\bY \rightarrow \bX, t \rightarrow 0} \frac{f (\bY+t \bm{d})-f (\bY)}{t}=f_{\mathcal{M}}^{\circ}(\bX, \bm{d}).
\end{equation}
Then $\partial_C f_{\mathcal{M}}(\bX)=\operatorname{proj}_{\operatorname{T}_{\bX} \mathcal{M}}(\partial_C f(\bX))$.
\end{fact}
From \cite[Lemma 5.1]{yang2014optimality} (extended to the weakly convex case), we know that $f$ in problem \eqref{eq:problem_composite} satisfies \eqref{eq:regular-euc}, and consequently $\partial_C f_{\mathcal{M}}(\bX)=\operatorname{proj}_{\operatorname{T}_{\bX} \mathcal{M}}(\partial f(\bX))$ since $\partial f(\bX)=\partial f_C(\bX)$ from \cite[Fact 5]{li2020understanding}.

\section {Useful Technical Lemmas} \label{sec:lemmas}

We introduce several useful perturbation error bounds in this section. First, under the assumptions of problem~\eqref{eq:problem_composite}, we directly obtain the following results.

\begin{fact}\label{fact-2}
Let $\bY\in \mathcal{Y}$. For all $\bX, \bar{\bX} \in\mathcal{X}$ it follows that
\[
\frac{- L_{\rho} - \lambda^{-1}}{2}\|\bX-\bar{\bX}\|_F^{2} \leq \mathcal{L}_{\rho}(\bX,\bY)-\mathcal{L}_{\bar{\bX},\lambda}(\bX, \bY) \leq \frac{L_{\rho}-\lambda^{-1}}{2}\|\bX-\bar{\bX}\|_F^{2}.
\]
\end{fact}

The following Lipschitz error bound in Lemmas \ref{lemma-sollip} and \ref{lemma-dualdiff} are important in our analysis, and the proof of which is similar to that of Lemmas B.2 and B.3 in \cite{zhang2020single}, respectively. For the sake of completeness, we present the proof here.
\begin{lemma}\label{lemma-sollip}
For any $\bY, \bY^{\prime} \in \mathcal{Y}$ and $\bZ, \bZ^{\prime} \in \R^{m\times n}$, the following inequalities hold:
    \begin{align}
        &\|\bX(\bY, \bZ)-\bX(\bY, \bZ^{\prime})\|_F \leq \sigma_{1}\|\bZ-\bZ^{\prime}\|_F, \label{lip-z}\\
        &\|\bX(\bZ)-\bX(\bZ^{\prime})\|_F \leq \sigma_{1}\|\bZ-\bZ^{\prime}\|_F, \label{lip-starz}\\
        &\|\bX(\bY, \bZ)-\bX(\bY^{\prime}, \bZ)\|_F \leq \sigma_{2}\|\bY-\bY^{\prime}\|_F, \label{lip-y}
        \end{align}
     where $\sigma_{1}:=\frac{r}{r-\mu_{\rho}}$ and $\sigma_{2}=\frac{2R^{\operatorname{op}}_{\bX}}{r-\mu_{\rho}}$.
    \end{lemma}
\begin{proof}
The proof of \eqref{lip-z} and \eqref{lip-starz} can be found in \cite[Lemma 2]{li2025nonsmooth}. Now, we start to prove \eqref{lip-y}. By the $( r-\mu_{\rho} )$-strong convexity of $ F(\cdot, \bY, \bZ) $ and the definition of $ \bX(\cdot, \cdot) $, we have that
\begin{align}
&F(\bX(\bY^{\prime}, \bZ),\bY, \bZ)-F(\bX(\bY, \bZ), \bY,\bZ) \geq \frac{r- \mu_{\rho}}{2}\cdot\|\bX(\bY^{\prime}, \bZ)-\bX(\bY, \bZ)\|_F^{2},\label{psi-strongcvxineq1} \\
&F(\bX(\bY, \bZ), \bY^{\prime},\bZ)-F(\bX(\bY^{\prime}, \bZ), \bY^{\prime},\bZ) \geq \frac{r- \mu_{\rho}}{2}\cdot\|\bX(\bY, \bZ)-\bX(\bY^{\prime}, \bZ)\|_F^{2}\label{psi-strongcvxineq2}.
\end{align}
Moreover, by the  $\varepsilon$-strongly concavity of $F(\bX, \cdot, \bZ)$ we have
\begin{align}
& F(\bX(\bY, \bZ), \bY^{\prime},\bZ)-F(\bX(\bY, \bZ), \bY,\bZ)\leq \langle G(\bX(\bY, \bZ))-\varepsilon\bY, \bY^{\prime}-\bY\rangle-\frac{\varepsilon}{2}\|\bY-\bY^\prime\|_F^2, \label{psi-yconcave}\\
& F(\bX(\bY^{\prime}, \bZ), \bY, \bZ)-F(\bX(\bY^{\prime}, \bZ), \bY^{\prime},\bZ) \leq \langle G(\bX(\bY^{\prime}, \bZ))-\varepsilon\bY^\prime, \bY-\bY^{\prime}\rangle-\frac{\varepsilon}{2}\|\bY-\bY^\prime\|_F^2. \label{psi-ylip}
\end{align}
Combining \eqref{psi-strongcvxineq1}-\eqref{psi-ylip} it follows that
\begin{align*}
&(r-\mu_{\rho})\|\bX(\bY, \bZ)-\bX(\bY^{\prime}, \bZ)\|_F^{2}
\leq\langle G(\bX(\bY, \bZ))-G(\bX(\bY^{\prime}, \bZ)), \bY^{\prime}-\bY\rangle-\varepsilon\|\bY-\bY^\prime\|_F^2
\end{align*}
This together with the consequence of the 2$R^{\operatorname{op}}_{\bX}$-Lipschitz continuity of $G(\cdot)$ on $\mathcal{X}$ implies that
\begin{align*}
&(r-\mu_{\rho})\|\bX(\bY, \bZ)-\bX(\bY^{\prime}, \bZ)\|_F^{2}\leq  2R^{\operatorname{op}}_{\bX}\|\bX(\bY^{\prime}, \bZ)-\bX(\bY, \bZ)\|_F\|\bY-\bY^{\prime}\|_F-\varepsilon\|\bY-\bY^\prime\|_F^2.
\end{align*}
Let $\zeta:=\frac{\|\bX(\bY^{\prime}, \bZ)-\bX(\bY, \bZ)\|_F}{\|\bY-\bY^{\prime}\|_F}$. 
Then we know that
\begin{align*}
\zeta^{2}
& \ \leq \frac{2R^{\operatorname{op}}_{\bX}}{r-\mu_{\rho}} \zeta-\frac{\varepsilon}{r-\mu_{\rho}}\leq \frac{1}{2}\zeta^2+\frac{4(R^{\operatorname{op}}_{\bX})^2}{2(r-\mu_{\rho})^2} -\frac{\varepsilon}{r-\mu_{\rho}}\\
& \ \leq \frac{1}{2} \zeta^{2}+\frac{4(R^{\operatorname{op}}_{\bX})^2-2\varepsilon(r-\mu_{\rho})}{2(r-\mu_{\rho})^2}  \leq \frac{1}{2} \zeta^{2}+\frac{
% (2R^{\operatorname{op}}_{\bX}+2(r-\mu_{\rho}))^2
 4(R^{\operatorname{op}}_{\bX})^2
}{2(r-\mu_{\rho})^2},
\end{align*}
where the second inequality is due to the basic inequality $ab\leq \frac{1}{2}(a^2+b^2)$ for $a,b\in\R$.
Thus,
\[
\|\bX(\bY, \bZ)-\bX(\bY^{\prime}, \bZ)\|_F \leq  \frac{2R^{\operatorname{op}}_{\bX}}{r-\mu_{\rho}}\cdot\|\bY-\bY^{\prime}\|_F,
\]
which shows that \eqref{lip-y} holds with Lipschitz constant $\sigma_{2}=\frac{2R^{\operatorname{op}}_{\bX}}{r-\mu_{\rho}}$. The proof is complete.
\end{proof}
\begin{lemma}\label{lemma-dualdiff} 
    The dual function $d(\cdot,\cdot)$ is differentiable on $\mathcal{Y}\times \R^{m\times n}$, and for each  $\bY\in\mathcal{Y}$, $\bZ\in\R^{m\times n}$
    \begin{align*}
    &\nabla_{\bY} d(\bY,\bZ)=\nabla_{\bY} F(\bX(\bY,\bZ),\bY,\bZ)=G(\bX(\bY,\bZ))-\varepsilon\bY,\\
    &\nabla_{\bZ} d(\bY,\bZ)=\nabla_{\bZ} F(\bX(\bY, \bZ),\bY, \bZ)=r(\bZ-\bX(\bY,\bZ)).
    \end{align*}
    Moreover, $\nabla d(\cdot,\cdot)$ is Lipschitz continuous, i.e.,
    \begin{align}
    &\|\nabla_{\bY} d(\bY^{\prime},\bZ)-\nabla_{\bY} d(\bY^{\prime \prime},\bZ)\|_F \leq L_{d}\|\bY^{\prime}-\bY^{\prime \prime}\|_F, \quad \text{for all}\ \bY^{\prime}, \bY^{\prime \prime} \in \mathcal{Y},\notag\\
    &\|\nabla_{\bZ} d(\bY,\bZ')-\nabla_{\bZ} d(\bY,\bZ'')\|_F \leq L_{d}'\|\bZ^{\prime}-\bZ^{\prime \prime}\|_F, \quad \text{for all}\ \bZ^{\prime}, \bZ^{\prime \prime} \in \R^{m\times n},\notag
    \end{align}
    with $L_{d}:=2\sigma_2R^{\operatorname{op}}_{\bX}$ and $L_{d}':=(\sigma_1+1)r$. 
    \end{lemma}

We further establish the smoothness of the function $p$.
     \begin{lemma}\label{lem-proximal-smooth}
     The function $ p(\cdot) $ is differentiable on $\R^{m \times n}$ and 
     \[
     \nabla p(\bZ) = \nabla_{\bZ}d(\bY(\bZ), \bZ) = r(\bZ - \bX(\bZ)).
     \]   
Moreover, $\nabla p(\bZ)$ is $ L_d' $-Lipschitz continuous.
    \end{lemma}

    \begin{proof}
      By definition, we know $p(\bZ) = \min_{\bX \in \mathcal{X}} \max_{\bY \in \mathcal{Y}}F(\bX, \bY, \bZ)$. Since $ F(\cdot, \bY, \bZ) $ is strongly convex  for all $\bY$ and $\bZ$ with uniform modules, the function $\max_{\bY \in \mathcal{Y}} F(\cdot, \bY, \bZ) $ also retains strong convexity and has a unique minimizer. By Danskin's theorem, this implies that $p(\bZ)$ is differentiable.
      Furthermore, noting that $ p(\bZ) = \max_{\bY \in \mathcal{Y}}d(\bY, \bZ) $, we have that $\partial p(\bZ) = \operatorname{conv}\{\nabla_{\bZ} d(\bY(\bZ), \bZ)\} = \{r(\bZ - \bX(\bZ))\}$. Consequently, $\nabla p(\bZ) = r(\bZ - \bX(\bZ))$, and the Lipschitz continuity of $\nabla p(\bZ)$ follows from~\eqref{lip-starz}.
    \end{proof}

\section{Proof Details of Basic Descent Property}
\label{sec:basic_decrease}

In this part, we present the proof of the basic descent inequality~\eqref{eq:basicdes}.

\begin{lemma}[Primal descent]
\label{lemma-F-des} 
For any $k\ge0$, it follows that
\begin{align*}
& F(\bX^{k}, \bY^{k},\bZ^k)-F(\bX^{k+1}, \bY^{k+1},\bZ^{k+1})\\
\ge\ &
\frac{2\lambda^{-1}+r-L_{\rho} - L_g}{2}\|\bX^k-\bX^{k+1}\|_F^2 + \langle G(\bX^{k+1})-\varepsilon\bY^k, \bY^{k}-\bY^{k+1}\rangle-\frac{\varepsilon}{2}\|\bY^k-\bY^{k+1}\|_F^2\\
&+ \frac{(2-\beta)r}{2 \beta}\|\bZ^{k}-\bZ^{k+1}\|_F^{2}.
\end{align*}
\end{lemma}

\begin{proof}
One can infer from the definition that $\mathcal{L}_{\bX^k,\lambda}(\cdot,\bY)$ is $(\lambda^{-1}-L_g)$-strongly convex. From the update of $\bX^{k+1}$ in~\eqref{eq:primal_update}, we have
\begin{equation*}
\begin{aligned}
F(\bX^{k}, \bY^{k},\bZ^{k}) & =\mathcal{L}_{\rho}(\bX^{k}, \bY^{k})-\frac{\varepsilon}{2}\|\bY^k\|_F^2+ \frac{r}{2}\|\bX^k-\bZ^{k}\|_F^2\notag \\
&= \mathcal{L}_{\bX^k,\lambda}(\bX^k,\bY^k) -\frac{\varepsilon}{2}\|\bY^k\|_F^2+ \frac{r}{2}\|\bX^k-\bZ^{k}\|_F^2 \\
& \ge  \mathcal{L}_{\bX^k,\lambda}(\bX^{k+1},\bY^k) -\frac{\varepsilon}{2}\|\bY^k\|_F^2+ \frac{r}{2}\|\bX^{k+1}-\bZ^{k}\|_F^2+\frac{\lambda^{-1}+r-L_g }{2}\|\bX^k-\bX^{k+1}\|_F^2.
\end{aligned}
\end{equation*}
Moreover, Fact \ref{fact-2} implies that
\begin{equation*}
\mathcal{L}_{\bX^k,\lambda}(\bX^{k+1},\bY^k)\ge \mathcal{L}_{\rho}(\bX^{k+1},\bY^k) + \frac{\lambda^{-1}-L_{\rho}}{2}\|\bX^{k+1}-\bX^k\|_F^2.
\end{equation*}
It follows that
\begin{align}\label{primaldec-key1}
F(\bX^{k}, \bY^{k},\bZ^{k}) &\ge \mathcal{L}_{\rho}(\bX^{k+1},\bY^k)-\frac{\varepsilon}{2}\|\bY^k\|_F^2+\frac{r}{2}\|\bX^{k+1}-\bZ^k\|_F^2\notag\\
&\quad+\frac{2\lambda^{-1}+r-L_{\rho} - L_g}{2}\|\bX^k-\bX^{k+1}\|_F^2\notag\\
&= F(\bX^{k+1},\bY^{k}, \bZ^{k})+\frac{2\lambda^{-1}+r-L_{\rho} - L_g}{2}\|\bX^k-\bX^{k+1}\|_F^2.
\end{align}

Next,  as $F(\bX,\cdot,\bZ)$ is $\varepsilon$-strong concave, we have
    \begin{equation}\label{primaldec-key2}
      F(\bX^{k+1}, \bY^{k},\bZ^k)-F(\bX^{k+1}, \bY^{k+1},\bZ^k)
      \geq \langle G(\bX^{k+1})-\varepsilon\bY^k, \bY^{k}-\bY^{k+1}\rangle + \frac{\varepsilon}{2}\|\bY^k-\bY^{k+1}\|_F^2.
    \end{equation}
At last, on top of the update of variable $\bZ^{k+1}$, i.e., $\bZ^{k+1}=\bZ^{k}+\beta(\bX^{k+1}-\bZ^{k})$, we can verify
\begin{equation}\label{primaldec-key3}
F(\bX^{k+1}, \bY^{k+1},\bZ^{k})-F(\bX^{k+1},\bY^{k+1}, \bZ^{k+1})  = \frac{(2-\beta)r}{2 \beta}\|\bZ^{k}-\bZ^{k+1}\|_F^{2}.
\end{equation}
Summing up \eqref{primaldec-key1}, \eqref{primaldec-key2} and \eqref{primaldec-key3}, the desired result is obtained.
\end{proof}

\begin{lemma}[Dual ascent]\label{lemma-d-aes} 
For any $k\ge0$, it follows that
 \begin{align}\label{dual-ascent}
    d(\bY^{k+1},\bZ^{k+1})-d(\bY^{k},\bZ^k)
    &\ge\langle G(\bX(\bY^{k},\bZ^k))-\varepsilon\bY^k, \bY^{k+1}-\bY^{k}\rangle
    -\frac{L_d}{2}\|\bY^k-\bY^{k+1}\|_F^2\notag\\
    &\quad+\frac{r}{2}\left\langle \bZ^{k+1}-\bZ^{k},\bZ^{k+1}+\bZ^{k}-2 \bX(\bY^{k+1}, \bZ^{k+1})\right\rangle
    \end{align}
\end{lemma}
 \begin{proof}
 From Lemma \ref{lemma-dualdiff} we know $\nabla_{\bY} d(\cdot,\bZ)$ is Lipschitz continuous
     with $L_{d}$. 
     Then 
     we know that
         \begin{align*}
         &d(\bY^{k+1},\bZ^k)-d(\bY^{k},\bZ^k) 
         \ge \langle G(\bX(\bY^{k},\bZ^k))-\varepsilon\bY^k, \bY^{k+1}-\bY^{k}\rangle
         -\frac{L_d}{2}\|\bY^k-\bY^{k+1}\|_F^2.
         \end{align*}
 On the other hand, one has that
         \begin{align*}
         &d(\bY^{k+1}, \bZ^{k+1})-d(\bY^{k+1}, \bZ^{k})\notag\\ 
         =\ & F(\bX(\bY^{k+1}, \bZ^{k+1}), \bY^{k+1}, \bZ^{k+1})-F(\bX(\bY^{k+1}, \bZ^{k}), \bY^{k+1},\bZ^{k}) \\
         \geq\ &  F(\bX(\bY^{k+1}, \bZ^{k+1}), \bY^{k+1},\bZ^{k+1})-F(\bX(\bY^{k+1}, \bZ^{k+1}), \bY^{k+1},\bZ^{k}) \\
         =\ & \frac{r}{2}\|\bX(\bY^{k+1}, \bZ^{k+1})-\bZ^{k+1}\|_F^{2}-\frac{r}{2}\|\bX(\bY^{k+1}, \bZ^{k+1})-\bZ^{k}\|_F^{2} \\
         =\ & \frac{r}{2} \left\langle \bZ^{k+1}-\bZ^{k},\bZ^{k+1}+\bZ^{k}-2 \bX(\bY^{k+1}, \bZ^{k+1})\right\rangle.
         \end{align*}
         Finally, combining above inequalities we know \eqref{dual-ascent} holds.
 \end{proof}
 
\begin{lemma}[Proximal descent]
\label{lemma-p-des} 
    For any $k\ge0$, it follows that
   \begin{equation}\label{proximal-descent}
    p(\bZ^k)-p(\bZ^{k+1}) \ge \frac{r}{2}\left\langle \bZ^{k+1}-\bZ^{k},2 \bX(\bY(\bZ^{k+1}), \bZ^{k})-\bZ^{k}-\bZ^{k+1}\right\rangle.
    \end{equation}
    \end{lemma}
 \begin{proof}
 Recall the definition of $p(\bZ)$:
     \[
     p(\bZ)=\max _{\bY \in \mathcal{Y}} d(\bY, \bZ).
     \]
     Then 
     it follows from the definition of $\bY(\bZ^{k+1})=\mathop{\argmax}\limits_{\bY\in\mathcal{Y}}d(\bY,\bZ^{k+1})$ that
     \begin{align*}
      &p(\bZ^{k+1})-p(\bZ^{k})\notag\\
     \leq\ & d(\bY(\bZ^{k+1}), \bZ^{k+1})-d(\bY(\bZ^{k+1}), \bZ^{k}) \\
     \leq\ & F(\bX(\bY(\bZ^{k+1}), \bZ^{k}), \bY(\bZ^{k+1}), \bZ^{k+1})-F(\bX(\bY(\bZ^{k+1}), \bZ^{k}), \bY(\bZ^{k+1}),\bZ^{k}) \\
     =\ & \frac{r}{2}\left\langle \bZ^{k+1}-\bZ^{k},\bZ^{k+1}+\bZ^{k}-2 \bX(\bY(\bZ^{k+1}), \bZ^{k})\right\rangle,
     \end{align*}
     where the second inequality is from that   $F(\bX', \bY, \bZ)\ge\min_{\bX\in\mathcal{X}} F(\bX,\bY,\bZ) = d(\bY, \bZ)$ holds for any $\bX'\in\mathcal{X}$. The proof is complete.
     \end{proof}

The following basic descent property is derived by combining the sufficient descent, dual ascent, and proximal descent properties established above.

\begin{prop}[Basic descent property]
\label{prop:decrease}
    Let $r\ge \max\{ L_{\rho} + L_{g} + 4R^{\operatorname{op}}_{\bX},  3(L_{\rho}+L_g) \} $, $\lambda\le \frac{1}{2R^{\operatorname{op}}_{\bX}}$, $ \alpha \leq \min\left\{\frac{1}{20R^{\operatorname{op}}_{\bX}}, \frac{1}{ 8R^{\operatorname{op}}_{\bX}\zeta^2} \right\}$, $\beta  \leq \frac{1}{28}\min\left\{1,\frac{(r-\mu_{\rho})^2}{2\alpha r(R^{\operatorname{op}}_{\bX})^2}\right\}$. Then for any $k\ge0$,
\begin{equation*}
\begin{aligned}
\Phi^k-\Phi^{k+1}\ge\ & \frac{7}{16\lambda}\|\bX^{k}-\bX^{k+1}\|_F^{2}+\frac{1}{8\alpha}\|\bY^{k}-\bY_+^{k}(\bZ^k)\|_F^2+\frac{4r}{7\beta}\|\bZ^{k}-\bZ^{k+1}\|_F^{2}\\
&-28r\beta\|\bX(\bY(\bZ^{k}),\bZ^{k})-\bX(\bY_+^{k}(\bZ^k), \bZ^{k})\|_F^{2}.
\end{aligned}
\end{equation*}
\end{prop}

\begin{proof}
        From Lemmas \ref{lemma-F-des}, \ref{lemma-d-aes}, \ref{lemma-p-des} we know that
        \begin{align}\label{desascent-key1}
        & \Phi(\bX^k,\bY^k,\bZ^k)-\Phi(\bX^{k+1},\bY^{k+1},\bZ^{k+1}) \notag\\
         =\ & F(\bX^k,\bY^k,\bZ^k)-F(\bX^{k+1},\bY^{k+1},\bZ^{k+1}) +2(d(\bY^{k+1},\bZ^{k+1})-d(\bY^k,\bZ^k)) + 2(p(\bZ^k)-p(\bZ^{k+1})) \notag\\
         \ge\ & \frac{2\lambda^{-1}+r-L_{\rho}-L_g}{2}\|\bX^k-\bX^{k+1}\|_F^2 +\frac{(2-\beta)r}{2 \beta}\|\bZ^{k}-\bZ^{k+1}\|_F^{2}
         -\left(L_d-\frac{\varepsilon}{2}\right)\|\bY^k-\bY^{k+1}\|_F^2
         \, + \notag\\
        &  \underbrace{\langle\nabla_{\bY} F(\bX^{k+1}, \bY^{k},\bZ^k)-2\nabla_{\bY} F(\bX(\bY^{k}, \bZ^{k}), \bY^{k}, \bZ^{k}), \bY^{k}-\bY^{k+1}\rangle}_{\text{\ding{172}}}\, +
      \notag \\ 
      & \underbrace{2 r\left\langle \bZ^{k+1}-\bZ^{k},\bX(\bY(\bZ^{k+1}), \bZ^{k})-\bX(\bY^{k+1}, \bZ^{k+1})\right\rangle}_{\text{\ding{173}}}.
        \end{align}
Subsequently, we simplify the terms \text{\ding{172}} and  \text{\ding{173}}. First, for \text{\ding{172}} we know that 
        \begin{align*}
        \text{\ding{172}} 
         =\ & \langle\nabla_{\bY} F(\bX^{k+1}, \bY^{k},\bZ^k), \bY^{k}-\bY^{k+1}\rangle+2\langle\nabla_{\bY} F(\bX(\bY^{k}, \bZ^{k}), \bY^{k}, \bZ^{k}), \bY^{k+1}-\bY^{k}\rangle\\
        =\ & \langle\nabla_{\bY} F(\bX^{k+1}, \bY^{k},\bZ^k), \bY^{k+1}-\bY^{k}\rangle\\
        &+2\langle\nabla_{\bY} F(\bX(\bY^{k}, \bZ^{k}), \bY^{k}, \bZ^{k})-\nabla_{\bY} F(\bX^{k+1}, \bY^{k},\bZ^k), \bY^{k+1}-\bY^{k}\rangle
        \end{align*}
        For the first term, one has that
        \begin{align*}
    &\langle\nabla_{\bY}F(\bX^{k+1}, \bY^{k},\bZ^k), \bY^{k+1}-\bY^{k}\rangle 
     \ge   \frac{1}{\alpha}\|\bY^{k}-\bY^{k+1}\|_F^2,
\end{align*}
where it follows from the update of dual variables. On the other hand, for the second term we have
                \begin{align*}
        %\label{desascent-key3}
         &2\langle\nabla_{\bY} F(\bX(\bY^{k}, \bZ^{k}), \bY^{k}, \bZ^{k})-\nabla_{\bY} F(\bX^{k+1}, \bY^{k},\bZ^k), \bY^{k+1}-\bY^{k}\rangle \notag\\
         \ge\ & -2\|\nabla_{\bY} F(\bX(\bY^{k}, \bZ^{k}),\bY^{k},\bZ^{k})-\nabla_{\bY} F(\bX^{k+1}, \bY^{k},\bZ^{k})\|_F \cdot\|\bY^{k+1}-\bY^{k}\|_F \notag\\
         \ge\ & - 4R^{\operatorname{op}}_{\bX}\|\bX^{k+1}-\bX(\bY^{k}, \bZ^{k})\|_F \cdot\|\bY^{k}-\bY^{k+1}\|_F \notag\\
         \ge\ & -2R^{\operatorname{op}}_{\bX} \zeta^2\|\bY^{k}-\bY^{k+1}\|_F^{2}- 2R^{\operatorname{op}}_{\bX}\|\bX^{k+1}-\bX^{k}\|_F^{2},
        \end{align*}
        where the last inequality follows from $2|x||y| \leq \frac{1}{\zeta^2} x^2 + 
        \zeta^2 y^2
        $ and Lemma \ref{lemma:lip}. Together them, we obtain, 
\begin{equation}
\label{desascent-key3}
    \text{\ding{172}}  \ge \left(\frac{1}{\alpha}- 2R^{\operatorname{op}}_{\bX}\zeta^2\right) \|\bY^{k}-\bY^{k+1}\|_F^{2}
    -2R^{\operatorname{op}}_{\bX}\|\bX^{k+1}-\bX^{k}\|_F^{2}.
\end{equation}
Then, we continue to bound $ \text{\ding{173}}$, 
\begin{equation}
\label{desascent-key2}
\begin{aligned}
  \text{\ding{173}} 
 =\ & 2 r \left\langle \bZ^{k+1}-\bZ^{k}, \bX(\bY(\bZ^{k+1}), \bZ^{k})-\bX(\bY^{k+1}, \bZ^{k+1})\right\rangle \\
 =\ & 2 r \left\langle \bZ^{k+1}-\bZ^{k}, \bX(\bY(\bZ^{k+1}), \bZ^{k})-\bX(\bY(\bZ^{k+1}), \bZ^{k+1})\right\rangle \\
 & + 2 r\left\langle \bZ^{k+1}-\bZ^{k}, \bX(\bY(\bZ^{k+1}), \bZ^{k+1})-\bX(\bY^{k+1}, \bZ^{k+1})\right\rangle \\
\ge\ &  -2r \sigma_1\|\bZ^{k+1}-\bZ^k\|_F^2+2 r\left\langle \bZ^{k+1}-\bZ^{k}, \bX(\bY(\bZ^{k+1}), \bZ^{k+1})-\bX(\bY^{k+1}, \bZ^{k+1})\right\rangle \\
\ge\ & -2r \sigma_1\|\bZ^{k+1}-\bZ^k\|_F^2-\frac{r}{7\beta}\|\bZ^{k+1}-\bZ^k\|_F^2 - 7r\beta \|\bX(\bY(\bZ^{k+1}), \bZ^{k+1})-\bX(\bY^{k+1}, \bZ^{k+1})\|_F^2, 
\end{aligned}
\end{equation}
where the first inequality is from \eqref{lip-z} with the Cauchy-Schwarz inequality and the second inequality follows from the AM-GM inequality. 
Thus, the inequalities \eqref{desascent-key1}-\eqref{desascent-key2} above imply that
\begin{align}\label{desas-keymid}
        & \Phi(\bX^k,\bY^k,\bZ^k)-\Phi(\bX^{k+1},\bY^{k+1},\bZ^{k+1}) \notag\\
        \geq\ & \frac{2\lambda^{-1}+r-L_{\rho} - L_g -  4R^{\operatorname{op}}_{\bX}}{2}\cdot\|\bX^{k}-\bX^{k+1}\|_F^{2}+\left(\frac{1}{\alpha}- 2R^{\operatorname{op}}_{\bX}\zeta^2-L_d+\frac{\varepsilon}{2}\right)\|\bY^{k}-\bY^{k+1}\|_F^{2}\notag\\
        &
        +\left(\frac{(2-\beta)r}{2 \beta}-2r\sigma_1-\frac{r}{7\beta}\right)\|\bZ^{k}-\bZ^{k+1}\|_F^{2}-7r\beta \|\bX(\bY(\bZ^{k+1}),\bZ^{k+1})-\bX(\bY^{k+1}, \bZ^{k+1})\|_F^2. 
\end{align}
 On top of \eqref{eq:pd-impor} that
 \begin{align*}
 % \label{eq:pd-impor}
& \|\bY^{k+1}-\bY_{+}(\bZ^k)\|_F 
\leq  2\alpha R^{\operatorname{op}}_{\bX} \zeta\|\bX^k-\bX^{k+1}\|_F,
\end{align*}
 we have with $\eta:=2\alpha R^{\operatorname{op}}_{\bX} \zeta$ that
 \begin{align}
 \label{desas-final-key1}
 \|\bY^{k+1}-\bY^{k}\|_F^{2} 
 &=\|\bY^{k+1}-\bY_{+}(\bZ^{k})+\bY_{+}(\bZ^{k})-\bY^{k}\|_F^{2}\notag\\
 &\geq\frac{1}{2}\|\bY^{k}-\bY_{+}(\bZ^{k})\|_F^{2} -\|\bY^{k+1}-\bY_{+}(\bZ^{k})\|_F^{2}\notag\\
 & \geq \frac{1}{2}\|\bY^{k}-\bY_{+}(\bZ^{k})\|_F^{2} -\eta^2\|\bX^{k}-\bX^{k+1}\|_F^2.
 \end{align}
 On the other hand, by Lemma \ref{lemma-sollip} and \eqref{eq:pd-impor} we have
\begin{align}\label{desas-final-key2}
&\|\bX(\bY(\bZ^{k+1}),\bZ^{k+1})-\bX(\bY^{k+1}, \bZ^{k+1})\|_F^{2}\notag\\
\leq\ &4\|\bX(\bY(\bZ^{k+1}),\bZ^{k+1})-\bX(\bY(\bZ^{k}),\bZ^{k})\|_F^{2}+4\|\bX(\bY(\bZ^{k}),\bZ^{k})-\bX(\bY_+(\bZ^k), \bZ^{k})\|_F^{2} \notag\\
 &+4\|\bX(\bY_+(\bZ^k), \bZ^{k})-\bX(\bY^{k+1}, \bZ^{k})\|_F^{2}+4\|\bX(\bY^{k+1}, \bZ^{k})-\bX(\bY^{k+1}, \bZ^{k+1})\|_F^{2} \notag\\
\leq\ & 8\sigma_1^2\|\bZ^{k}-\bZ^{k+1}\|_F^2+ 4\|\bX(\bY(\bZ^{k}),\bZ^{k})-\bX(\bY_+(\bZ^k), \bZ^{k})\|_F^{2}+4\sigma_2^2\eta^2\|\bX^k-\bX^{k+1}\|_F^2.
\end{align}
Substituting \eqref{desas-final-key1} and \eqref{desas-final-key2} into \eqref{desas-keymid} yields
\begin{align*}
&\Phi(\bX^k,\bY^k,\bZ^k)-\Phi(\bX^{k+1},\bY^{k+1},\bZ^{k+1})\notag\\
 \geq\ &\left(\frac{2\lambda^{-1}+r-L_{\rho} - L_g - 4R^{\operatorname{op}}_{\bX}}{2}-28r\beta\sigma_2^2\eta^2\right)\|\bX^{k}-\bX^{k+1}\|_F^{2}\ + \left(\frac{1}{\alpha}-2R^{\operatorname{op}}_{\bX}\zeta^2-L_d+\frac{\varepsilon}{2}\right)\cdot\notag\\
 &\left(\frac{1}{2}\|\bY^{k}-\bY^{k}_+(\bZ^{k})\|_F^{2}-\eta^2\|\bX^{k+1}-\bX^k\|_F^2 \right)
+\left(\frac{(2-\beta)r}{2 \beta}-2r\sigma_1-\frac{r}{7\beta}-56r\beta\sigma_1^2\right)\|\bZ^{k}-\bZ^{k+1}\|_F^{2}\\ 
&- 28r\beta\|\bX(\bY(\bZ^{k}),\bZ^{k})-\bX(\bY_+^{k}(\bZ^k), \bZ^{k})\|_F^{2}.
\end{align*}

Suppose that $r\ge L_{\rho} + L_g + 4R^{\operatorname{op}}_{\bX}$, which implies $\sigma_2 \leq 3$ and  $L_{d}-\frac{\varepsilon}{2} \leq  10R^{\operatorname{op}}_{\bX}$, we observe the following:
\begin{itemize}
\item As for $\alpha$, we have $\alpha \leq \min\left\{\frac{1}{ 20R^{\operatorname{op}}_{\bX}}, \frac{1}{8R^{\operatorname{op}}_{\bX}\zeta^2} \right\}$, and then $\frac{1}{ \alpha}-L_{d}+\frac{\varepsilon}{2} \ge \frac{1}{2\alpha}$ and
$
\frac{1}{2\alpha} - 2R^{\operatorname{op}}_{\bX}\zeta^2 \ge \frac{1}{4\alpha}.
$
\item As $\beta\leq\frac{1}{28}$ and $\sigma_1 \leq \frac{3}{2}$ since $r \geq   3(L_{\rho}+L_g)$, 
we have
\begin{align*}
\frac{(2-\beta)r}{2 \beta}-2r\sigma_1-\frac{r}{7\beta}-56r\beta\sigma_1^2  & \ge \frac{6r}{7\beta} -\frac{7r}{2}- 126r\beta = \frac{r}{\beta}\left(\frac{6}{7}-\frac{7}{2}\beta- 126\beta^2\right) \ge \frac{4r}{7\beta}.
\end{align*}
\item As $\lambda^{-1} \ge  2R^{\operatorname{op}}_{\bX}$ and recalling $\eta:= 2\alpha R^{\operatorname{op}}_{\bX} \zeta$, we have
\[\alpha\leq\frac{1}{8R^{\operatorname{op}}_{\bX} \zeta^2} = \frac{\lambda^{-1}}{ 16(R^{\operatorname{op}}_{\bX})^2\zeta^2} \frac{2R^{\operatorname{op}}_{\bX}}{\lambda^{-1}} \leq \frac{\lambda^{-1}}{ 16(R^{\operatorname{op}}_{\bX})^2\zeta^2}\quad \text{and}\quad
\frac{\eta^2}{2\alpha} = \frac{4\alpha (R^{\operatorname{op}}_{\bX})^2\zeta^2}{2} \leq \frac{1 }{8\lambda}.
\]
Moreover, due to $\beta  \leq \frac{1}{14\alpha r\sigma_2^2}$  and $ r\geq L_{\rho}+L_g+4R^{\operatorname{op}}_{\bX} $,
 we can obtain 
 \[
 28 r\beta \sigma_2^2\eta^2 \leq \frac{2 \eta^2}{\alpha} \leq \frac{1}{2\lambda}\quad\text{and}\quad 
  \frac{2\lambda^{-1}+r-L_{\rho}-L_g-4R^{\operatorname{op}}_{\bX}}{2} -28 r\beta \sigma_2^2\eta^2 - \frac{\eta^2}{4\alpha} \ge \frac{7}{16\lambda}.
\]
\end{itemize}
Together all pieces, we get 
\begin{align*}
&\Phi(\bX^k,\bY^k,\bZ^k)-\Phi(\bX^{k+1},\bY^{k+1},\bZ^{k+1}) \\
\ge\ & \frac{7}{16\lambda}\|\bX^{k}-\bX^{k+1}\|_F^{2}+\frac{1}{8\alpha}\|\bY^{k}-\bY_+^{k}(\bZ^k)\|_F^2+\frac{4r}{7\beta}\|\bZ^{k}-\bZ^{k+1}\|_F^{2}\\
&-28r\beta\|\bX(\bY(\bZ^{k}),\bZ^{k})-\bX(\bY_+^{k}(\bZ^k), \bZ^{k})\|_F^{2}.
\end{align*}
The proof is complete.
\end{proof}

\section{Proof Details of Sufficient Descent Property}\label{subsec:suff}

From the basic descent property in Appendix~\ref{sec:basic_decrease}, we observe that the main challenge to derive sufficient descent lies in controlling the term
\[
\|\bX(\bY(\bZ^{k}),\bZ^{k})-\bX(\bY_+^{k}(\bZ^k), \bZ^{k})\|_F.
\]
To address this, we explicitly bound it using the quantity  $\|\bY^{k}-\bY_{+}^{k}(\bZ^{k})\|_F$, which corresponds to a one-step projected gradient update of the dual function as shown in Lemma \ref{prop:dual_eb_KL}. Here we first show the proof of Lemma \ref{prop:dual_eb_KL} and then derive the sufficient descent property.

\paragraph{Proof of Lemma \ref{prop:dual_eb_KL}}

Let $\psi:\R^{m\times n}\times \R^{m\times n}\rightarrow\R$ be the function defined by
\[
\psi(\bX,\bZ) :=  F(\bX,\bY(\bZ),\bZ).
\]
Consider arbitrary $\bX\in\mathcal{X}$, $\bY \in \mathcal{Y}$, and $\bZ \in \mathbb{R}^{m\times n}$. Note that the function $\psi(\cdot,\bZ)$ is $(r-\mu_{\rho})$-strongly convex. Since
\begin{align*}
\mathop{\argmin}_{\bX'\in\mathcal{X}}\psi(\bX',\bZ) & =\mathop{\argmin}_{\bX'\in\mathcal{X}} F(\bX',\bY(\bZ),\bZ) =  \bX(\bY(\bZ),\bZ),
\end{align*}
we see that 
\begin{equation}
\label{eq:dual1}
\psi(\bX,\bZ) - \psi(\bX(\bY(\bZ),\bZ),\bZ) \ge \frac{r-\mu_{\rho}}{2}\|\bX-\bX(\bY(\bZ),\bZ)\|_F^2.
\end{equation}
In addition, we have
\begin{align} 
\label{eq:dual2}
\psi(\bX,\bZ) - \psi(\bX(\bY(\bZ),\bZ),\bZ)\notag
\leq \  &\psi(\bX,\bZ)-F(\bX(\bY_+(\bZ),\bZ),\bY_+(\bZ),\bZ)\notag \\
\leq \ & \max_{\bY^\prime\in \mathcal{Y}} F(\bX,\bY^\prime,\bZ)-F(\bX(\bY_+(\bZ),\bZ),\bY_+(\bZ),\bZ)
\end{align}
where the first inequality follows from
\begin{align*}
&F(\bX(\bY_+(\bZ),\bZ),\bY_+(\bZ),\bZ) \\
  =\ & \min_{\bX'\in\mathcal{X}}\left\{\mathcal{L}_{\rho}(\bX',\bY_+(\bZ))-\frac{\varepsilon}{2}\|\bY_+(\bZ)\|_F^2+\frac{r}{2}\|\bX'-\bZ\|_F^2\right\} \\
\leq\ & \max_{\bY' \in \mathcal{Y}}\min_{\bX'\in\mathcal{X}}\left\{\mathcal{L}_{\rho}(\bX',\bY')-\frac{\varepsilon}{2}\|\bY^\prime\|_F^2 +\frac{r}{2}\|\bX'-\bZ\|_F^2\right\} \\
=\ & \min_{\bX'\in\mathcal{X}}\left\{\mathcal{L}_{\rho}(\bX',\bY(\bZ))-\frac{\varepsilon}{2}\|\bY(\bZ)\|_F^2 +\frac{r}{2}\|\bX'-\bZ\|_F^2\right\}\\
=\ & \min_{\bX'\in\mathcal{X}} \psi(\bX',\bZ)
= \psi(\bX(\bY(\bZ),\bZ),\bZ). % \label{eq:Fr-one}
\end{align*}
As \eqref{eq:dual1} and \eqref{eq:dual2} hold for any $\bX\in\mathcal{X}$, we obtain the following intermediate relation by taking $\bX= \bX(\bY_+(\bZ),\bZ)$
\begin{equation}\label{OS-GS-better-key3-main} 
\begin{aligned}
    & \, \frac{r-\mu_{\rho}}{2}\cdot\|\bX(\bY(\bZ),\bZ)-\bX(\bY_+(\bZ),\bZ)\|_F^2\\
     \leq & \, \max\limits_{\bY^\prime\in \mathcal{Y}} F(\bX(\bY_+(\bZ),\bZ),\bY^\prime,\bZ) - F(\bX(\bY_+(\bZ),\bZ),\bY_+(\bZ),\bZ).
\end{aligned}
\end{equation}
If $F(\bX(\bY_+(\bZ),\bZ),\bY_+(\bZ),\bZ)=\max_{\bY'\in\mathcal{Y}} F(\bX(\bY_+(\bZ),\bZ),\bY',\bZ)$, then the desired inequality follows trivially from~\eqref{OS-GS-better-key3-main}. Otherwise, we have $\bY_+(\bZ) \in \mathcal{Y}\setminus \mathcal{Y}^\star(\bX(\bY_+(\bZ),\bZ))$, where $\mathcal{Y}^{\star}(\bX):=\argmax_{\bY \in \mathcal{Y}} F(\bX,\bY,\bZ)$.

Since $\varepsilon>0$, we know that $F(\bX,\cdot,\bZ)$ is $\varepsilon$-strongly concave, which implies that
\begin{align*}
% \label{y-linear-key2}
\frac{\varepsilon}{2}\cdot\dist^2(\bY, \mathcal{Y}^{\star}(\bX))\leq \max\limits_{\bY'\in \mathcal{Y}} F(\bX,\bY',\bZ) - F(\bX,\bY,\bZ)\quad \text{for all}\ \bX\in \mathcal{X},\ \bY\in\mathcal{Y},\ \bZ\in\mathbb{R}^{m\times n}.
\end{align*}
In view of the equivalence between the quadratic growth condition and the K{\L} property with exponent $\theta=1/2$ for convex functions \citep[Theorem 5]{bolte2017error}, we know that for all  $\bX \in \mathcal{X}$,  $\bY \in \mathcal{Y} \setminus \mathcal{Y}^\star(\bX)$
\begin{equation*}
% \label{eq:weak-sharp}
\dist(\bz, -G(\bX)+\varepsilon\bY+\partial\iota_{\mathcal{Y}}(\bY))\ge \sqrt{2\varepsilon}\left(\max\limits_{\bY'\in \mathcal{Y}} F(\bX,\bY',\bZ) - F(\bX,\bY,\bZ)\right)^{\frac{1}{2}}.
\end{equation*}
Then it follows that
\begin{align*}
% \label{KL-key-01}
& \sqrt{2\varepsilon}\left(\max\limits_{\bY^\prime\in \mathcal{Y}} F(\bX(\bY_+(\bZ),\bZ),\bY^\prime,\bZ) - F(\bX(\bY_+(\bZ),\bZ),\bY_+(\bZ),\bZ)\right)^{\frac{1}{2}}\notag\\
\leq \ & \dist(\bz, -\nabla_{\bY} F(\bX(\bY_+(\bZ),\bZ),\bY_+(\bZ),\bZ)+\partial\iota_{\mathcal{Y}}(\bY_+(\bZ))) \notag\\
\leq \ & \dist(\bz, -\nabla_{\bY} F(\bX(\bY,\bZ),\bY_+(\bZ),\bZ) +\partial\iota_{\mathcal{Y}}(\bY_+(\bZ)))   \notag\\
\ & + \|\nabla_{\bY} F(\bX(\bY,\bZ),\bY_+(\bZ),\bZ)-\nabla_{\bY} F(\bX(\bY_+(\bZ),\bZ),\bY_+(\bZ),\bZ)\|_F\notag\\ 
{\leq} \ & \dist(\bz, -\nabla_{\bY} F(\bX(\bY,\bZ),\bY_+(\bZ),\bZ)+\partial\iota_{\mathcal{Y}}(\bY_+(\bZ)))+ 2\sigma_2R^{\operatorname{op}}_{\bX}\|\bY-\bY_+(\bZ)\|_F\notag\\
{\leq} \ & \left(\frac{1}{\alpha}+\varepsilon+2\sigma_2R^{\operatorname{op}}_{\bX}\right)\|\bY_+(\bZ)-\bY\|_F,
\end{align*}
where the third inequality follows from the $2R^{\operatorname{op}}_{\bX}$-Lipschitz continuity  of $G(\cdot)$ and  \eqref{lip-y};  the last inequality follows from the relative error condition of the projected gradient ascent method. This, together with \eqref{OS-GS-better-key3-main}, yields
\[
\|\bX(\bY(\bZ),\bZ)-\bX(\bY_+(\bZ),\bZ)\|_F\leq \frac{1}{\sqrt{r-\mu_{\rho}}}\left(\frac{1+  (2 \sigma_2 R^{\operatorname{op}}_{\bX}+\varepsilon)\alpha}{\sqrt{\varepsilon}\alpha}\right)
\|\bY-\bY_+(\bZ)\|_F.
\]
The proof is complete.

\paragraph{Proof of Proposition \ref{prop:suff-decrease}}
From Proposition \ref{prop:decrease} and Lemma \ref{prop:dual_eb_KL} we know that 
\begin{align*}
\Phi^k-\Phi^{k+1} 
\ge \ &  \frac{7}{16\lambda}\|\bX^{k}-\bX^{k+1}\|_F^{2}+\left(\frac{1}{8\alpha}-28r\beta\omega^2\right)\|\bY^{k}-\bY_{+}^{k}(\bZ^{k})\|_F^2
 +\frac{4r}{7\beta}\|\bZ^{k}-\bZ^{k+1}\|_F^{2}.
\end{align*}
Since 
$\beta  \leq \frac{ 1}{448  r\omega^2\alpha}$, we have
\begin{align*}
&\Phi^k-\Phi^{k+1} \ge  \frac{7}{16\lambda}\|\bX^{k}-\bX^{k+1}\|_F^{2}+\frac{1}{16\alpha}\|\bY^{k}-\bY_{+}^{k}(\bZ^{k})\|_F^2+\frac{4r}{7\beta}\|\bZ^{k}-\bZ^{k+1}\|_F^{2}.
\end{align*}
The desired result can then be obtained.

\end{document}